\documentclass{article}

\usepackage{graphicx}
\usepackage{amssymb}
\usepackage{amsmath, amsthm}
\usepackage[affil-it]{authblk}
\usepackage{color}

\usepackage{hyperref}

\usepackage{subcaption}

\newtheorem{thm}{Theorem}
\newtheorem{lem}[thm]{Lemma}
\newtheorem{remark}[thm]{Remark}

\usepackage[%
  backend=bibtex,bibencoding=ascii,
  style=numeric-comp,
  giveninits=true, uniquename=init, 
  natbib=true,
  url=true,
  doi=true,
  isbn=false,
  backref=false,
  maxnames=99,
  ]{biblatex}
\begin{document}

\title{Analytical and numerical insights into wildfire dynamics: Exploring the advection--diffusion--reaction model}

\author[1]{Cordula Reisch}
\author[2,$\ddagger$]{Adri\'an Navas-Montilla}
\author[3,4]{Ilhan \"Ozgen-Xian}

\affil[1]{Institute of Partial Differential Equations, Technische
  Universit\"at Braunschweig, Germany}
\affil[2]{Fluid Dynamics Technologies Group, Aragon Institute of Engineering Research (I3A), University of Zaragoza, Spain}
\affil[3]{Theoretical Ecohydrology Lab, Institute of Geoecology, Technische Universit\"at
  Braunschweig, Germany}
\affil[4]{Earth and Environmental Sciences Area, Lawrence Berkeley National Laboratory, California, USA}

\affil[$\ddagger$]{Corresponding author at Aragon Institute of Engineering Research (I3A), University of Zaragoza, C/Mariano Esquillor s/n, 50018 Zaragoza, Spain. E-mail: {\tt anavas@unizar.es}}

\date{February 01, 2024}

\maketitle

\begin{abstract}

Understanding the dynamics of wildfire is crucial for developing management and intervention strategies. Mathematical and computational models can be used to improve our understanding of wildfire processes and dynamics. This paper presents a systematic study of a widely used advection--diffusion--reaction wildfire model with non-linear coupling. The importance of single mechanisms is discovered by analysing hierarchical sub-models. Numerical simulations provide further insight into the dynamics. As a result, the influence of wind and model parameters such as the bulk density or the heating value on the wildfire propagation speed and the remaining biomass after the burn are assessed. Linearisation techniques for a reduced model provide surprisingly good estimates for the propagation speed in the full model. 


\end{abstract}


\section{Introduction}

Wildfires are a crucial part of the Earth system that have shaped ecosystems since the appearance of terrestrial plants \cite{Bowman:2009}.  While many ecosystems depend on a certain wildfire regime to sustain themselves \cite{Harrison:2021}, wildfires pose a significant threat to anthropogenic biomes, destroying property and, in the worst case, leading to casualties.  In addition, large and uncontrolled wildfires are often accompanied by air pollution and significant change of the hydrological function of the impacted area that may harm water security \cite{Maina:2019}.  At the time of writing, a large wildfire in Nova Scotia, Canada, is making headlines as it burned down a record area of about ten million hectares.  A huge wildfire is ravaging the Greek island of Corfu, another wildfire at the coast of Algier claimed 34 lives, and yet another wildfire is approaching the city of Palermo, Italy.  In previous years, Californian \cite{Li:2021, Maina:2019} as well as Australian wildfires \cite{Clarke:2022} wreaked havoc on the regions.  The Mediterranean region that has been historically prone to wildfires, is scorched by early wildfires, see, for example \cite{Touma:2021, Viedma:2018}.  South East Asia is reportedly bracing for increasing wildfires, see for example \cite{Nguyen:2023}.   The African Sahel region is more frequently subject to wildfires and agricultural fires that may in return trigger prolonged droughts \cite{Ichoku:2016}. For a current global review of wildfire dynamics, see the excellent discussion in \cite{Harrison:2021}. 

In general, the anthropogenic climate change continues to create favorable conditions for wildfires and both frequency and intensity of wildfires are expected to increase \cite{Doerr:2016, SenandeRivera:2022, UNEP:2022}.  Previously unaffected regions such as Northern Europe have also become vulnerable to wildfires due to climate change \cite{Carnicer:2022}. As wildfire risk increases, the demand to better understand and predict wildfire dynamics through mathematical models and their numerical solution to support management decisions is increasing.

Wildfire dynamics consist of two main processes: (i) ignition and (ii) propagation \cite{Harrison:2021}.  Neither of these two processes are completely understood \cite{Harrison:2021}.  In this context, mathematical and computational models can help to understand key dynamics of the fire--vegetation relationship \cite{Crompton:2022c}. A comprehensive review of the modelling of wildfire propagation dynamics can be found in \cite{sullivan_wildland_2009-1, sullivan_wildland_2009-2, sullivan_wildland_2009-3}.  In this work, we focus on a specific class of physically-based wildfire propagation models, the advection--diffusion--reaction equation-based (ADR) model.  This ADR model has been initially proposed in \cite{Asensio:2002} for wildfire propagation in an idealised two-dimensional (2D) forest. It was numerically explored in \cite{Burger:2020a, Burger:2020b}.
Mathematical properties---like the existence of travelling wave solutions sensu \cite{hadeler_travelling_1975}---were studied for simplified models in \cite{babak_effect_2009}.

Variations and extensions of this ADR wildfire model exist in the literature. For example, in \cite{baranovskiy_mathematical_2022}, a three-dimensional (3D) form of the ADR wildfire model is presented that allows to distinguish between crown and surface forest fires.  In \cite{Asensio:2005}, the model was extended with convection above the fire by means of coupling the ADR wildfire model with an atmospheric flow model.  Another variation of the ADR wildfire model is discussed in \cite{viegas_two-dimensional_2018}, which explicitly accounts for the underlying chemical reactions of combustion and spread in the framework of an ADR. In \cite{mandel_wildland_2008}, an ensemble Kalman filter technique is used to account for measured temperatures in running ADR-based wildfire simulations.

While the ADR wildfire model is widely used and adapted, the individual model components and their influence on the dynamics have not yet been systematically studied. 
In this paper, we formulate sub-models of the ADR wildfire model, each including combinations of the mechanisms advection, diffusion, reaction and ambient cooling. We study the influence of these hierarchical sub-models on the full ADR wildfire model's emerging non-linear behaviour using both analytical and numerical methods.  The rest of this paper is structured as follows:
In Sec.~\ref{sec:wildfiremodel}, the ADR wildfire model is derived using physics-based arguments and the numerical solver used to solve the resulting system of equations is presented. The sub-models are analysed in Sec.~\ref{sec:submodels}.
The interplay of all mechanisms is studied in Sec.~\ref{sec:full}, first for a one-dimensional domain and then for a two-dimensional more realistic setting. Finally, conclusions are drawn in Sec.~\ref{sec:conclusions}.

\section{The wildfire propagation model}\label{sec:wildfiremodel}

\subsection{Advection--diffusion--reaction equation-based wildfire model}

The model presentation in the literature, for example \cite{Asensio:2002, Burger:2020a, Burger:2020b}, is focused on mathematical aspects. Here, we provide a  physics-based description of the model. The model is characterised by the biomass fraction $Y(\mathbf{x},t)$ and the temperature $T(\mathbf{x},t)$, defined inside a two-dimensional spatial domain $\Omega\subset\mathbb{R}^2$, with $\mathbf{x}=(x,y)\in\Omega$ being the Cartesian coordinates in space and $t > 0$ being time, see Fig.~\ref{fig:domain}. The temperature is a measure for the internal energy and we will derive the model by energy considerations. Bulk parameters are used to characterise the properties of the biomass and air mixture, which are considered as a continuum.

\begin{figure}
	 \centering
		 \includegraphics[width=0.49\textwidth]{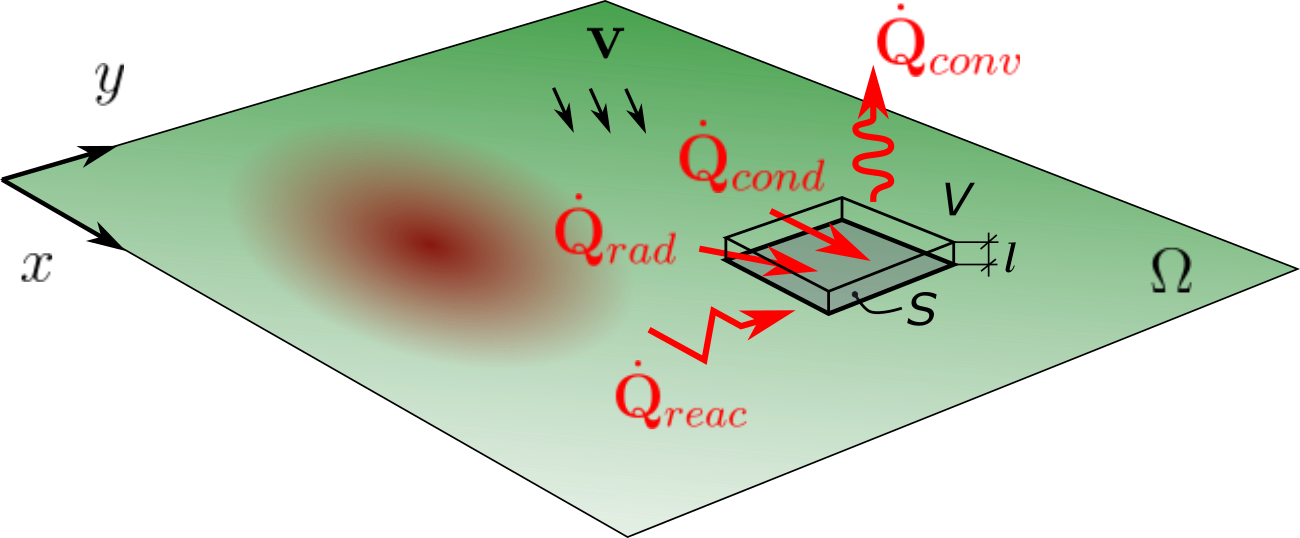}
	 \caption{Representation of the two-dimensional spatial domain $\Omega$ and heat fluxes relevant to wildfire spreading: the conductive flux $\dot{{Q}}_{\mathrm{cond}}$, the convection heat flux $\dot{{Q}}_{\mathrm{conv}}$, radiation heat flux $\dot{{Q}}_{\mathrm{rad}}$, and the reaction heat flux $\dot{{Q}}_{\mathrm{reac}}$. The convective flux is driven by the wind velocity $\mathbf{v}$. All heat fluxes can be defined for an arbitrary control volume $V \subset \Omega$. Colours inside the domain are meant to represent the temperature distribution.}
	 \label{fig:domain}
 \end{figure}

Our starting point is a fixed control volume defined as $V=S\times [0,\ell]$ inside a forest ecosystem. Here, $S \subset \Omega$ is the horizontal extent of the forest and $\ell$ is the forest layer thickness or depth in the vertical direction, associated with the tree height in the forest.  The boundary of the control volume is denoted by $\partial V$. The equation for the conservation of internal energy inside $V$ is
\begin{equation}\label{eq:energy1}
    \frac{d}{d t}\int_{V}\rho e \, \mathrm{d}V + \int_{\partial V}\rho e \mathbf{v}\cdot \mathbf{n} \, \mathrm{d}\Gamma = \dot{{Q}}_\mathrm{cond}- \dot{{Q}}_\mathrm{conv} + \dot{{Q}}_\mathrm{rad} + \dot{{Q}}_\mathrm{reac} \,,
\end{equation}
where $e$ is the specific internal energy per unit mass, $\rho$ is the density defined below, and $\dot{{Q}}_\mathrm{cond}$, $\dot{{Q}}_\mathrm{conv}$, $\dot{{Q}}_\mathrm{rad}$ and $\dot{{Q}}_\mathrm{reac}$ are the conduction, convection, radiation and reaction heat fluxes, respectively. A conceptual sketch of these fluxes is shown in Fig. \ref{fig:domain}.  Note that compressibility effects and viscous dissipation in the gaseous phase are neglected in the energy balance.  The advection velocity, $\mathbf{v}$, is related to the wind velocity. 
We assume that all modelled quantities are homogenised over the forest layer thickness $\ell$ and thus, represent depth-averaged quantities on a two-dimensional horizontal plane defined by $S$. Consequently, heat fluxes and wind velocity only have horizontal components. The fuel mass fraction $Y$ relates to the density $\rho$ as
\begin{equation}\label{eq:massfraction}
Y(\mathbf{x}, t) = \frac{m_f(\mathbf{x},t)}{m_{f_0}(\mathbf{x})}, \text{ with } \rho_{0} (\mathbf{x}) = \frac{m_{f_0}(\mathbf{x})}{V} \text{ and } \rho_f (\mathbf{x},t) = \frac{m_{f}(\mathbf{x},t)}{V},
\end{equation}
where $m_f$ is the fuel mass and $m_{f_0}$ is the maximum fuel mass at $t=0$. Here, we assume that the total mass of air in $V$ is negligible such that the total mass $m_t = m_a+m_f \approx m_f$, and therefore $\rho \approx \rho_{f}$.

Assuming a constant density  $\rho_0$ and specific heat $C$ in $V$, the internal energy can be expressed as $e = CT$. Therefore, the temperature is a quantity for the internal energy. This allows us to rewrite Eq. \eqref{eq:energy1}  as
\begin{equation}\label{eq:energy2}
    \rho_0 C \left(\frac{d}{d t}\int_{V} T \, \mathrm{d}V + \int_{\partial V}T \mathbf{v}\cdot \mathbf{n} \, \mathrm{d}\Gamma \right)= \dot{Q}_\mathrm{cond}- \dot{Q}_\mathrm{conv} + \dot{Q}_\mathrm{rad} + \dot{Q}_\mathrm{reac}.
\end{equation}
The conduction heat flux $\dot{Q}_\mathrm{cond}$ is given by Fourier's law as
\begin{equation}\label{eq:heatflux1}
    \dot{Q}_\mathrm{cond}=
    \int_{\partial V} k\nabla T\cdot \mathbf{n} \, \mathrm{d}\Gamma,
\end{equation}
where $k$ is the thermal conductivity. The convection heat flux is given by Newton's cooling law
\begin{equation}\label{eq:heatflux2}
    \dot{Q}_\mathrm{conv}= \int_{ S} h_0(T-T_{\infty})\, \mathrm{d}\Gamma,
\end{equation}
with $h_0$ being the convection coefficient and $T_{\infty}$ being the temperature of the air over the surface. The radiation heat flux can be written in the form of a non-linear conduction heat flux as
\begin{equation}\label{eq:heatflux3}
    \dot{Q}_\mathrm{rad}=  
    \int_{\partial V} 4\sigma  \epsilon \delta T^3 \nabla T\cdot \mathbf{n} \, \mathrm{d} \Gamma,
\end{equation}
provided that the optical path length for radiation $\delta$ is smaller than the characteristic length of the control volume.  The constant $\sigma$ is the Stefan-Boltzmann constant and $\epsilon$ is an emissivity factor.  Finally, the reaction heat flux is given by
\begin{equation}\label{eq:heatflux41}
    \dot{Q}_\mathrm{reac}= \int_{V} \dot{\rho}_f H \, \mathrm{d}V,
\end{equation}
where $H$ is the combustion heat per unit mass of fuel and
\begin{equation}\label{eq:heatflux42}
    \dot{\rho}_f =\rho_{0} \frac{\partial Y}{\partial  t}
\end{equation}
is the fuel mass disappearance rate per unit volume, with
 \begin{equation}\label{eq:heatflux43}
    \frac{\partial Y}{\partial  t}= -s(T)A\exp{\left(- 
    \frac{T_\mathrm{ac}}{T} 
    \right)}Y
\end{equation}
being the rate of variation of the mass fraction of fuel, given by the Arrhenius law where  
$T_\mathrm{ac}=\frac{E_A}{R}$ 
with the activation energy $E_A$, the universal gas constant $R$, a pre-exponential factor $A$  and the activation function $s(T)$ given by
 \begin{equation}  \label{eq:s}
    s(T)= \left\{ \begin{array}{ccc}
      0 & \hbox{if} & T<T_\mathrm{pc}  \\
      1 & \hbox{if} & T\geq T_\mathrm{pc}
     \end{array}\right.
 \end{equation}
 that triggers the reaction when the temperature is higher than $T_\mathrm{pc}$. By combining the previous definitions and inserting them into Eq. \eqref{eq:heatflux41}, we obtain
\begin{equation}\label{eq:heatflux44}
    \dot{Q}_\mathrm{reac}= \int_{V} s(T)\rho_{0} HA\exp{\left(- 
    \frac{T_\mathrm{ac}}{T} 
    \right)}Y \, \mathrm{d}V.
\end{equation}

Using the Green--Gau{\ss} divergence theorem and considering that the flow is incompressible, so $\nabla\cdot \mathbf{v}=0$, we can rewrite Eq. \eqref{eq:energy2} as
\begin{equation}\label{eq:energy3}
\begin{aligned}
    \rho_0 C \int_{V} \left(\frac{\partial T}{\partial t}  + \mathbf{v}\cdot\nabla T \right)\, \mathrm{d}V  =\\ \int_{V} \nabla\cdot \left( \left( 4\sigma  \epsilon\delta T^3 + k \right) \nabla T \right) \, \mathrm{d}V & -\int_{ V} h(T-T_{\infty})\, \mathrm{d}V + \int_{V} s(T)\rho_{0} HA\exp{\left(- \frac{T_\mathrm{ac}}{T} \right)}Y \, \mathrm{d}V \,,
    \end{aligned}
\end{equation}
with $h=h_0/l$, which allows to obtain
\begin{equation}\label{eq:pde1}
    \rho_0 C  \left(\frac{\partial T}{\partial t}  + \mathbf{v}\cdot\nabla T \right)  =  \nabla\cdot \left( \left( 4\sigma   \epsilon \delta T^3 + k \right) \nabla T \right)     - h(T-T_{\infty}) +  s(T)\rho_{0} HA\exp{\left(- 
\frac{T_\mathrm{ac}}{T} 
\right)}Y.
\end{equation}
Finally, the complete wildfire propagation model is written as
 \begin{equation}  \label{eq:model1}
     \left\{ \begin{aligned}
       \rho_0 C   \left(\frac{\partial T}{\partial t}  + \mathbf{v}\cdot\nabla T \right) & =  \nabla\cdot \left( k_t(T) \nabla T \right)     - h(T-T_{\infty}) +  \Psi(T)\rho_{0} H Y,   \\
 \frac{\partial Y}{\partial  t} &= -\Psi(T)Y,
     \end{aligned}\right.
 \end{equation}
where we define $k_t(T)=4\sigma \epsilon\delta T^3 + k$ and $\Psi(T)=s(T)A\exp{\left(- 
\frac{T_\mathrm{ac}}{T} 
\right)}$. For simplicity, we assume linear diffusion, so $k_t(T)=k$. 
The temperature $T=T(\mathbf{x},t)>0$ and the biomass $Y=Y(\mathbf{x},t)\in[0,1]$ are the model variables. 
The model parameters are summarised in Tab.~\ref{table:general-parameters} with values taken from the literature when available. Parameters that could not be assigned a value range based on the literature are estimated, often through heuristic arguments. In general, many of these parameters may vary in space due to spatial heterogeneity in the environment. Examples for potentially spatially heterogeneous parameters are the bulk density $\rho_0$ and the specific heat $C $. In this work, we will not consider spatial variations of this kind.  To complete the mathematical model, Eq.~\eqref{eq:model1} must be subjected to appropriate initial and boundary conditions, IC and BC, respectively. Many choices exist, for example
  \begin{equation}  \label{eq:cc}
     \left\{ \begin{aligned}
        &\text{IC: } \quad &T(\mathbf{x},0)=T_0(\mathbf{x}),\, Y(\mathbf{x},0)=Y_0(\mathbf{x}), \quad &\mathbf{x}\in \Omega, \\
&\text{BC: } &\left(  k \nabla T    - \rho_0 {C}  \mathbf{v} T \right) \cdot \mathbf{n}=0, \quad &\mathbf{x}\in \partial\Omega,\, t>0.
      \end{aligned}\right.
 \end{equation}
The zero-flux boundary condition defined in Eq.~\eqref{eq:cc} could be applied if the domain is chosen such that the wildfire does not propagate outside of it. 

\begin{table}
\centering
\begin{tabular}{llcccc}
Parameter & Symbol   & Value or range & Units &  Reference & Default (simulations)\\
\hline
Bulk density & $\rho_0$  & [40,640] &\hbox{kg·m$^{-3}$}  & \cite{Asensio:2002}  & 40 \\
Specific heat & {$C$}  & [1.0, 3.0] &\hbox{kJ·kg$^{-1}$·K$^{-1}$}  & \cite{Asensio:2002} & 1.0\\
Pre-exponential factor & $A$  & $0.05$   &\hbox{s$^{-1}$}  & \cite{Grasso:2020} & 0.05\\
Heating value & $H$  & $[100,10000]$  &\hbox{kJ·kg$^{-1}$}  & \cite{Asensio:2002} & 4000 \\
Convection coefficient & $h$  & $[0.1 , 10]$   &\hbox{kW·m$^{-3}$·K$^{-1}$}  & N/A & 4\\
Ambient air temperature & $T_{\infty}$  & $[ 280, 320]$   &\hbox{K}  & N/A & 300\\
Ignition temperature & $T_\mathrm{pc}$  & $[ 400, 500]$   &\hbox{K}  & \cite{Grasso:2020} & 400 \\
Activation temperature & $T_\mathrm{ac}$  & $[ 400, 500]$   &\hbox{K}  & \cite{Grasso:2020} & 400 \\
Thermal diffusivity  & $k$ & $[0, 4]$ & \hbox{kW·m$^{-1}$· K$^{-1}$}  & \cite{Asensio:2002, Burger:2020a} & 2
\end{tabular}
\caption{ Summary of model parameters and their range used in this study for realistic scenarios with references when applicable. Note that parameter values outside of this range may be used for sub-model analysis. }
\label{table:general-parameters}
\end{table}

\subsection{Numerical solver}\label{section:solver}

The ADR wildfire model outlined in Eq.~\eqref{eq:model1} and Eq.~\eqref{eq:cc} is discretised using a high order finite volume scheme. The weighted essentially non-oscillatory (WENO) method \cite{jiang1996} with up to 7-th order of accuracy is used to compute the convective fluxes. The WENO method provides stable reconstructions of cell variables without Gibbs oscillations in the presence of sharp gradients or discontinuities in the solution, which may be the case for the problems herein considered. Second order central differences are used for the diffusive fluxes. This combination proved to be suitable for other applications in \cite{NAVASMONTILLA2023, navas2019depth}.

Different Runge--Kutta integrators of varying order are used in combination with Strang splitting to step forward in time, similar to the numerical solver in \cite{Burger:2020a}. The time integration is as follows: First, the reaction terms in Eq.~\eqref{eq:model1} are evaluated to update the model variables for half a time step using an implicit second order Runge--Kutta (RK2) method. Then, the advection and diffusion terms in the partial differential equation are evaluated to evolve the system for one full time step using a third order Strong Stability Preserving Runge--Kutta (SSP-RK3) method. Finally, the reaction terms are evaluated again for half a time step using the RK2 method.

The verification of the numerical solvers for a spatially integrated model and for the spatially explicit ADR wildfire model are presented in Appendix \ref{sec:validation1} and \ref{sec:validation2}, respectively. The need of very high order of accuracy in space for the discretisation of the convective fluxes is also motivated in  Appendix \ref{sec:validation2}.

\section{Analysis of hierarchical sub-models}\label{sec:submodels}

The ADR wildfire model in Eq.~\eqref{eq:model1} includes the interacting mechanisms advection, diffusion, and reaction that lead to highly non-linear behaviour.  See, for example \cite{perthame_parabolic_2015} for a discussion of the manifold behaviour of reaction--diffusion equations. In order to better understand the effects of these individual mechanisms, we formulate  sub-models of system~\eqref{eq:model1} with fewer mechanisms.  All models together form a model family and every model will provide insights into some characteristic behaviour of the system. In the following, for every sub-model, we summarise relevant results from the literature, prove analytical results, and carry out numerical simulations.  We interpret the analytical and numerical results in the light of wildfire management applications and specify open questions.

\subsection{Dynamical system analysis for the spatially lumped sub-model}\label{sec:ODE}

We start by analysing the model in a spatially integrated or ``lumped'' version, neglecting any spatial mechanisms. The ADR wildfire model in Eq.~\eqref{eq:model1} reduces to an ordinary differential equation model of the form
\begin{equation}\label{eq:modelODE}
\left\{
 \begin{aligned}
       \rho_0 C   \frac{d T}{d t}  & = f_T( T,Y) :=  - h(T-T_{\infty}) +  \Psi(T)\rho_{0} H Y,   \\
 \frac{d Y}{d t}&= f_Y(T,Y) := -\Psi(T)Y,
 \end{aligned}
 \right.
 \end{equation}
with initial conditions $T(0)=T_0 > T_\mathrm{pc}$ and $Y(0)=Y_0>0$.
The combustion function $\Psi(T)$ is given by $\Psi(T)=s(T)A\exp \left (- \frac{T_\mathrm{ac}}{T} \right )$ with the discontinuous activation function $s(T)$ in Eq.~\eqref{eq:s} depending on the ignition temperature $T_\mathrm{pc}$.
The discontinuity of $s(T)$ results in a discontinuity of the differential equation.

\begin{thm}
The initial value problem in Eq.~\eqref{eq:modelODE} has a unique solution for $(T(0), Y(0)) \in [T_\mathrm{pc}, T_\mathrm{max}] \times [0,1]$ and $t\in [0, \tau]$ where $\tau$ is the maximal time with $T(\tau)\geq T_\mathrm{pc}$.
$T_\mathrm{max}$ is a finite maximum value for the temperature.
\end{thm}

\begin{proof}
For $T\geq T_\mathrm{pc}$ the activation function $s(T)$ has the value 1.
The right-hand side of Eq.~\eqref{eq:modelODE} is continuous w.r.t. $T$ and $Y$.
Additionally, the right-hand side is differentiable w.r.t. $T$ and $Y$, and the derivatives are bounded for the given domain.
Consequently, the Picard-Lindel\"{o}f theorem guarantees the existence of a unique solution.

For $T< T_\mathrm{pc}$, the activation function is zero.
For $Y>0$ and $T \in (T-\epsilon, T+ \epsilon)$, $\epsilon >0$, the function $f_T(T,Y)$ in $\frac{\mathrm{d}}{\mathrm{d}t}T=f_T(T,Y)$ is discontinuous.
The existence of a unique solution is therefore not guaranteed and worth a proof.
\end{proof}

\begin{lem}
The discontinuity of the functions $f_T$ and $f_Y$ for $T= T_\mathrm{pc}$ does not affect the validity of the differential equations describing the combustion process.
By interpreting the state $(T_\mathrm{pc}, Y)$ as new initial values for the differential equation
\begin{align}
\begin{aligned}
	\frac{\mathrm{d}T}{\mathrm{d}t} & = -h(T-T_\infty),\\
	\frac{\mathrm{d}Y}{\mathrm{d}t} & = 0,
\end{aligned}
\end{align}
the Picard-Lindel\"{o}f theorem guarantees the existence of a unique solution for the case $T<T_\mathrm{pc}$.
\end{lem}

Consequently, the system given in Eq.~\eqref{eq:modelODE} has a unique solution for given initial values $(T_0, Y_0) \in [0, T_\mathrm{max}] \times [0,1]$.

When modelling the dynamics of a wildfire, the boundedness of the temperature by a lower and an upper bound is desired, where the latter may depend on the initial values.  The system state $Y$ describes a fuel mass density---compare Eq.~\eqref{eq:massfraction}---and should be bounded therefore by $Y\in [0,1]$.  We then prove that the solutions of system \eqref{eq:modelODE} are bounded.

\begin{thm}
The solutions of \eqref{eq:modelODE} are bounded with $(T,Y) \in D= [T_\mathrm{low}, T_\mathrm{max}] \times [0,1]$ for any initial values $(T_0, Y_0) \in D$ and $T_\mathrm{max}$ sufficiently large.
\end{thm}

\begin{proof}
We start with showing the boundedness of $Y$.

Assume $Y(0) \in [0,1]$.
For any $T$
\begin{align}\label{eq:estimate_Psi}
	0 \leq \Psi(T) = s(T) A \exp \left(- \frac{T_\mathrm{ac}}{T} \right ) \leq A
\end{align}
yields.
Consequently,
$$ \frac{\mathrm{d}Y}{\mathrm{d}t}  = - \Psi (T) Y \geq - AY$$
yields as long as $Y$ is positive. Since  $\bar{Y} (t)= Y_0 e^{-At} >0$ gives a lower bound for $Y(T)$, any solution $Y(T)$ for $Y_0 \geq0$ is bounded by $[0, Y_0]$.
Even more, the case $Y(t)=0$ only occurs if $Y_0=0$.
The biomass $Y(t)$ is positive for any meaningful initial condition $Y_0 \in (0, 1]$.

The temperature $T$ is bounded from above because $Y$ is bounded by $1$ and Eq.~\eqref{eq:estimate_Psi} gives
\begin{align}
  \frac{\mathrm{d}T}{\mathrm{d}t} =  - \frac{h}{\rho_0 C}(T-T_{\infty}) +  \frac{1}{C } \Psi(T) H Y
    \leq \frac{AH}{C } -  \frac{h}{\rho_0 C }(T-T_{\infty}).
\end{align}
This ensures a negative derivative for $T > T_\mathrm{max} =\frac{\rho_0}{h} A H + T_\infty$.
Besides, the temperature is bounded from below by $T_\mathrm{low} = \min \{T_0, T_\infty\}$:
If $T_0 < T_\infty < T_\mathrm{pc}$, the temperature increases by
$ \frac{\mathrm{d}T}{\mathrm{d}t} = -h(T- T_\infty)$ until $T= T_\infty$.
On the other hand, if $ T_\infty <T_0 < T_\mathrm{pc}$, the temperature decreases with a lower bound $T_\infty$.
These two cases neglect the activation of the combustion process.
If, in the last case, $T_0 > T_\mathrm{pc}> T_\infty$ yields, the combustion process starts and $T_\infty$ again gives a lower bound for $T$.
\end{proof}

Consequently, the solutions of System~\eqref{eq:modelODE} are bounded, which is crucial for the interpretation of the solutions as a model for real world processes.  Further, the switching behaviour of the combustion function leads to infinitely many stationary points, which differ in the biomass at ambient temperature.

\begin{thm}
    System~\eqref{eq:modelODE} has a continuous line of stationary points $(T_\infty, Y^\star)$ with $Y^\star \in [0,1]$.
\end{thm}
\begin{proof}
    The change of biomass is only zero for either $Y=0$ or $s(T)=0$.
    In the first case, for $Y=0$, the temperature decreases until $T=T_\infty$ and we find the stationary point $(T_\infty,0)^\mathrm{T}$.
    The switching function is zero for $T<T_\mathrm{pc}$.
    Then, the derivative of the temperature is zero if $T=T_\infty$, which gives us infinitely many stationary states $(T_\infty, Y^\star)^\mathrm{T}$ with $Y^\star \in [0,1]$.
\end{proof}

For gaining further understanding of the system dynamics, we visualize in Fig.~\ref{fig:ODE_map} the phase space portrait of System~\eqref{eq:modelODE} using the  default parameters from Tab.~\ref{table:general-parameters}. The heating and cooling regions in the phase space are indicated with red and blue colours. These regions are separated for $T> T_\mathrm{pc}$ by the tipping line with $ \frac{\mathrm{d}T}{\mathrm{d}t}  =0$ given by
\begin{equation}\label{eq:tipping}
         Y_\mathrm{tip}(T)= \frac{h}{\rho_0 H A}  \mathrm{exp}\left({\frac{T_\mathrm{ac}}{T}} \right)  (T - T_\infty).
\end{equation}
The dashed gray tipping line in Fig.~\ref{fig:ODE_map} corresponds to the maximal temperature for any trajectory.  The continuous line of stationary points of the system is plotted with a green line in Fig.~\ref{fig:ODE_map}.

\begin{figure}
	 \centering
		 \includegraphics[width=0.75\textwidth]{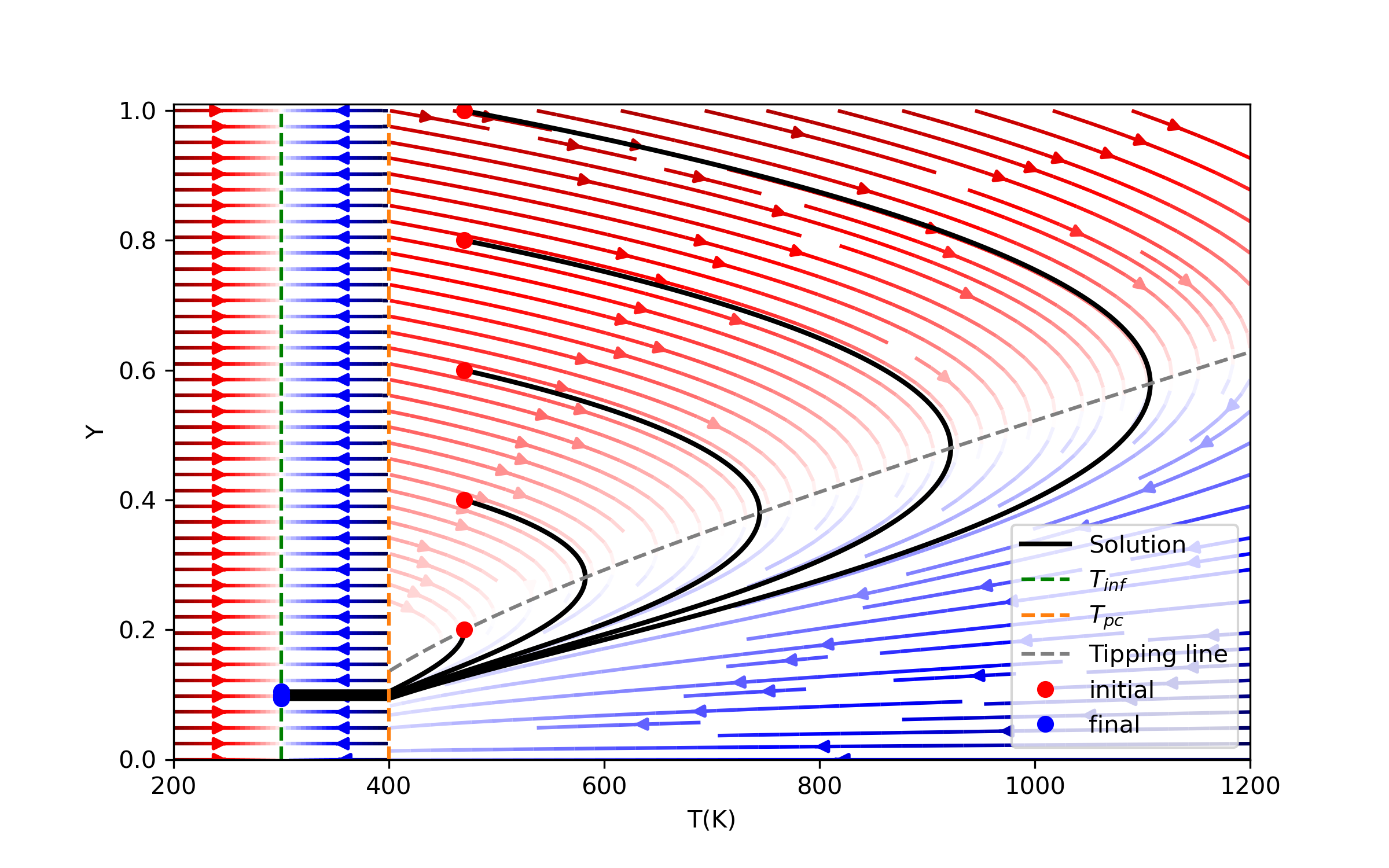}
	 \caption{Phase space portrait of the system, showing the heating (red) and cooling (blue) regions of the trajectories, separated by the dashed gray line.
  Trajectories of the solutions for the initial conditions $T_0=470$~K and $Y_0=\left\{ 0.2, 0.4, 0.6, 0.8, 1.0 \right\}$ are plotted with black continuous lines.}
	 \label{fig:ODE_map}
 \end{figure}

\begin{remark}
In the case where combustion takes place, so for initial values $T_0 > T_\mathrm{pc}$, only stationary points with $Y^\star < Y_\mathrm{tip}(T_\infty)$ are stable and occur as results of the systems dynamics.
\end{remark}
We conclude the analysis of the spatially lumped model: Even if the equations contain noncontinuous terms, the existence of unique and bounded solutions is guaranteed. The set of stationary points is a continuous line with varying values for the biomass. Based on these analytical results, we now investigate the system numerically. 

System~\eqref{eq:modelODE} is solved for the initial conditions $T_0=470$ K and $Y_0=\left\{ 0.2, 0.4, 0.6, 0.8, 1.0 \right\}$ using the numerical solver described in Sec.~\ref{section:solver}. The evolution in time of the temperature and biomass is depicted in Fig.~\ref{fig:ODE_example11}.  The solutions show the switching behaviour of the combustion function for $T = T_\mathrm{pc} = 400$ K resulting in an unsteady slope of the temperature decrease.  The trajectories of these numerical solutions are also depicted in the phase space portrait in Fig.~\ref{fig:ODE_map} using black continuous lines.  Starting with the same initial temperature and various initial biomass values, the trajectories lead to comparable final values for the biomass, and the ambient temperature. The influence of the initial biomass on the remaining biomass after the combustion process is therefore small. This is more clearly visualised in Fig.~\ref{fig:Terminal_map}, where the terminal biomass $Y^{\star}$ is plotted for various combinations of initial temperature $T_0$ and initial biomass values $Y_0$. Fig.~\ref{fig:Terminal_map} shows that a majority of the initial value combinations result in similar terminal biomass at the end of the simulation.

 \begin{figure}
	 \centering
		 \includegraphics[width=0.99\textwidth]{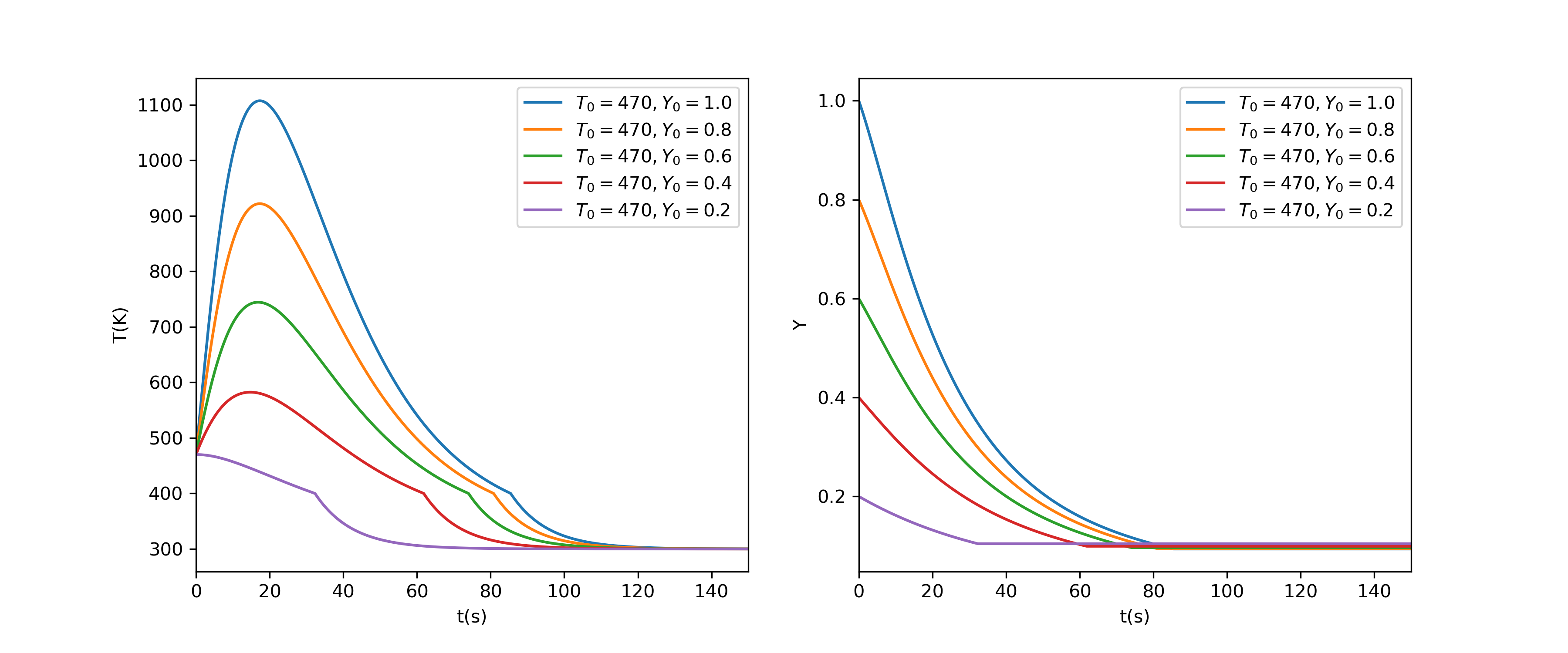}
	 \caption{Evolution in time of the temperature (left) and biomass concentration (right) for the initial conditions $T_0=470$ K and $Y_0=\left\{ 0.2, 0.4, 0.6, 0.8, 1.0 \right\}$. }
	 \label{fig:ODE_example11}
 \end{figure}

  \begin{figure}
	 \centering
		 \includegraphics[width=0.75\textwidth]{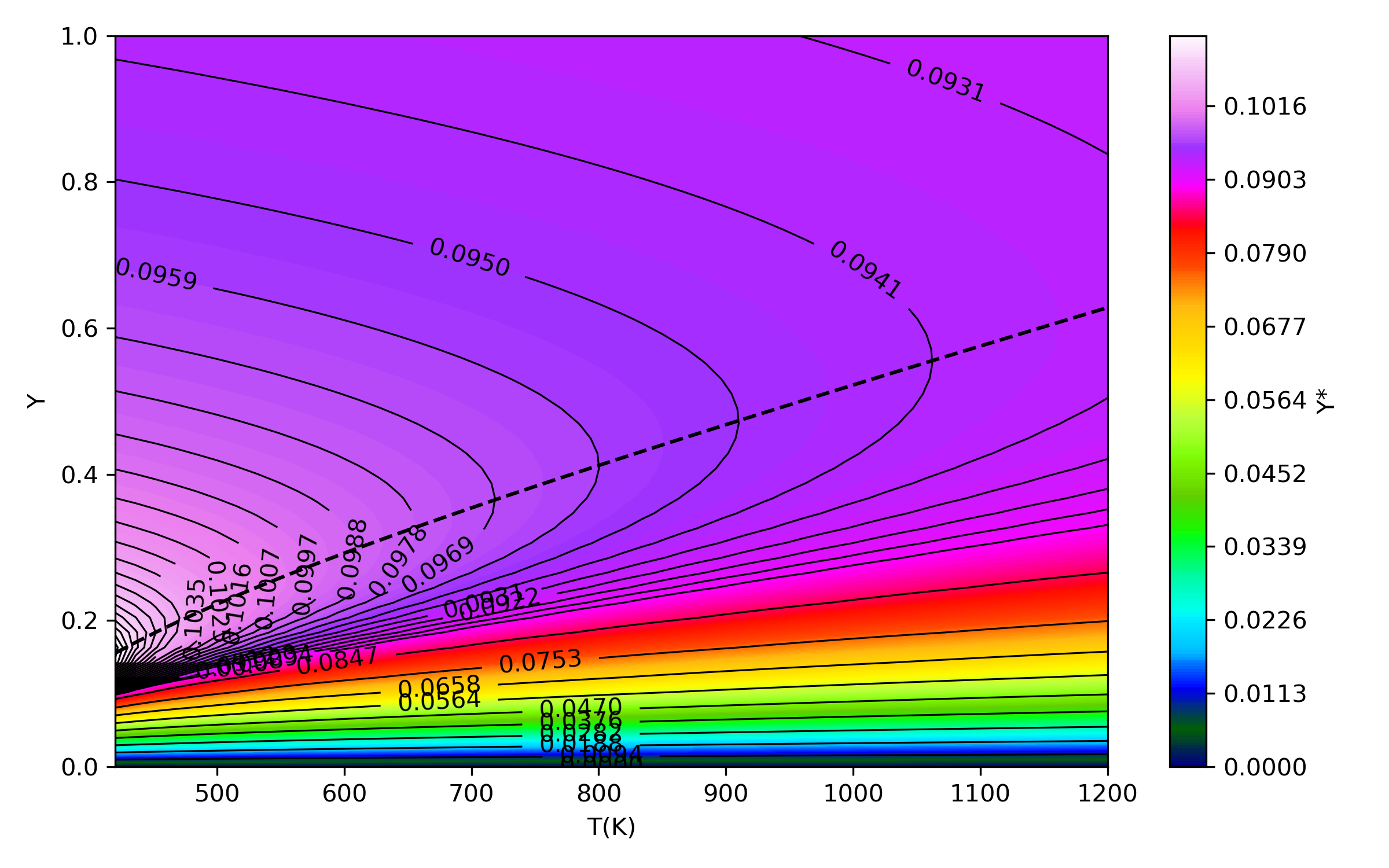}
	 \caption{Terminal biomass $Y^{\star}$ as a function of the initial temperature and biomass. }
	 \label{fig:Terminal_map}
 \end{figure}

Fig. \ref{fig:sensitivity_tipping} shows the tipping line in Eq. \eqref{eq:tipping} for different values of the relevant parameters. We consider as the base parameterisation the values in Tab.~\ref{table:general-parameters}. In the plot, $m$ is a factor scaling
the denominator of $\frac{h}{\rho_0 H A}$, where $\rho_0$ is the bulk density. 
A larger parameter $m$ leads to a smaller upper bound for the remaining biomass.  This can be interpreted as a larger bulk density leading to a higher temperature in the burning process and therefore to a lower biomass in the end.  A larger activation temperature $T_\mathrm{ac}$ leads to a larger upper bound for the remaining biomass. This can be interpreted as a reduction of the availability starting a combustion process, perhaps linked to forest management decisions.

  \begin{figure}
	 \centering
	\includegraphics[width=0.99\textwidth]{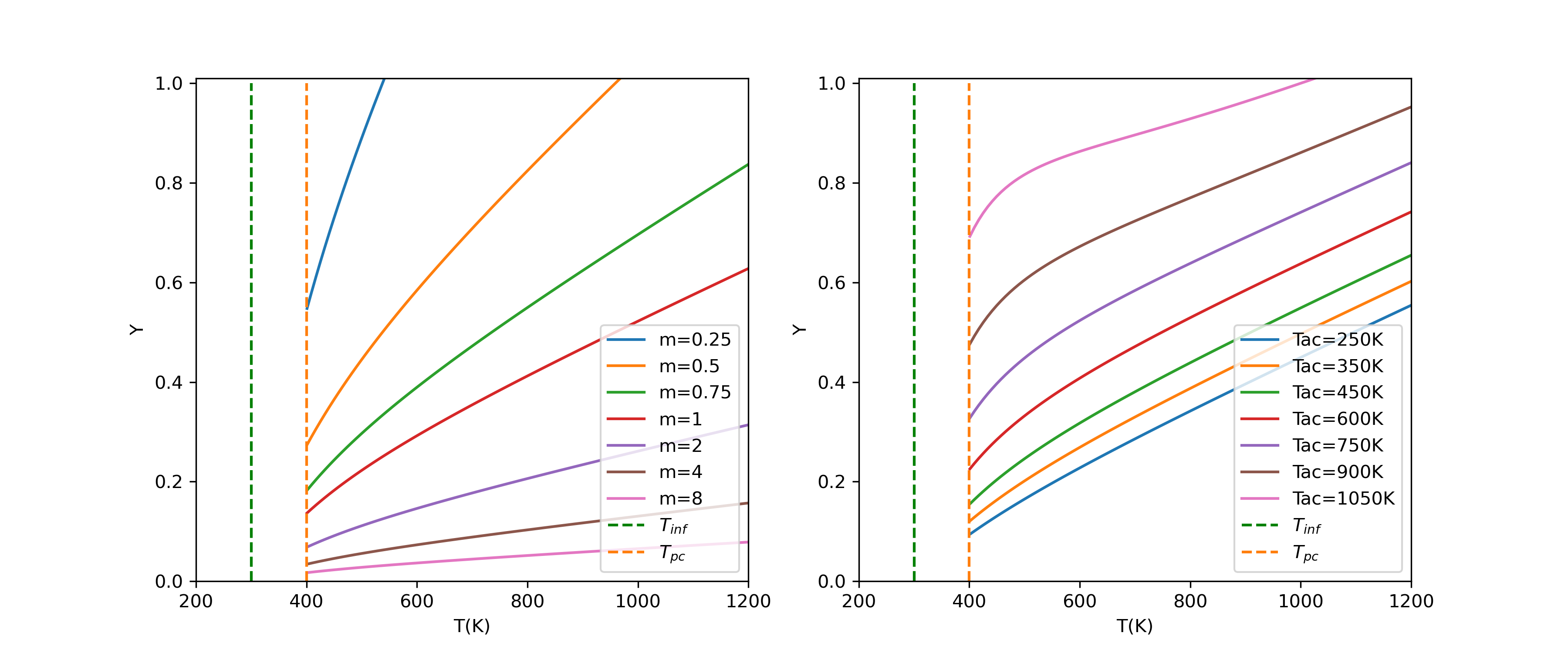}
	 \caption{Sensitivity analysis of the tipping line in Eq.~\eqref{eq:tipping}. Left: Change of the  factor $m$ in $ m\rho_0 H A$. Right: Change of the activation temperature $T_\mathrm{ac}$.  Note that the range of variation is larger than the realistic parameter ranges in Table~\ref{table:general-parameters} for highlighting the changes. }
	 \label{fig:sensitivity_tipping}
 \end{figure}

These effects on the tipping line are as well visible if we calculate the sensitivity of the tipping line on the parameters directly.
The sensitivity of $Y_\mathrm{tip}$ on $\rho_0$ is given by
\begin{equation*}
    \frac{\partial Y_\mathrm{tip}}{\partial \rho_0}= - \frac{h}{\rho_0^2 H A }  \mathrm{exp}\left(\frac{T_\mathrm{ac}}{T_\infty}\right) (T-T_\infty).
\end{equation*}
Enlarging the parameter $\rho_0$ therefore has a negative effect on the tipping line, shifting it downwards. The same is observed for the heating value $H$.  In comparison, the influence of the activation temperature is positive, due to
\begin{equation*}
    \frac{\partial Y_\mathrm{tip}}{\partial T_\mathrm{ac}}=  \frac{h}{\rho_0 H A } \frac{1}{T_\infty}  \mathrm{exp}\left(\frac{T_\mathrm{ac}}{T_\infty}\right)  (T-T_\infty).
\end{equation*}
These expressions of the sensitivity therefore support the findings in Fig.~\ref{fig:sensitivity_tipping}.

The tipping line in the phase space of System~\eqref{eq:modelODE} gives an estimation for the maximum of the remaining biomass $Y^\star$ after the combustion process.  The minimum of $Y^\star$ is an interesting quantity as well.  As the system is non-linear and the dynamics are cut off by the switching function for $T=T_\mathrm{pc}$, a prediction of the steady states depending on the parameters and the initial values is challenging, even if System~\eqref{eq:modelODE} is linearised.  For example, a linearisation of the non-linear combustion term is
\begin{equation}
\exp (-T_\mathrm{ac}/ T) \geq - \delta T_\mathrm{pc} + \delta T,
\end{equation}
where $\delta$ is chosen such that the linearised function is larger than 1 for $T>T_\mathrm{max}$.  With this linearisation, the system reads
\begin{equation}
\begin{aligned}
   \frac{\mathrm{d}T}{\mathrm{d}t} &= - \frac{h}{ \rho_0 C} (T-T_\infty) + \frac{HA}{C} (- \delta T_\mathrm{pc} + \delta T) Y, \\
       \frac{\mathrm{d}Y}{\mathrm{d}t}  &= - A (- \delta T_\mathrm{pc} + \delta T) Y,
\end{aligned}
\end{equation}
and can be reformulated to
\begin{equation}
    \frac{\mathrm{d}T}{\mathrm{d}Y}  = \frac{h}{\rho_0 C A \delta} \frac{T-T_\infty}{(T-T_\mathrm{pc})Y} - \frac{H}{C} .
\end{equation}
The solution of the homogeneous equation is
\begin{equation}
    Y(T) = \frac{\rho_0  C A \delta}{h} (T_\infty -T_\mathrm{pc} ) \ln (T- T_\infty) + \frac{\rho_0 C A \delta}{h} T +c_\mathrm{int},
\end{equation}
where the integration constant $c_\mathrm{int}$ is a very large value, shifting the whole approximation above the domain of~$Y$.   This approach therefore is not giving useful bounds for the inhomogeneous and non-linear system.

The analysis of the lumped System~\eqref{eq:modelODE} shows the influence of some parameters on the tipping line and gives insight into the boundedness of solutions and the sensitivity of the model parameters.  However, the minimum of $Y^\star$ is further affected by diffusion and advection.  In what follows, we analyse the effect of these mechanisms on the solution through sub-models of the ADR wildfire model.  We start with combustion-free sub-models.

\subsection{Combustion-free sub-models}\label{sec:combustionfree}

Combustion-free sub-models of the ADR wildfire model neglect the combustion function $\Psi$ in Eq.~\eqref{eq:model1} to remove the discontinuity in the model.  This allows to investigate paths to stationary states that result primarily from energy dissipation and the development of advection-driven travelling wave solutions.  Here, we study two combustion-free sub-models: (i) a  reaction--diffusion sub-model without combustion in Sec.~\ref{sec:energydissipation} and (ii) a combustion-free ADR sub-model in Sec.~\ref{sec:cfADR}.

 \subsubsection{Energy dissipation in a combustion-free reaction--diffusion sub-model}
 \label{sec:energydissipation}

 First, we consider a combustion-free reaction--diffusion model for $\mathbf{x} \in \Omega \subset \mathbb{R}^d$ and constant diffusion. The reaction only describes the exchange with the ambient temperature and the word reaction is meant only in the context of reaction--diffusion equations. Without any combustion, the biomass is unchanged and $\frac{\mathrm{d}}{\mathrm{d}t}Y=0$.
 Further, we neglect advection, so $\mathbf{v}= \mathbf{0}$.
 The system in Eq.~\eqref{eq:model1} becomes
   \begin{equation} \label{eq:model_pure_rd}
     \left\{ \begin{aligned}
       \frac{\partial T}{\partial t}    =  \alpha \nabla \cdot \nabla T 
       -  \beta (T-T_{\infty})  
     \end{aligned}\right. \end{equation}
  with $\alpha=k/(\rho_0 {C})$ the thermal diffusivity and $\beta=h/(\rho_0 {C})$.
  We adapt the zero-flux boundary conditions from Eq.~\eqref{eq:cc} and use initial conditions with $T\geq T_\infty$ point-wise and a constant $Y(0,x)= Y_0 \in (0,1]$.

  For a given domain $\Omega \subset \mathbb{R}^d$, we define the eigenvalues $\lambda_k$ of the negative Laplacian with periodic boundary conditions as solutions $U_k(\mathbf{x})$ of
  \begin{equation}\label{eq:ev_Laplace}
      \begin{aligned}
          -\nabla \cdot \nabla U_k(\mathbf{x}) & = \lambda_k U_k(\mathbf{x}) && \text{ for } \mathbf{x} \in \Omega, \\
      \end{aligned}
  \end{equation}
All eigenvalues of Eq.~\eqref{eq:ev_Laplace} are real and non-negative.  Then, the solution of Eq.~\eqref{eq:model_pure_rd} can be calculated by hands of a Fourier approach and separation into a homogeneous and inhomogeneous problem.
With coefficients $c_k \in \mathbb{R}$, the solution reads
\begin{equation}\label{eq:sol_pure_rd}
    T(\mathbf{x},t) = T_\infty + \sum_{k=0}^\infty c_k \exp \left( - \frac{\alpha t} {\lambda_k + \beta} \right) U_k(\mathbf{x}).
\end{equation}
The coefficients can be determined by using the Fourier series of the initial conditions. Analogous formulations are possible for different boundary conditions as well.  The solution in Eq.~\eqref{eq:sol_pure_rd} is decaying in time towards the ambient temperature $T_\infty$.

\begin{remark}
    The combustion-free reaction--diffusion problem excluding advection  shows a decaying and levelling solution behaviour.
\end{remark}

This levelling effect gives the dissipation of the system, compare Eq.~\eqref{eq:heatflux2}.
The convection heat flux $\dot{Q}_\mathrm{conv}$ can be integrated over space, resulting in an expression for the energy dissipation rate as
\begin{equation}\label{eq:definition_energy}
    E(t) = \int_\Omega h (T-T_\infty ) \, \mathrm{d}V .
\end{equation}
The change of $E$ is then given by
\begin{equation}
\begin{aligned}
        \frac{\mathrm{d}}{\mathrm{d}t} E(t) &= \int_\Omega h \frac{\mathrm{d}}{\mathrm{d}t}V  = \int_\Omega h (\alpha \nabla \cdot \nabla T  - \beta(T-T_{\infty})) \, \mathrm{d}V
     = h \alpha \int_{\partial \Omega} \nabla T \cdot \mathbf{n} \, \mathrm{d}V  -h \beta \int_\Omega T-T_\infty \, \mathrm{d}V  \\
    & = -\beta E(t),
    \end{aligned}
\end{equation}
where we use the Green--Gau{\ss} divergence theorem and the zero-flux boundary conditions.  The differential equation in $E$ has the solution
\begin{equation}\label{eq:energy_dissipation}
    E(t) = E(0) \exp(- \beta t),
\end{equation}
clearly showing the dissipative nature of the spreading and levelling. The energy dissipation rate tends to zero as the system's temperature approaches the ambient temperature.

We explore the energy dissipation now numerically.
Fig.~\ref{fig:Energy_decay_sec_321} shows the solution of Eq.~\eqref{eq:model_pure_rd} for the parameters given in Tab.~\ref{table:general-parameters}. We impose periodic boundary conditions to ensure that the total energy is not affected by the fluxes across the domain boundaries.
 The periodic boundary conditions are reasonable in this example because there is no combustion process and the fuel is not burning. The initial condition is given by
\begin{equation}\label{eq:ic1}
   T(x,y,0) = 
		\left\{
		\begin{array}{lll}
      	400 \mbox{ K}  &\mbox{if} &  r(x,y)< 50 \mbox{ m} \\
        300 \mbox{ K} & \multicolumn{2}{l}{\mbox{otherwise}}  \\  
		\end{array}
		\right. 
\end{equation}
with 
\begin{equation}
r(x,y)=\sqrt{(x-x_1)^2+(y-y_1)^2}
\end{equation}
with $(x_1,y_1)=(500,500)$ m.
Fig.~\ref{fig:Energy_decay_sec_321}(left) shows the decrease of the energy dissipation rate, $E$, following Eq.~\eqref{eq:energy_dissipation} and for comparison and verification of the numerical solver, the value of $E$ computed by Eq.~\eqref{eq:definition_energy} for the numerical solution. Simulation snapshots at different time steps are shown in Fig.~\ref{fig:Energy_decay_sec_321}(right). The maximum temperature at every time step is denoted in each of the snapshots.  As time evolves, the heat diffuses from an initial maximum temperature in the centre of the domain.  The maximum temperature reduces from $T=400$ K to almost ambient temperature $T=300.8$ K after $t=1000$ s. Energy dissipation is maximum in the beginning and tends to zero when the solution evolves in time, due to the levelling effect.  The energy dissipation rate for the numerical solution fits the analytical estimation, which is a verification for the numerical solver.

 \begin{figure}
	 \centering
		 \includegraphics[width=0.8\textwidth]{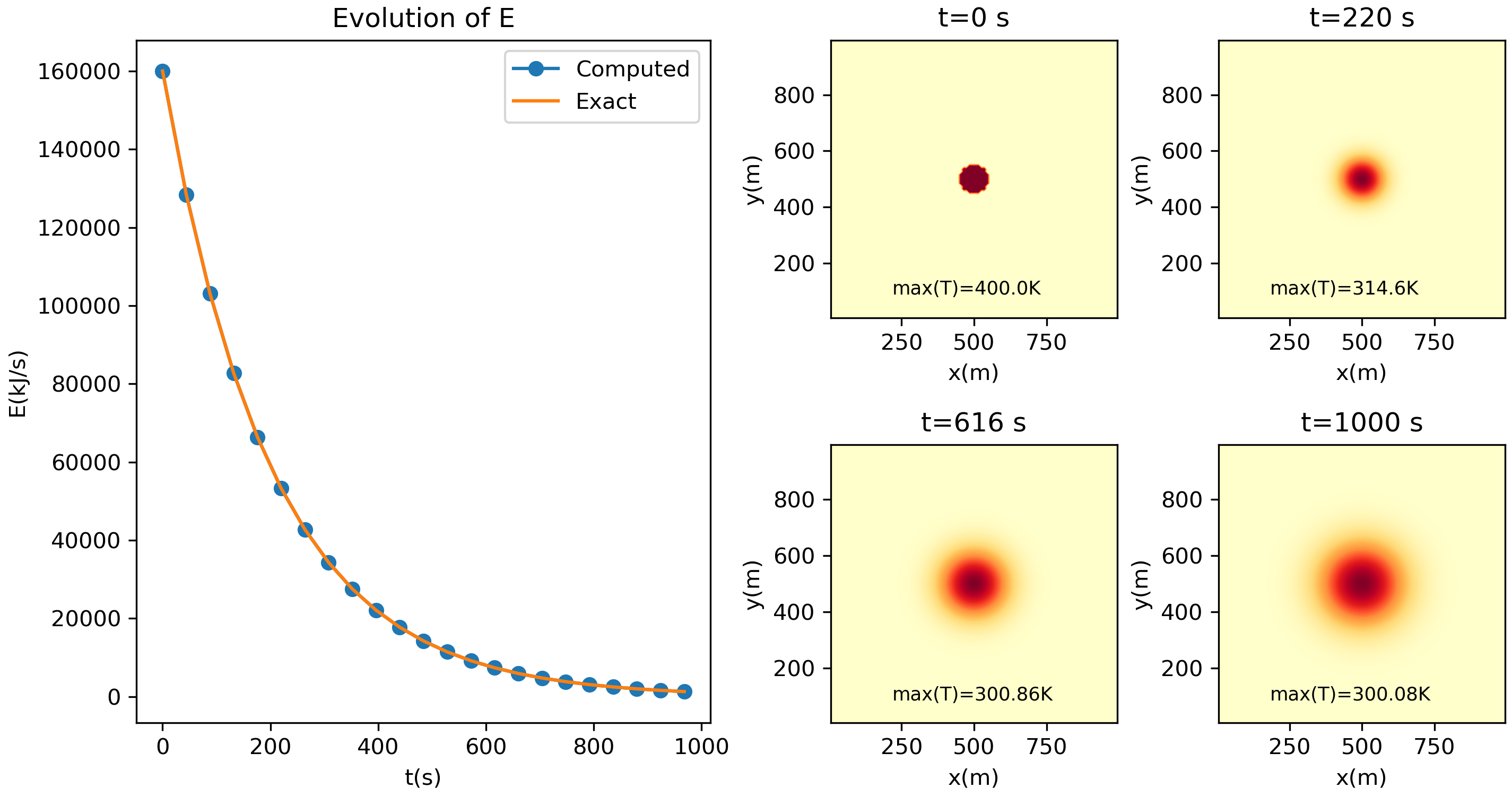}
	 \caption{Numerical example with the parameters in Tab.~\ref{table:general-parameters} for the combustion-free reaction--diffusion model in Eq.~\eqref{eq:model_pure_rd} with periodic boundary conditions. Left: Evolution of the total energy dissipation rate in the system calculated analytically (line) and numerically (circles). Right: Simulation snapshots of the temperature distribution in the domain at different time instants.  Red color indicates high and yellow color indicates low temperature.}
	 \label{fig:Energy_decay_sec_321}
 \end{figure}

\subsubsection{Combustion-free advection--diffusion--reaction sub-model}\label{sec:cfADR}

After regarding the energy dissipation of the advection-free model, we focus on the wind-driven propagation of wildfire.
In this sub-model, in addition to the energy dissipation and diffusion mechanisms discussed in Sec.~\ref{sec:energydissipation}, temperature is also advected. Combustion is still not regarded, so $\frac{\mathrm{d}}{\mathrm{d}t}Y=0$. 
We consider the partial differential equation for the temperature as
   \begin{equation} \label{eq:model_rd_wind}
       {\rho_0} {C} \left ( \frac{\partial T}{\partial t} + \mathbf{v} \cdot \nabla T \right )   =  \nabla \cdot ({ k} \nabla T ) - h (T-T_\infty),
 \end{equation}
 with periodic boundary conditions.
Assuming a spatially constant wind velocity $\mathbf{v}$, the equation reads
    \begin{equation}
        {\rho_0}   {C} \frac{\partial T}{\partial t}    =  \nabla \cdot \left ( { k} \nabla T -  {\rho_0}   {C} T \mathbf{v} \right) - h (T-T_\infty)  ,
 \end{equation}
 and the solution can be calculated according to Sec.~\ref{sec:energydissipation} by substituting the space variable with $\mathbf{x} - \mathbf{v}$.
The energy dissipation in Eq.~\eqref{eq:energy_dissipation} remains identical because the wind only shifts the whole solution in space and the periodic boundary conditions ensure that the total energy is not affected by the fluxes across the domain boundaries.

 \begin{figure}
	 \centering
		 \includegraphics[width=0.8\textwidth]{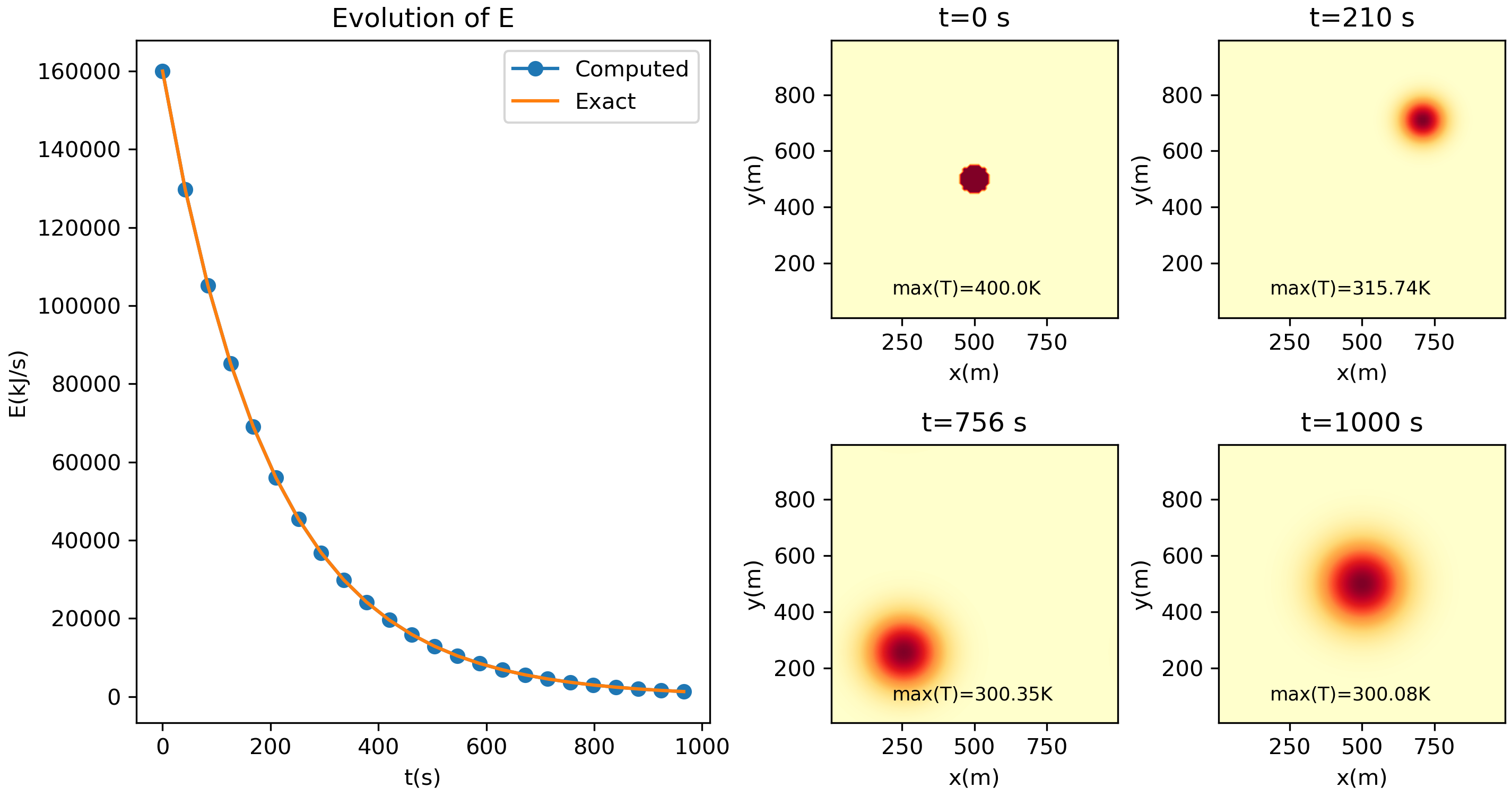}
	         \caption{Numerical example with the parameters in Tab.~\ref{table:general-parameters} for the combustion-free advection--diffusion--reaction model in Eq.~\eqref{eq:model_rd_wind} with wind velocity in direction $(1,1)^\mathrm{T}$ and periodic boundary conditions. Left: Evolution of the total energy dissipation rate in the system calculated analytically (line) and numerically (circles). Right: Simulation snapshots of the temperature distribution in the domain at different time instants.  Red colour indicates high and yellow colour indicates low temperature.}
	 \label{fig:Energy_decay_sec_322}
 \end{figure}

Fig.~\ref{fig:Energy_decay_sec_322}(left) shows the evolution of total energy inside the domain based on the combustion-free ADR sub-model in Eq.~\eqref{eq:model_rd_wind} with periodic boundary conditions and a wind velocity that is pointing towards the direction $(1,1)^\mathrm{T}$.  Comparison of the analytically obtained energy dissipation rate (line) and the numerical solution (circles) shows that the numerical solution preserves the energy as expected.  Fig.~\ref{fig:Energy_decay_sec_322}(right) shows simulation snapshots at different time steps.  The wind advects the high temperature from the centre of the domain at $t=0$ to the upper right-hand corner at $t= 210$ s towards the lower left-hand corner at $t=756$ s due to the periodic boundary conditions.  Along the way, the temperature is subject to diffusion such that the maximum temperature decreases as the simulation evolves. 
As in Fig.~\ref{fig:Energy_decay_sec_321}, the maximum temperature decreases until it nearly reaches the ambient temperature after $t=1000$ s.

The combustion-free ADR model in Eq.~\eqref{eq:model_rd_wind} does not have travelling wave solutions because the reaction function has only one stationary state for $T=T_\infty$.  The energy is not conserved which is shown by the non-zero energy dissipation rate in Fig.~\ref{fig:Energy_decay_sec_322}(left).

\subsection{Advection-free sub-models}\label{sec:advectionfree}

As a next sub-model, we investigate the model in Eq.~\eqref{eq:model1} without wind by setting  $\mathbf{v}=\mathbf{0}$.
This results in a reaction--diffusion equation coupled to an ordinary differential equation as
 \begin{align}  \label{eq:modelAdf}
 \begin{aligned}
        \rho_0  {C}  \left(\frac{\partial T}{\partial t}   \right)  &=  \nabla\cdot \left( k_t(T) \nabla T \right)     - h(T-T_{\infty}) +  \Psi(T)\rho_{0} H Y,   \\
 \frac{\partial Y}{\partial  t}&= -\Psi(T)Y. 
 \end{aligned}
 \end{align}
 We impose periodic boundary conditions and set the diffusion to a constant value $k_t(T)=k$.

The solution of Eq.~\eqref{eq:modelAdf} is either dominated by diffusion or has the form of travelling waves.  Both solution types are illustrated in Fig.~\ref{fig:diff_travel} that shows them for some time steps.  The diffusion dominated solution (green lines) occurs for parameter choices with comparable small values of $H$ and $h$ that diminish the effect of the reaction terms.  The temperature profile diffuses as time progresses.  The travelling wave solution (red lines) shows a distinct temperature profile with a steep gradient emerging and travelling through the domain as the simulation evolves.

\begin{figure}
	 \centering
    \includegraphics[width=0.45\textwidth]{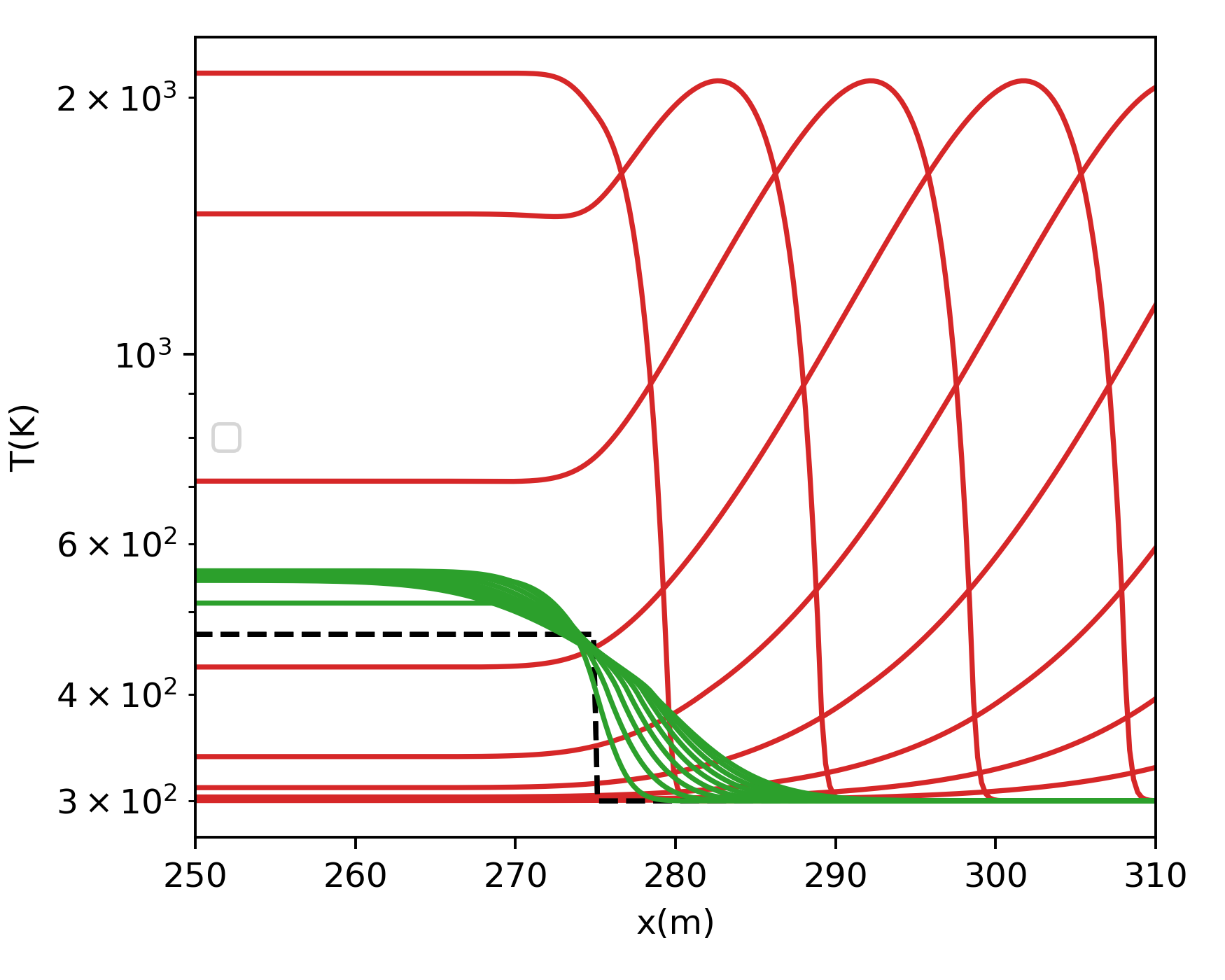}
	 \caption{ Comparison of numerically obtained solutions dominated by diffusion (green) and travelling waves (red).  Each line represents the solution at a different time instants. The initial condition is represented as black dashed line. The travelling wave solution has been computed using the default parameters in Tab. \ref{table:general-parameters}, whereas the diffusion-dominated solution has been computed setting $H=100$ kJ/kg and $h=0.01$ {kWm$^{-3}$ K$^{-1}$}.  The solutions are plotted every 50 s.
        }
	 \label{fig:diff_travel}
\end{figure}

\subsubsection{Travelling wave speed}\label{sec:travellingwaves}

Analytical results proving upper and lower bounds for the travelling wave speed in the full ADR wildfire model are still an open issue.  Here, we investigate the emergence and propagation of travelling waves through numerical simulations and give analytical estimates for the propagation speed.  We choose constant initial conditions on a bounded subdomain with the parameters from Tab.~\ref{table:general-parameters} and then study variations of the parameters $h$ and $k$ in Tab.~\ref{table:subcases-h} and Tab.~\ref{table:subcases-k}. In particular, we consider the computational domain [0,500] m and we set as initial condition

\begin{equation}\label{eq:ic}
   T(x,0) = 
		\left\{
		\begin{array}{lll}
      	470 \mbox{ K}  &\mbox{if} &  225< x <275,  \\
        300 \mbox{ K} & \multicolumn{2}{l}{\mbox{otherwise}}  \\  
		\end{array}
		\right. \text{and} \qquad Y(x,0)=1.0.
\end{equation}

\begin{table}
\centering
\scalebox{0.99}{
\begin{tabular}{c|ccccccc}
Sub-case & A-1 & A-2 & A-3 &A-4 &A-5 &A-6 & A-7  \\
\hline
$h$ (kW·m$^{-3}$·K$^{-1}$)  & 0.0 & 0.025 & 0.25 & 1.0 & 2.0 & 4.0 & 6.0   \\
\end{tabular}}
\caption{Cases computed for different parameters $h$.}
\label{table:subcases-h}
\end{table}

 \begin{table}
\centering
\scalebox{0.99}{
\begin{tabular}{c|cccccc}
Sub-case & B-1 & B-2 & B-3 &B-4 &B-5 &B-6   \\
\hline
$k$ (kW·m$^{-1}$·K$^{-1}$)  & 0.125 & 0.25 & 0.5 & 1.0 & 2.0 & 4.0   \\
\end{tabular}}
\caption{Cases computed for different constant diffusion coefficients $k_t=k$.}
\label{table:subcases-k}
\end{table}

Fig. \ref{fig:front_propagation_all} shows travelling wave solutions that result from the simulation of case A-6 in Tab.~\ref{table:subcases-h} and the parameters from Tab.~\ref{table:general-parameters}.  The initial condition in Eq.~\eqref{eq:ic} evolves to a fire front propagating to the right-hand side and a fire front propagating to the left-hand side.  The amplitude of the travelling wave solution is constant and equal for the spread in both directions.  The wave profile of the temperature is steep on the propagation front and less steep where the temperature decays.  The biomass burns down to a small amount of remaining biomass and the value is different in the support of the initial condition.   Our results are consistent with the findings reported in the literature, for example, \cite{mandel_wildland_2008}.

\begin{figure}
	 \centering
    \includegraphics[width=0.95\textwidth]{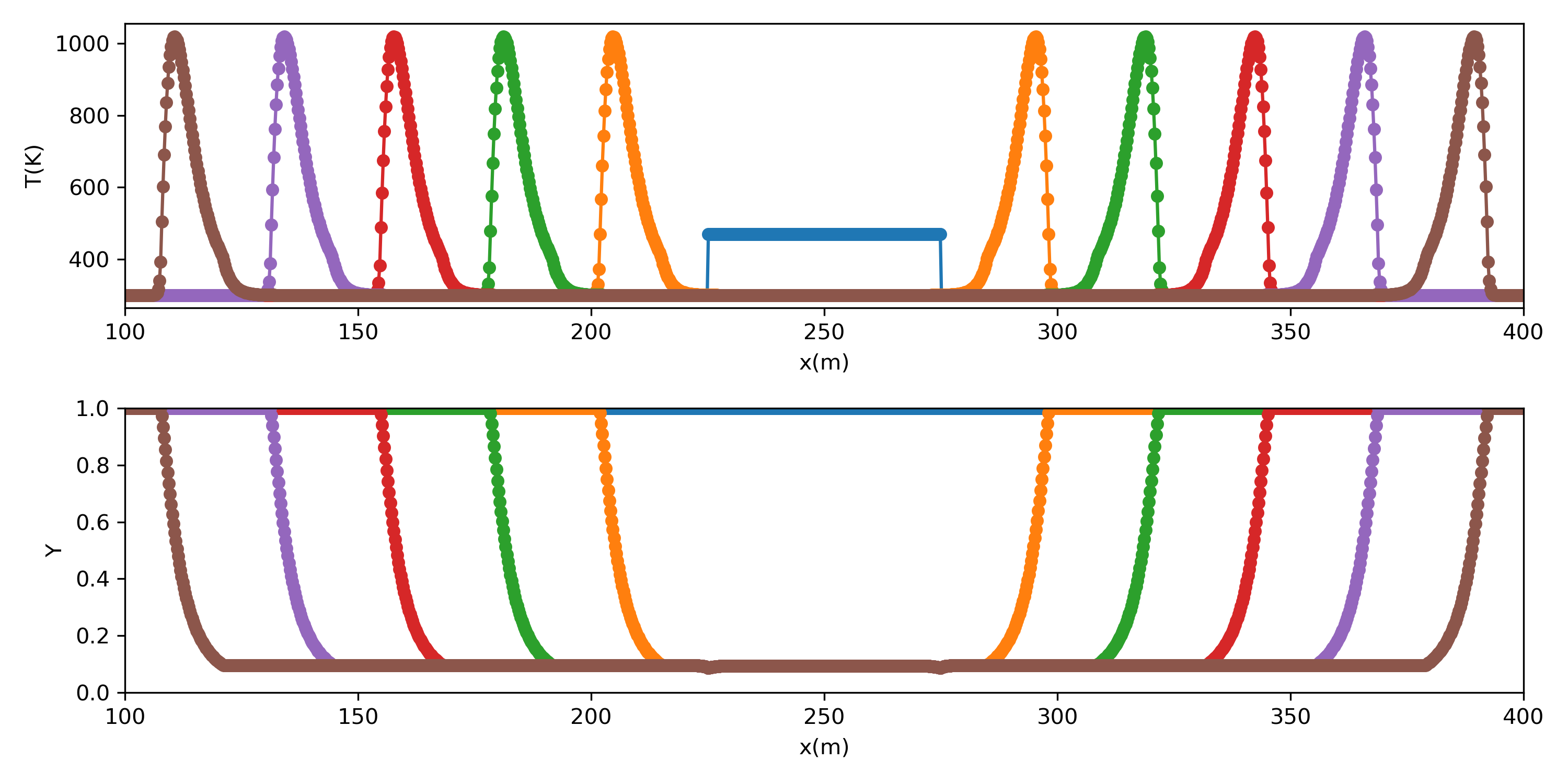}
	 \caption{Numerical example with travelling wave solution for case A-6 in Tab.~\ref{table:subcases-h} on a domain $x \in [0,500]$ with $N=2000$ cells.  All other model parameters are from Tab.~\ref{table:general-parameters}. The solution is plotted for every $150$ s.
        }
	 \label{fig:front_propagation_all}
 \end{figure}

The wave speed and the profile of the wave front depend on the model parameters. Fig.~\ref{fig:speed_front_h}(left) shows the influence of the parameter $h$ on the position of the wave front.  In all cases, the dependency of the wave front position over time is linear, after a short transient time.  The slope of these lines is the wave speed and decreases for increasing values of $h$, due to a higher heat release into the atmosphere.  As seen in Fig.~\ref{fig:speed_front_h}, the maximum temperature decreases and the front profile narrows for increasing values of $h$, as the energy dissipation (i.e. heat release into the atmosphere) enhances.

 \begin{figure}
	 \centering
   \begin{subfigure}[b]{0.45\textwidth}
		 \includegraphics[width=\textwidth]{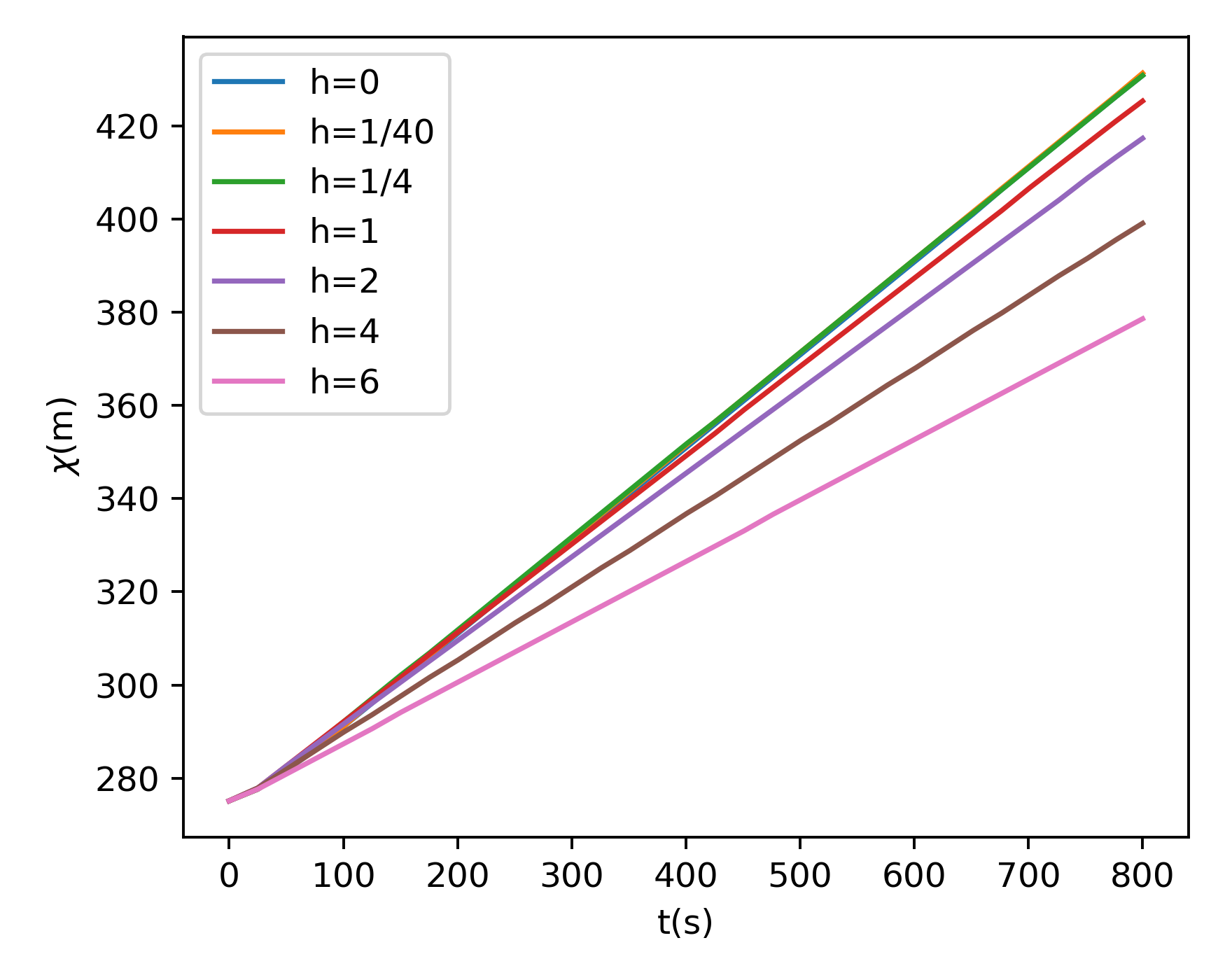}
   \end{subfigure}
      \begin{subfigure}[b]{0.45\textwidth}
		 \includegraphics[width=\textwidth]{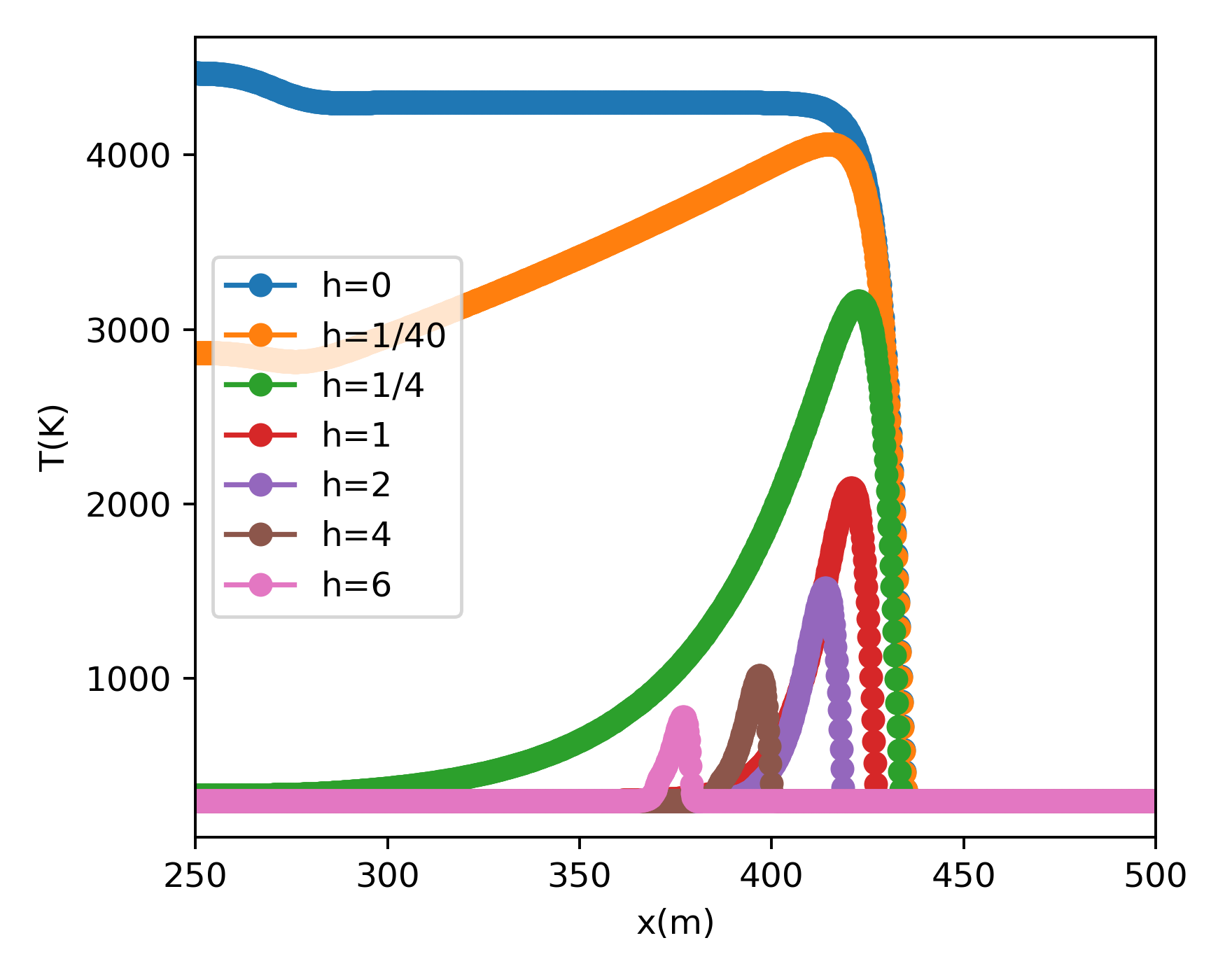}
   \end{subfigure}
	 \caption{Numerical example with travelling wave solution to illustrate the sensitivity of the parameter~$h$. The values of $h$ are chosen according to Tab.~\ref{table:subcases-h}. Left: Position of the wave front depending on the parameter~$h$  at $t=800$ s. Right: Profile of the wave front depending on the parameter~$h$. The parameter $k=2$ kW·m$^{-1}$·K$^{-1}$ is fixed.}
	 \label{fig:speed_front_h}
 \end{figure}

The observed larger maximum temperature and larger wave speed are connected.  For a larger maximum temperature, the spatial gradient of the temperature that drives the wave propagation is larger and thus, leads to a larger wave speed.  This observation is as well supported by numerical studies of the influence of the diffusion parameter $k$ in Fig.~\ref{fig:speed_front_k}, where the six different sub-cases in Tab.~\ref{table:subcases-k} are computed.  Here, larger diffusion parameters lead to a higher wave speed and a wider wave profile. Recall that the diffusion term models the heat transfer due to conduction and radiation.

\begin{figure}
	 \centering
   \begin{subfigure}[b]{0.45\textwidth}
		 \includegraphics[width=\textwidth]{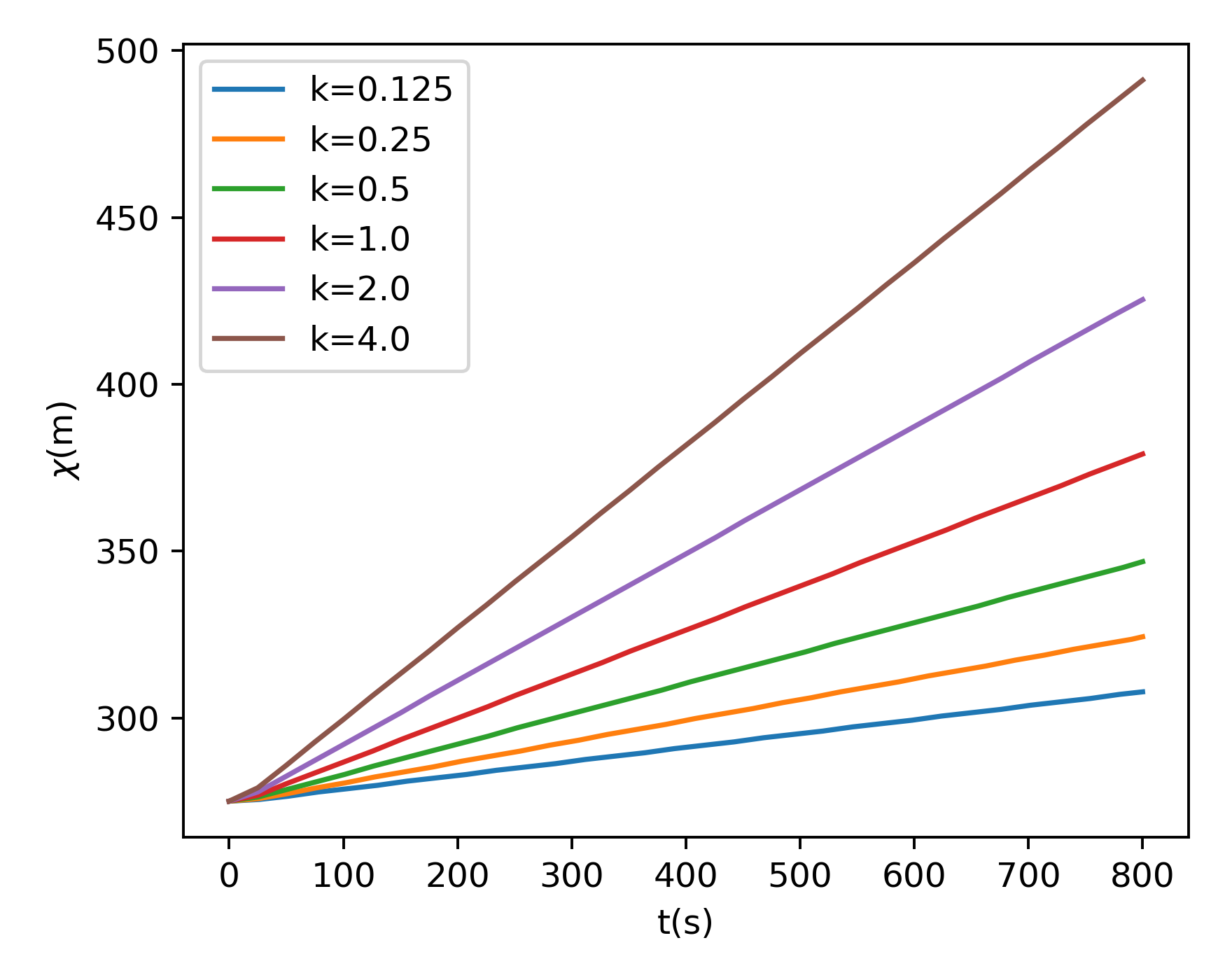}
   \end{subfigure}
      \begin{subfigure}[b]{0.45\textwidth}
		 \includegraphics[width=\textwidth]{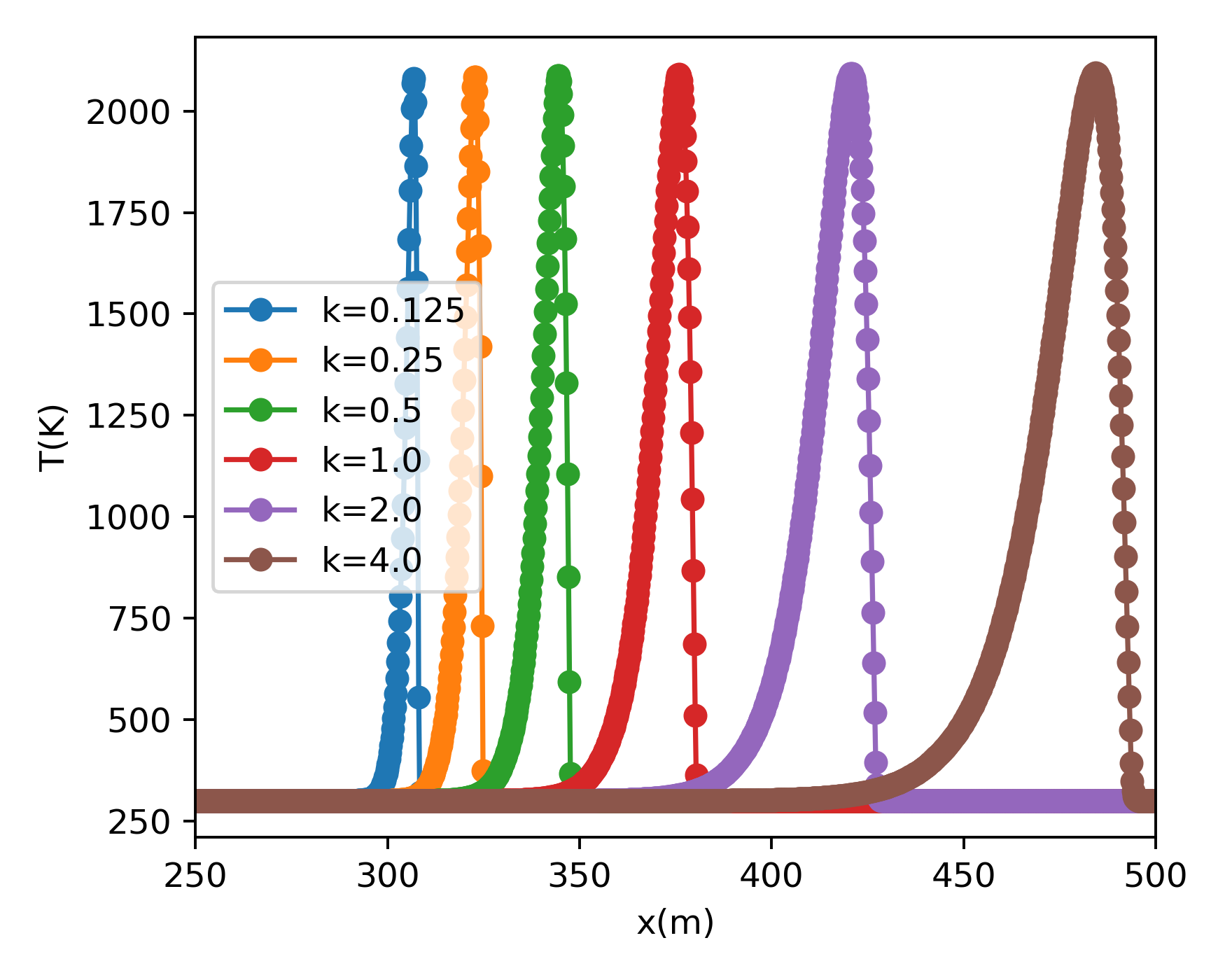}
   \end{subfigure}
	 \caption{Numerical example with travelling wave solution to illustrate the sensitivity of the parameter~$k$. The values of $k$ are chosen according to Tab.~\ref{table:subcases-k}. Left: Position of the wave front depending on the diffusion parameter $k$. Right: Profile of the wave front depending on the parameter $k$. The parameter $h=1$ kW·m$^{-3}$·K$^{-1}$ is fixed.  }
	 \label{fig:speed_front_k}
 \end{figure}

In the literature, results for travelling waves have been obtained for sub-models of the ADR wildfire model.  For example, the evolution of a travelling wave solution was shown in \cite{babak_effect_2009} for a model that neglects the effects of Newton's cooling law in Eq.~\eqref{eq:heatflux2}.  Further neglecting advection, the speed of a travelling wave is bounded by a sub-solution and a super-solution with $c_\star < c <c^\star$.  The values for $c_\star$ and $c^\star$ are given in transformed non-dimensional parameters in \cite{babak_effect_2009}.  A transformation to the dimensional variables and parameters of Eq.~\eqref{eq:model1} gives lower and upper bounds as
\begin{equation}\label{eq:bounds_wavespeed}
    c_\star = \sqrt{k} \sqrt{ \frac{H}{\rho_0  {C}} \frac{1-\frac{ {C}}{H}(T_\mathrm{ac}-T_\infty)}{T_\mathrm{ac} - T_\infty} \frac{A}{\mathrm{e}} }
    < c <
    c^\star = \sqrt{k} \sqrt{\frac{H}{ { \rho_0} {C}}\frac{1-\frac{C}{H}(T_\mathrm{ac}-T_\infty)}{T_\mathrm{ac} - T_\infty}  A \, \mathrm{exp}\left (- \frac{T_\mathrm{ac}}{T_\infty + H/ {C}} \right) }
\end{equation}
for the travelling wave speed $c$.  The extension of the proof to the full model including the heat exchange with the environment is still an open problem.  In the following, we use the values of $c_\star$ and $c^\star$ for comparison with numerically obtained wave speeds.

We apply a different approach by  comparing the system again to a reaction--diffusion equation.
In \cite{hadeler_travelling_1975}, for reaction--diffusion equations of the form
\begin{equation}\label{eq:general_RD}
    \frac{\partial u}{\partial t} = d \frac{\partial^2 u }{\partial x^2} + f(u),
\end{equation}
with a monostable reaction function $f$, the existence of a travelling wave front is proven.  The influence of advection on the front speed in case of a monostable reaction--diffusion equation is discussed in \cite{al-kiffai_lack_2016}.

The extension of the results to an advection--diffusion--reaction model coupled with an ordinary differential equation still needs to be addressed and requires similar results for super- and sub-solutions like the proof of the existence of a travelling wave solution.

Therefore, we regard the reaction--diffusion equation
 \begin{align}  \label{eq:reacdiff_onlyT}
 \begin{aligned}
        \rho_0  {C}  \left(\frac{\partial T}{\partial t}   \right)  &=  k \nabla\cdot  \nabla T      - h(T-T_{\infty}) +  \Psi(T)\rho_{0} H Y,   
 \end{aligned}
 \end{align}
for a fixed $Y$. Then, the conditions of \cite{hadeler_travelling_1975} for a travelling wave solution are fulfilled.  The reaction function has two roots, where one is at $T=T_\infty$ and the other one is the solution of
\begin{equation*}
    \frac{h}{\rho_0  { C}} (T-T_\infty) = \frac{H}{\rho_0 { C}} A Y  \mathrm{exp} \left ({- \frac{T_\mathrm{ac}}{T}} \right) 
\end{equation*}
for fixed $Y$.  Varying $Y$ will vary the root, shifting the value for smaller $Y$ to smaller $T$.

Further, the reaction function $f$ (compare Eq.~\eqref{eq:general_RD}) of Eq.~\eqref{eq:reacdiff_onlyT} is positive between the two roots and the derivative of $f$ is positive for the first and negative for the second root.  Then, a travelling wave solution of the single reaction--diffusion equation in Eq.~\eqref{eq:reacdiff_onlyT} exists and the linearised wave speed of the travelling wave is 
\begin{equation}\label{eq:lin_wavespeed}
    c_\mathrm{lin} = 2 { \sqrt{\frac{\mathrm{d}}{\mathrm{d}T}f(T=T_\infty)}}=2 \sqrt{\frac{k}{\rho_0  { C}} \left ( \frac{H}{ { C}} A Y \frac{T_\mathrm{ac}}{T_\infty^2}  \mathrm{exp} \left (- \frac{T_\mathrm{ac}}{T_\infty} \right )  - \frac{h}{\rho_0  {C}}\right ) },
\end{equation}
where $f$ is the reaction-function according to Eq.~\eqref{eq:general_RD}.
This linearised wave speed may give a lower limit, see \cite{hadeler_travelling_1975}. Here, it is only derived for a sub-model and we investigate its approximation quality numerically.

 \begin{figure}
	 \centering
   \begin{subfigure}[b]{0.45\textwidth}
		 \includegraphics[width=\textwidth]{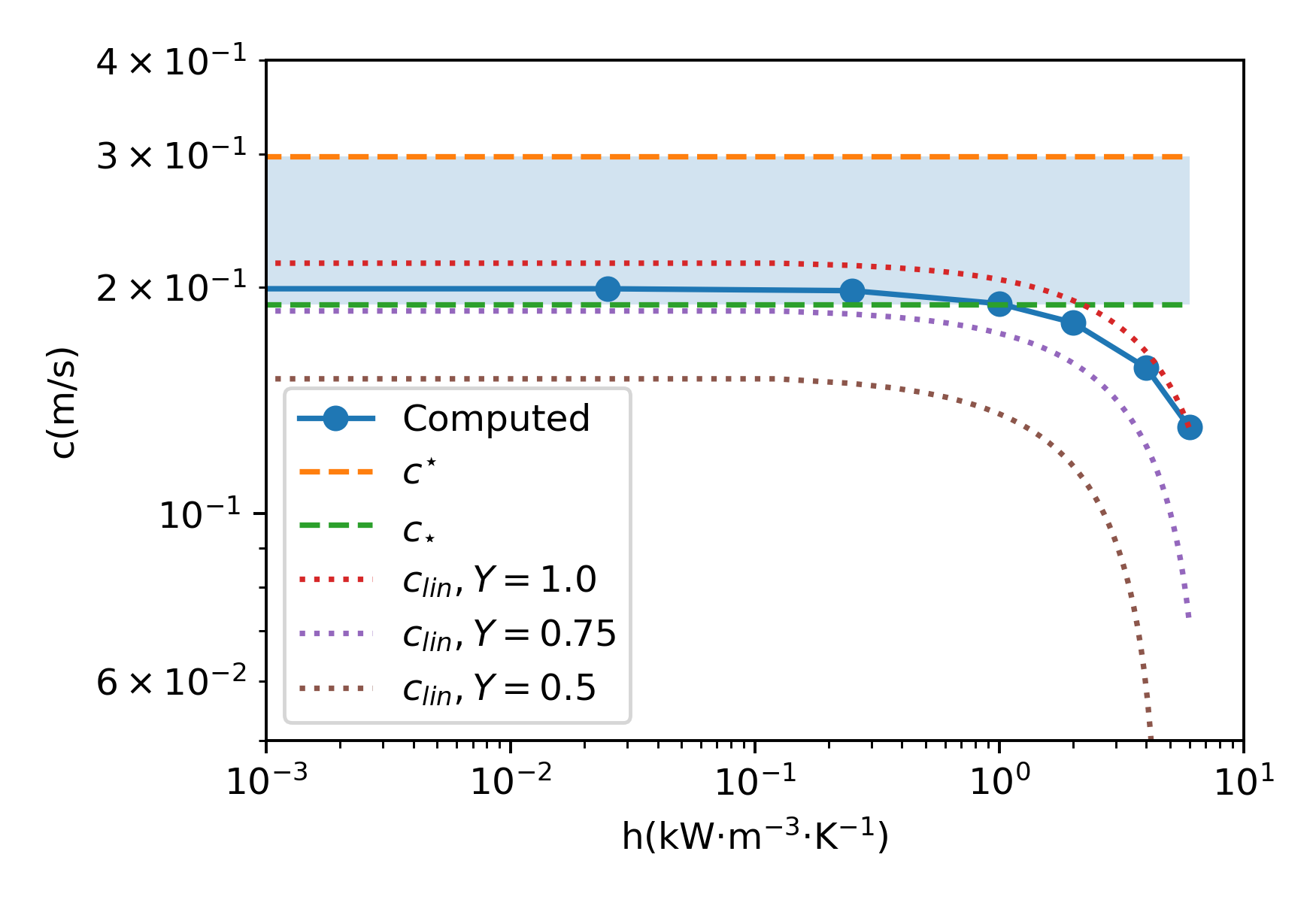}
   \end{subfigure}
      \begin{subfigure}[b]{0.45\textwidth}
		 \includegraphics[width=\textwidth]{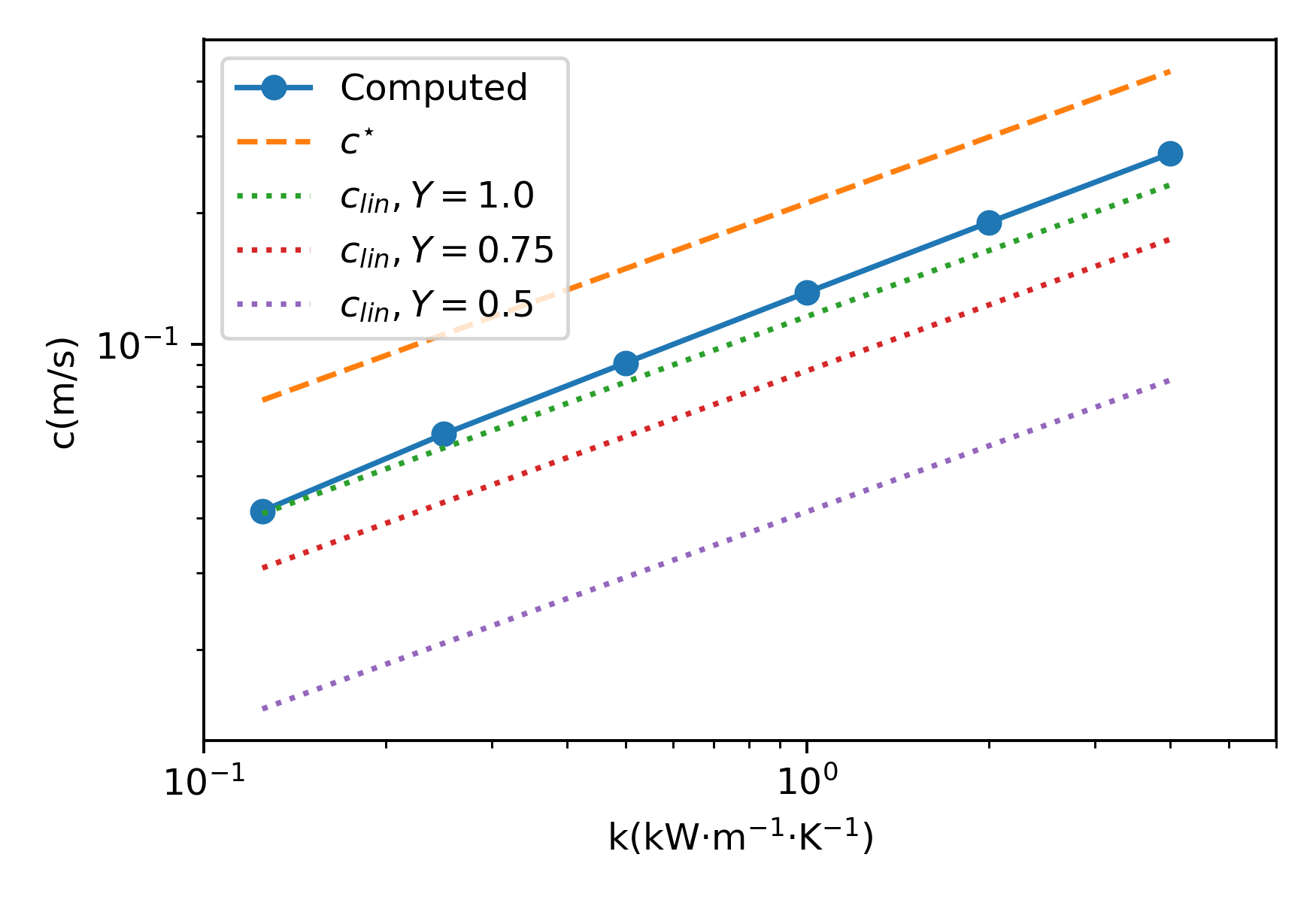}
   \end{subfigure}
	 \caption{Left: Wave speed for different values of $h$ and default $k$. Right: wave speed for different values of $k$ and default $h$. The bounds from Eq.~\eqref{eq:bounds_wavespeed} are depicted in blue, and the linearised wave speed of Eq.~\eqref{eq:lin_wavespeed} is displayed for various fixed values $Y$. }
	 \label{fig:speeds_bounds}
 \end{figure}

Fig.~\ref{fig:speeds_bounds} compares the bounds in Eq.~\eqref{eq:bounds_wavespeed} for the model with $h=0$ and numerical simulations. Additionally, the linearised minimal wave speed of Eq.~\eqref{eq:lin_wavespeed} for the model in Eq.~\eqref{eq:reacdiff_onlyT} is shown. 
The upper bound for $h=0$ in Eq.~\eqref{eq:bounds_wavespeed} is an upper bound for any $h$.  The dependency of the linearised propagation speed in Eq.~\eqref{eq:lin_wavespeed} on $h$ is non-linear with a decrease for increasing $h$.  This behaviour is expected from the results plotted Fig.~\ref{fig:speed_front_h}, where we found that the front speed decreases with increasing $h$.  From an application point of view, this insight is useful because even if $h$ cannot be determined precisely, a lower estimate for $h$ gives conservative estimates for the speed of the fire front.
However, the lower bound of Eq.~\eqref{eq:bounds_wavespeed} is not a lower bound of the model with $h>0$.
For the case $h=0$, we find a dependency proportional to $\sqrt{k}$ of the upper and lower bound of the wave speed on the diffusion parameter $k$.  The upper bound for the wave speed in Eq.~\eqref{eq:bounds_wavespeed} is again an upper bound for the numeric calculated wave speed.  The linearised minimal wave speed from Eq.~\eqref{eq:lin_wavespeed} gives a rough estimate for the dependency on $k$ but not exact bounds.

In Fig.~\ref{fig:ODEvsPDE_phaseY}, we compare the travelling wave speed for different initial biomass values. The smaller the initial biomass is, the slower propagates the combustion wave. This is qualitatively as well the interpretation of the linearised wave speed in Eq.~\eqref{eq:lin_wavespeed}. The dashed line in Fig.~\ref{fig:ODEvsPDE_phaseY}, bottom right, gives the dependency of the linearised travelling wave speed depending on $Y_0$. The travelling wave speed was gained from fixing the biomass $Y$ and regarding only the reaction--diffusion equation in Eq.~\eqref{eq:reacdiff_onlyT}. This rough approximation of the coupled dynamics shows  good results in the computations.

 \begin{figure}
	 \centering
    \includegraphics[width=0.9\textwidth]{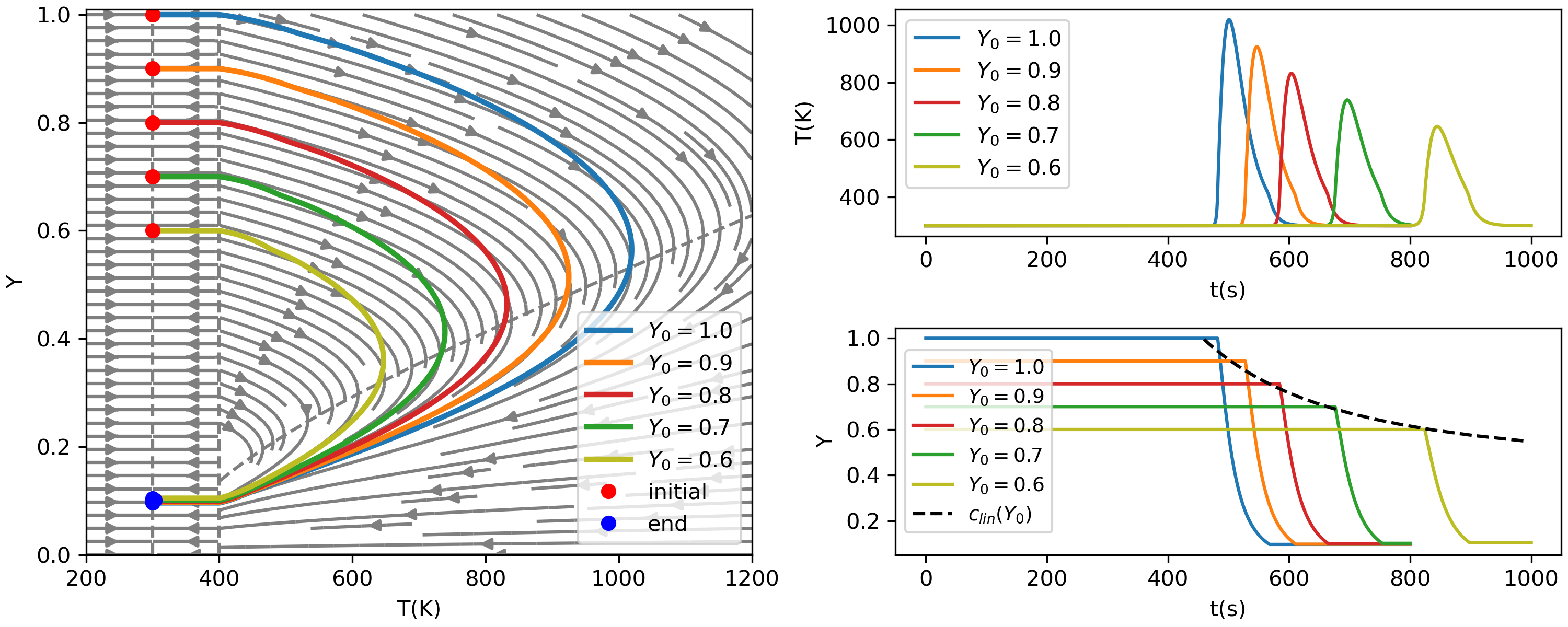}
	 \caption{ Left: Phase space portrait of the solution for case A-6 at $x=350$ m with $k=2$ kW·m$^{-1}$ K$^{-1}$, $h=4$ kW·m$^{-3}$·K$^{-1}$. Right: Temperature $T$ and biomass $Y$ over time. Different initial conditions for the biomass are used, $Y_0=\left\{0.6,0.7,0.8,0.9,1.0 \right\}$.  }
	 \label{fig:ODEvsPDE_phaseY}
 \end{figure}

The simulations in Sec.~\ref{sec:travellingwaves} use a constant diffusion coefficient of $k_t(T)=k \in \mathbb{R}$.  In \cite{Burger:2020a}, the influence of different diffusion coefficients $k_t(T)$ was studied numerically.  A qualitative comparison of constant diffusion and non-constant diffusion from \cite{Burger:2020a} shows that the burnt area is smaller if a constant diffusion $k_t(T)=k$ is used.

\subsubsection{Terminal biomass}\label{sec:terminalbiomass}

The analysis of the lumped sub-model in Eq. \eqref{eq:modelODE} in Sec. \ref{sec:ODE} provides an upper bound for the remaining terminal biomass after the combustion process by the tipping line.  A limitation of this analysis is that spatially explicit diffusion processes can not be modelled through the lumped model.  Thus, the spatially explicit diffusion mechanism in Eq.~\eqref{eq:modelAdf} may lead to a different tipping line and consequently a shift in the upper bound.
In spatially explicit sub-models, at any point $p$ in the domain, the diffusion process leads to a rise in temperature before the combustion process reaches it.  During the combustion process, heat is transferred to cooler regions by the diffusion mechanism.  The increase of temperature at $p$ is amplified by diffusion from warmer regions, leading to a shift of the tipping line in the phase plot to the right-hand side.  The diffusion-induced decrease of temperature leads to a shift of the tipping line to the left-hand side in the phase plot.

This is illustrated in Fig.~\ref{fig:ODEvsPDE_phase} that compares trajectories in the phase space of the lumped sub-model in Eq. \eqref{eq:modelODE} and the advection-free sub-model in Eq.~\eqref{eq:modelAdf} for case A-6 in Tab.~\ref{table:subcases-h} with the initial condition in Eq.~\eqref{eq:ic}.  For the latter, trajectories for two different points ($p = x$) are plotted.  One of these points is in the support of the initial condition, where the diffusion effect is small; the second point $x$ further away from the support of the initial condition, where the travelling wave behaviour dominates.
Due to the initial temperature distribution, there is a difference in the initial value in the phase space portrait in Fig.~\ref{fig:ODEvsPDE_phase} for the two points $p$ in space.  The trajectory for $x=250$~m starts with an initial temperature above the activation temperature, the trajectory for $x=350$~m has an initial temperature smaller than the activation temperature.  The trajectory for a point where diffusion is small ($x=250$~m) follows the trajectories of the ordinary differential equations.  In contrast, the trajectory for a point where the travelling wave front passes ($x = 350$~m) has a decreased temperature compared to the trajectories of the ordinary differential equations.  In this particular case, the remaining biomass in both sub-models is similar. However, this is not true in general.  In some cases, we observed a decreased biomass for the sub-model in Eq.~\eqref{eq:modelAdf} compared to the lumped model in Eq.~\eqref{eq:modelODE}.

\begin{figure}
	 \centering
    \includegraphics[width=0.9\textwidth]{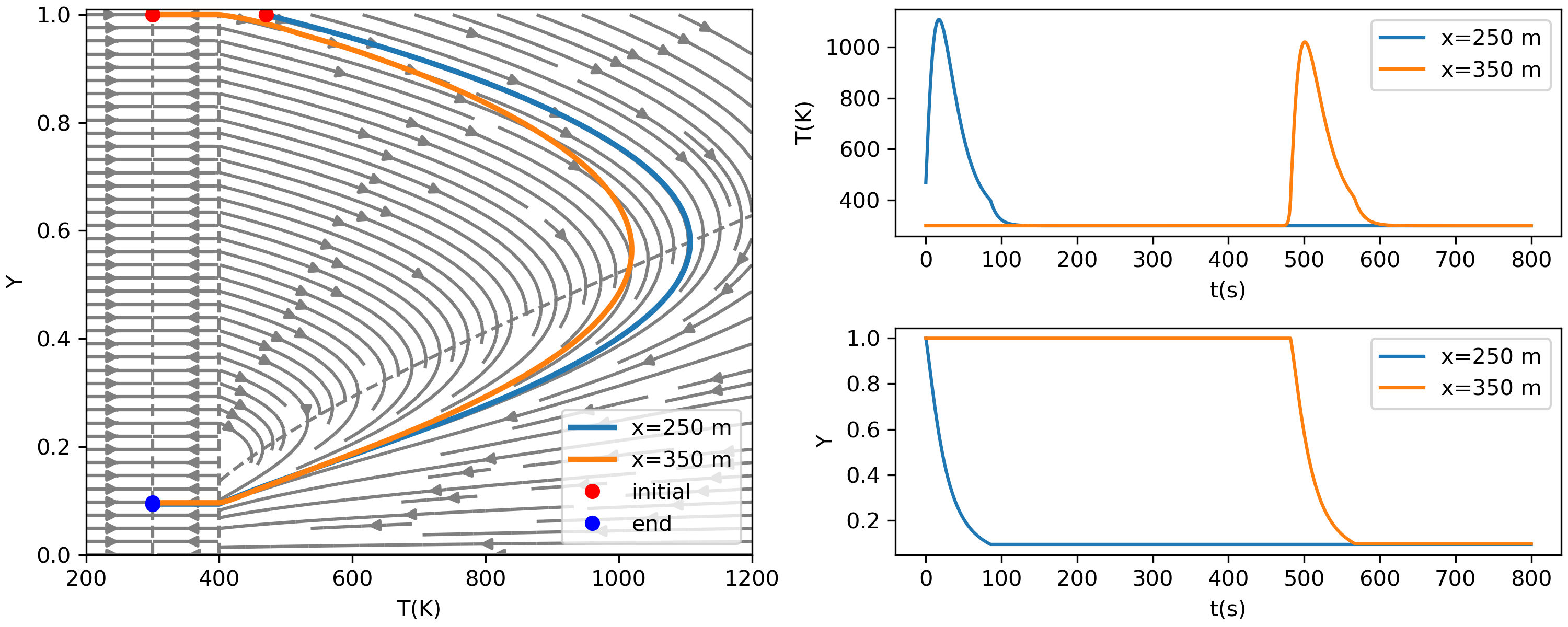}
	 \caption{ Left: Phase space portrait of the solution for case A-6 at $x=250$ and $x=350$ m with $k=2$ kW·m$^{-1}$ K$^{-1}$, $h=4$ kW·m$^{-3}$·K$^{-1}$. Right: Temperature $T$ and biomass $Y$ over time. The initial condition for the biomass is $Y_0=1.0$. }
	 \label{fig:ODEvsPDE_phase}
 \end{figure}

The influence of the initial biomass on the terminal biomass is qualitatively similar to the lumped model in Fig.~\ref{fig:Terminal_map}. Fig.~\ref{fig:ODEvsPDE_phaseY} compares the trajectories for different initial values $Y_0$. We discussed in Sec.~\ref{sec:travellingwaves} the slowing down of the travelling wave speed by smaller initial biomass. In contrast, the initial biomass has a low impact on the terminal biomass.

\subsection{Summary of the analysis of the hierarchical sub-models}

The spatially lumped sub-model in Sec.~\ref{sec:ODE} has a tipping line for the shift between heating and cooling which gives an upper limit for the remaining biomass after the combustion process.  Analysis of the system reveals a continuous line of stationary states due to the switching function in the reaction function modelling the activation and deactivation of the combustion process.  The terminal biomass of the spatially lumped sub-model reacts less sensitive to changes in the initial biomass but more sensitive to changes of the activation temperature $T_\mathrm{ac}$ and changes of the bulk density $\rho_0$ or the heating value $H$.  Changes in the activation temperature can be interpreted as the influence of management strategies such as pre-emptive watering while changes in the vegetation result in a different bulk density and heating value, in the context of the ADR model herein considered.

We define an energy dissipation functional for the combustion-free sub-model in Sec.~\ref{sec:combustionfree} which allows to verify the numerical model for the diffusion and ambient cooling.  The energy dissipation does not change if advection caused by wind is introduced.  Therefore, the verification of the numerical model is valid for the combustion-free sub-model as well with advection.  Meanwhile, this sub-model shows the diffusive behaviour and the exchange with the environment, both leading to a decaying temperature.

The advection-free model in Sec.~\ref{sec:advectionfree} including combustion has travelling wave solutions with steep wave profiles. Such travelling wave behaviour is not caused by wind and advection, but it is a result of the interplay between reaction and diffusion. After a short transient phase, a wave profile forms and propagates with a constant speed in space. The wave speed depends on the reaction parameters and the diffusion constant, compare Fig.~\ref{fig:speed_front_h} and~\ref{fig:speed_front_k}.  As a reference value for the wave speed, we find upper and lower bounds for $h=0$.  For a constant biomass, the linearised wave speed gives an approximation for the dependency on the wave speed as well.   When the initial biomass is reduced, there is a noticeable decrease in wave speed, compare Fig.~\ref{fig:ODEvsPDE_phaseY}. However, this reduction in initial biomass does not have a significant impact on the terminal biomass, which remains relatively unchanged.

\section{Insights from the full model}\label{sec:full}

\subsection{One-dimensional wildfire simulation}

The full ADR wildfire model allows to study the interplay of combustion, temperature exchange with the environment, diffusion, and advection.
Here, we run several one-dimensional wildfire simulations using the full ADR wildfire model in Eq.~\eqref{eq:model1} with different wind speed $v$ as summarised in Tab.~\ref{table:subcases-v}. In general, the advection speed of the wildfire should always be lower than the wind speed. In fact, a correction factor must be used to obtain the advection speed \cite{Grasso:2020}, but herein, for the sake of simplicity, we will  set both of these speeds equal.

\begin{table}
\centering
\scalebox{0.99}{
\begin{tabular}{c|ccccccc}
Sub-case & C-1 & C-2 & C-3 &C-4 &C-5 &C-6  & C-7 \\
\hline
$v$ (m/s)  & 0.00 & 0.02 & 0.04 & 0.06 & 0.08 & 0.09 & 0.1   \\
\end{tabular}}
\caption{Cases computed for different wind speed $v$.}
\label{table:subcases-v}
\end{table}

\begin{figure}
	 \centering
    \includegraphics[width=0.95\textwidth]{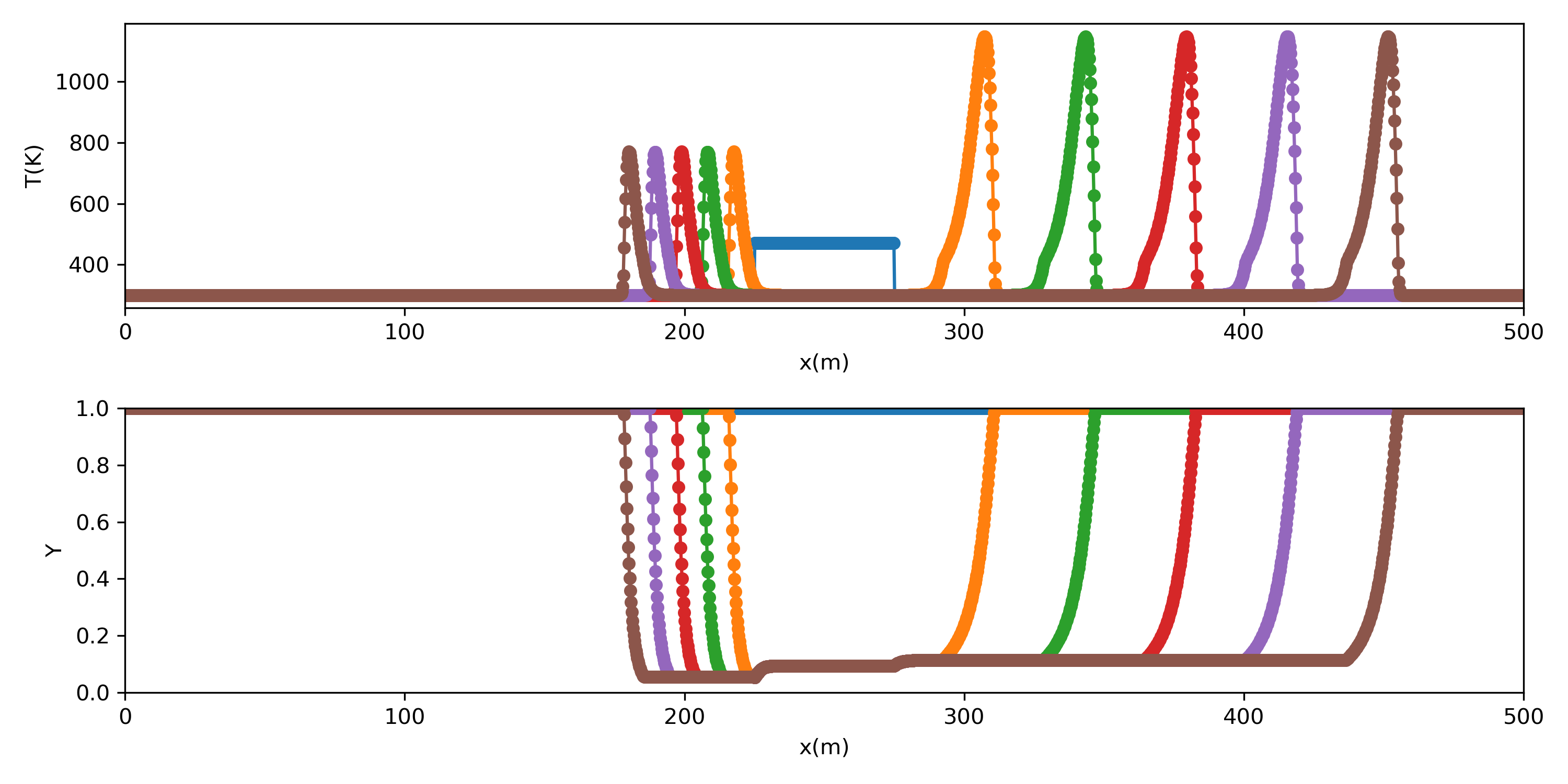}
	 \caption{Solutions of the model in Eq.~\eqref{eq:model1} each 150 s for constant diffusion and case C-5 of Tab.~\ref{table:subcases-v} with $v=0.08$~m/s. The domain is [0,500] m and the discretisation uses $N=2000$ cells. The parameters have the default values of Tab.~\ref{table:general-parameters} with $k=2$ { kW·m$^{-1}$ K$^{-1}$} and $h=4${ kW·m$^{-3}$·K$^{-1}$}.}
	 \label{fig:front_propagation_all_w}
 \end{figure}

Adding wind into the system dynamics changes the amplitude and propagation speed of the waves propagating in the same and in the opposite direction of the wind.  Fig.~\ref{fig:front_propagation_all_w} shows the wave profile and the spread of the waves for different times $t$ and a constant wind speed $v=0.08$ m/s.
Compared to the simulations without wind in Fig.~\ref{fig:front_propagation_all}, the temperature wave propagating to the left and to the right have different amplitudes, profiles and speed.  The biomass consumption differs for the two directions.
While the temperature for the wave propagating to the right is much higher, the remaining biomass is higher in this case.  This suggests that temperature is not the primary control of the remaining biomass.  This observation is consistent with our study of the lumped sub-model in Fig.~\ref{fig:Terminal_map}.
A higher temperature leads to a shift to the right in the phase space of Fig.~\ref{fig:Terminal_map}.
The lumped sub-model does not react sensitive to this change with respect to the remaining biomass.
Hence, the full model highlights an additional dependency of the remaining biomass.
As the wave propagating to the left has a much lower wave speed, the combustion time at a certain position $x$ is much longer than for the wave propagating to the right.
This leads to a longer burning process at this position and therefore to a lower remaining biomass.

Fig.~\ref{fig:ODEvsPDE_phase_w}(left) illustrates the difference in the trajectories through the phase space of the lumped sub-model and the full ADR wildfire model.  For the latter, trajectories for three distinct points $p = x$ are plotted.
The blue line shows a trajectory of a point inside the support of the initial temperature pulse.  It follows closely the trajectory of the lumped sub-model.
The orange line shows the trajectory of a point at the right-hand side of the initial temperature subject to a travelling wave.
Here, the phase plot shows an increased temperature for the full ADR wildfire model trajectory, compared to the lumped sub-model, which is caused by the transport of temperature in the wind direction.  The opposite effect is seen for the green curve, where the wind direction and the wave propagation are opposite.  As a consequence, the temperature is lower for medium temperatures but the combustion process at one point in space takes longer.  This leads to a lower temperature after the combustion process but to a lower terminal biomass at the equilibrium state.  Fig.~\ref{fig:ODEvsPDE_phase_w}(right) shows the profiles of the waves propagating to the left and to the right.  The profile is steeper on the activation front and the amplitude of the temperature is higher for the wave where wind and propagation have the same direction.

\begin{figure}
	 \centering
    \includegraphics[width=0.9\textwidth]{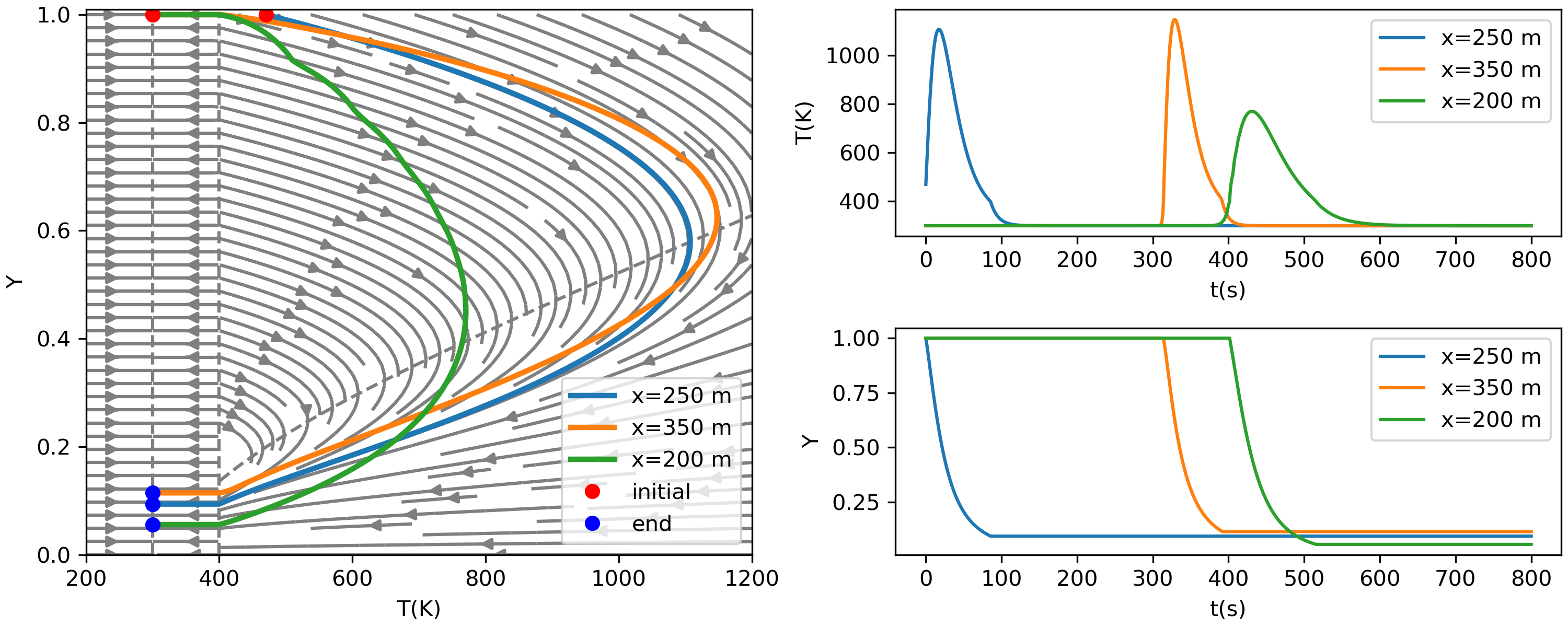}
	 \caption{Phase space trajectories for three points for case C-5 in Tab.~\ref{table:subcases-v}. At $x=250$ m, the initial condition is non-zero and the behaviour is dominated by the ordinary differential equation. The gray lines give the trajectories of the ordinary differential equation, compare Fig.~\ref{fig:ODE_map}.  At $x=200$ m, the left travelling wave passes by and at $x=350$ m the right travelling wave passes by.}
	 \label{fig:ODEvsPDE_phase_w}
 \end{figure}

Fig.~\ref{fig:profiles_w} shows the wave profiles of the temperature for different wind speed.  The amplitude and the speed of the wave propagating in the opposite direction to the wind decrease with higher wind speed.  For a certain wind speed, the opposite wave even stops propagating and only one wave is travelling in the wind direction.  This qualitative change of the system's behaviour is as well depicted in Fig.~\ref{fig:velocities_w}.

 \begin{figure}
	 \centering
   \begin{subfigure}[b]{0.45\textwidth}
		 \includegraphics[width=\textwidth]{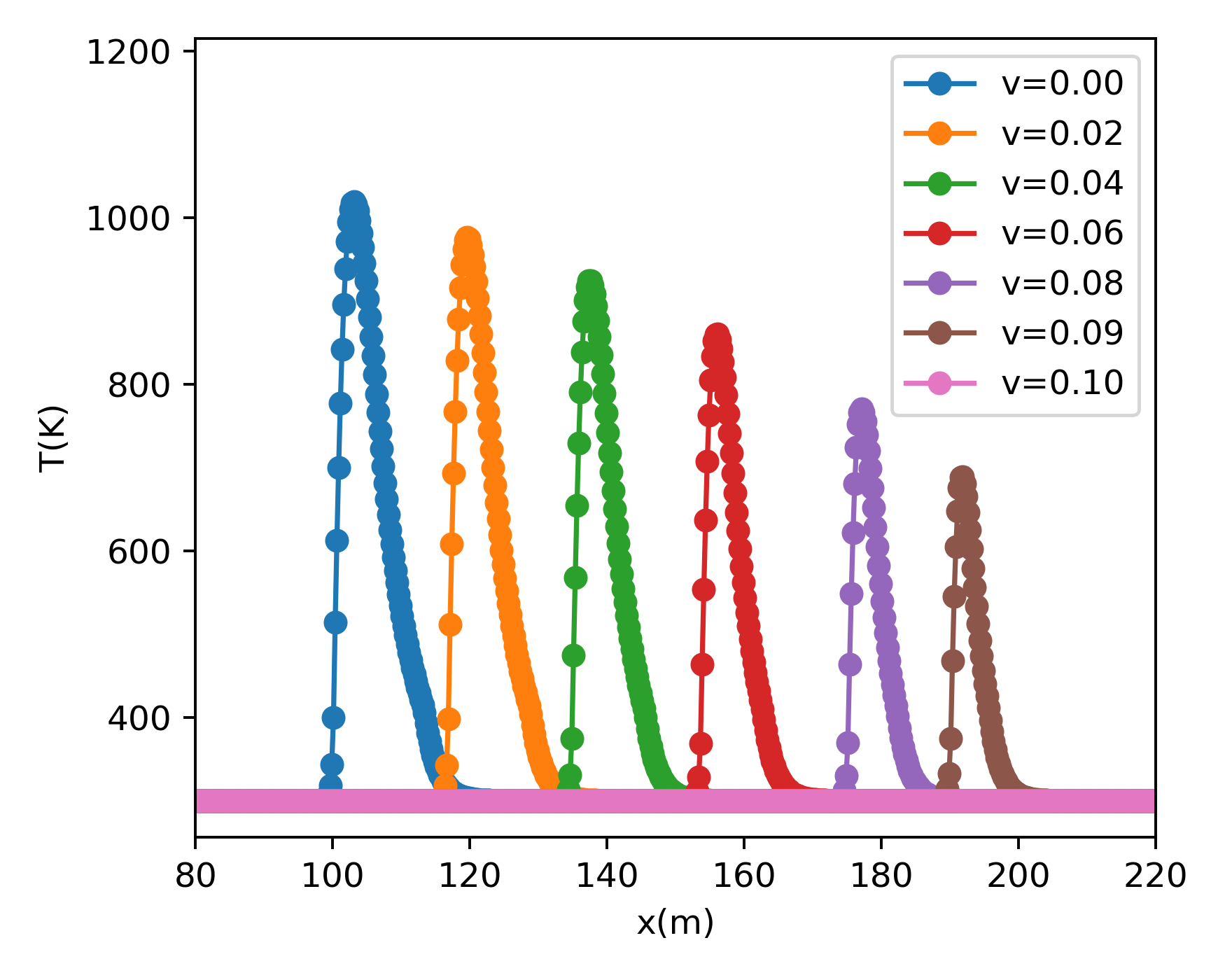}
   \end{subfigure}
      \begin{subfigure}[b]{0.45\textwidth}
		 \includegraphics[width=\textwidth]{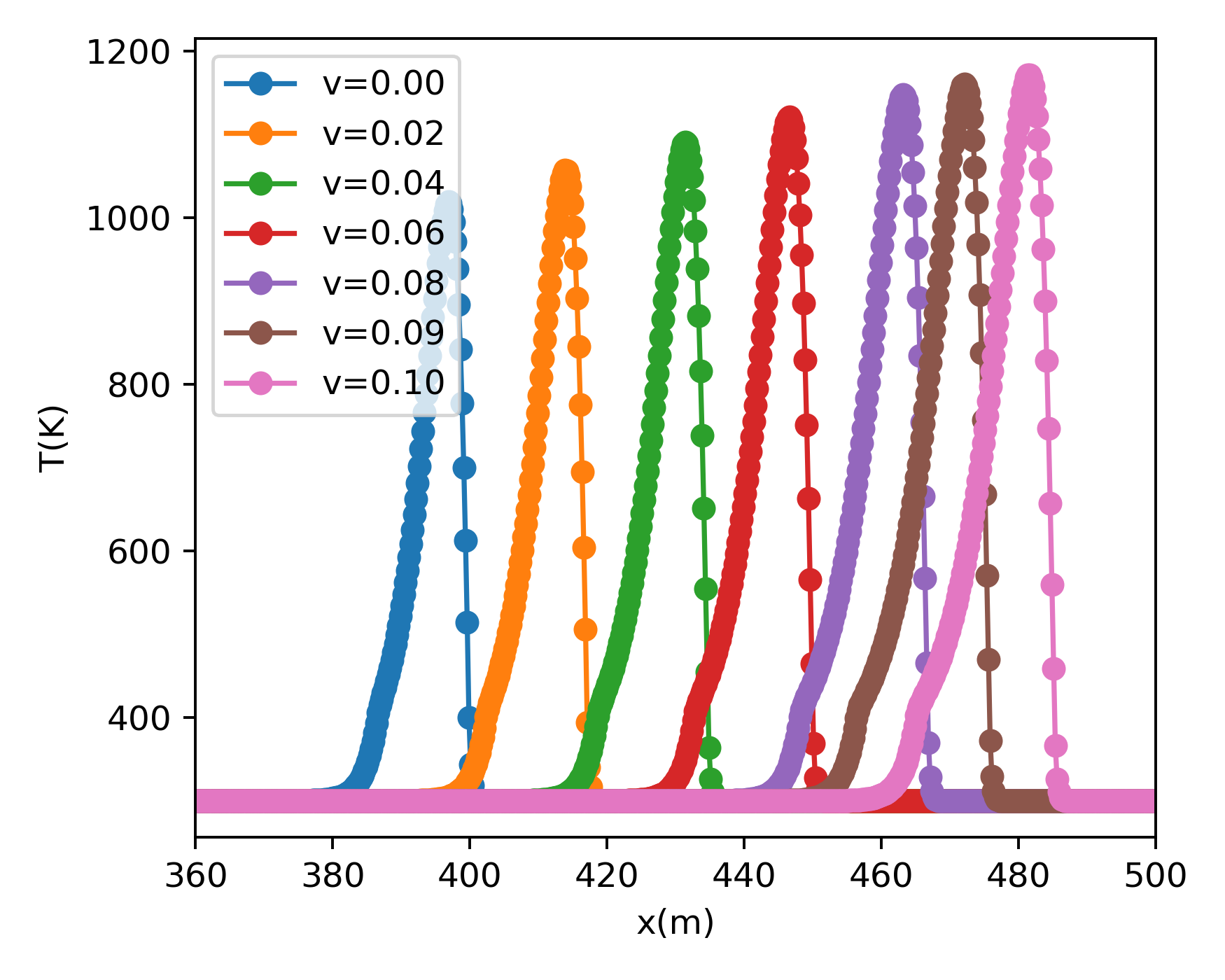}
   \end{subfigure}
	 \caption{ Profiles of the travelling temperature waves to the left and to the right for fixed parameters and different wind speed. For $v=0.1$ m/s, there is no wave propagating to the left.
  }
	 \label{fig:profiles_w}
 \end{figure}

  \begin{figure}
	 \centering
   \begin{subfigure}[b]{0.45\textwidth}
		 \includegraphics[width=\textwidth]{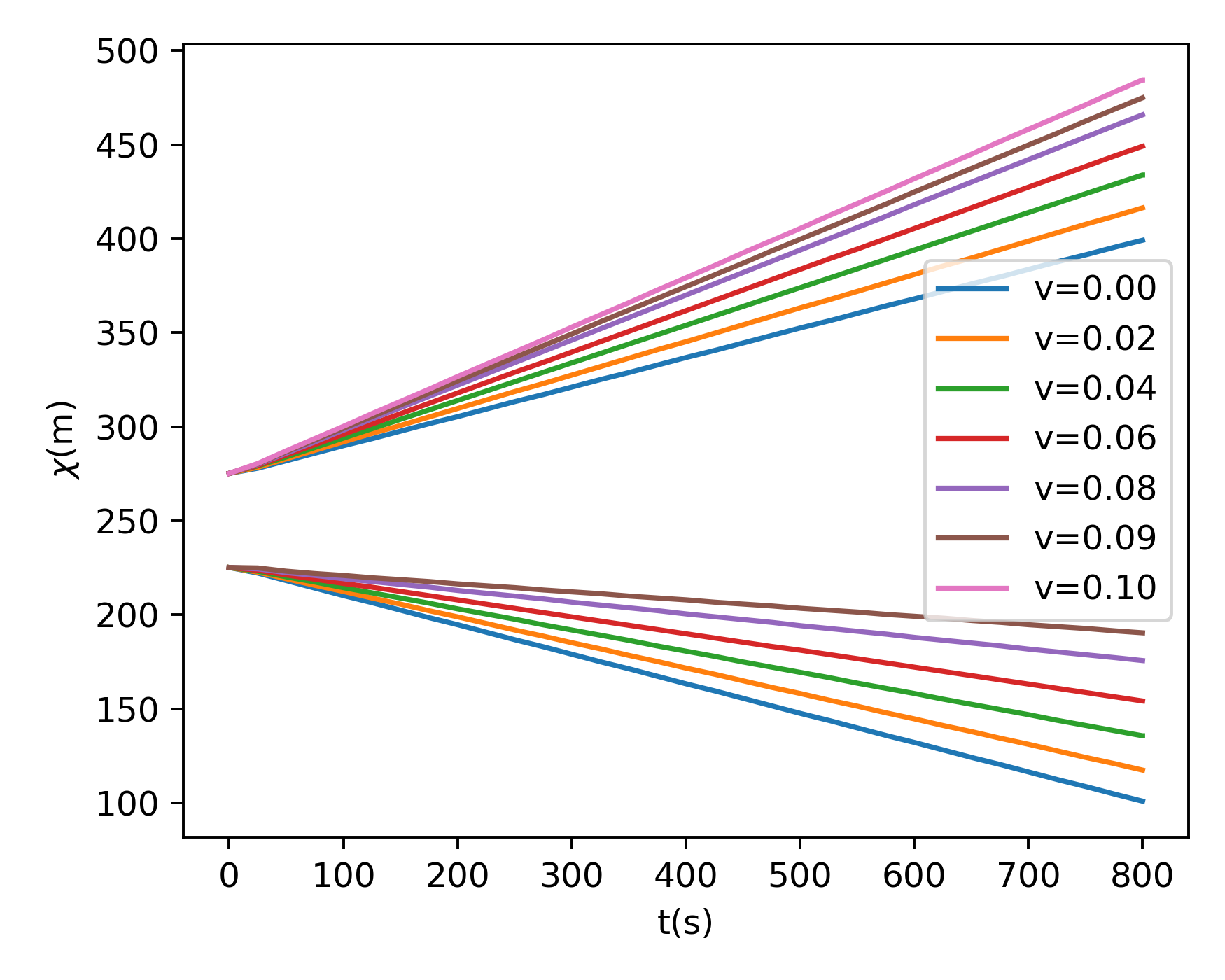}
   \end{subfigure}
      \begin{subfigure}[b]{0.45\textwidth}
		 \includegraphics[width=\textwidth]{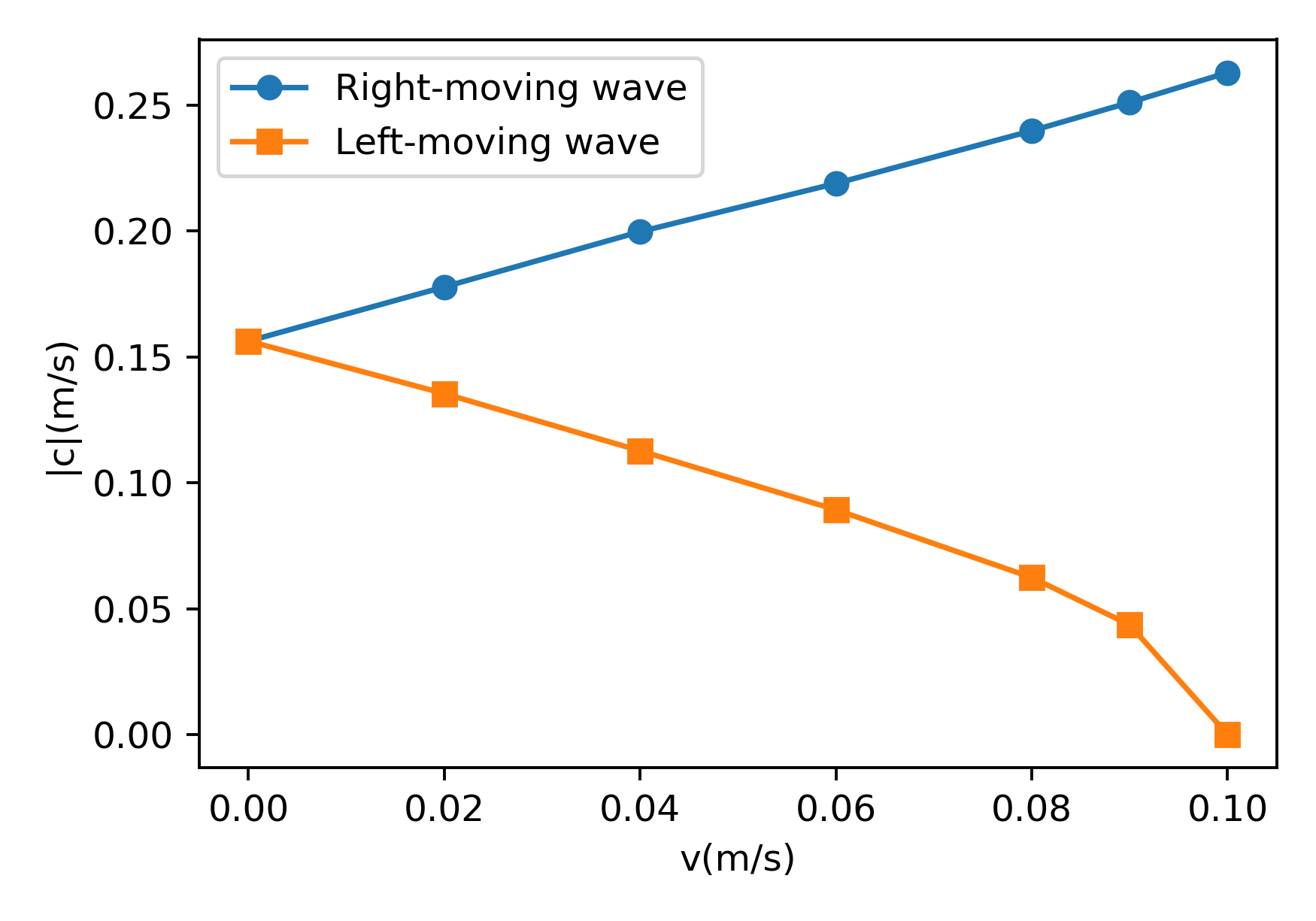}
   \end{subfigure}
	 \caption{ Left: Position of the wave profiles for different wind speed. Right: Propagation speed of the traveling wave profiled depending on the wind speed.
  }
	 \label{fig:velocities_w}
 \end{figure}

The propagation speed of the travelling wave in the opposite direction to the wind decreases approximately linearly until a certain threshold is reached.  Then, the speed drops to zero---in our case at $v=0.1$ m/s.  For this speed, which depends on the fixed parameters in Tab.~\ref{table:general-parameters}, the solution switches from two travelling waves to only one.  The travelling wave propagating in the opposite direction of the wind stops.

   \begin{figure}
	 \centering
   \begin{subfigure}[b]{0.45\textwidth}
		 \includegraphics[width=\textwidth]{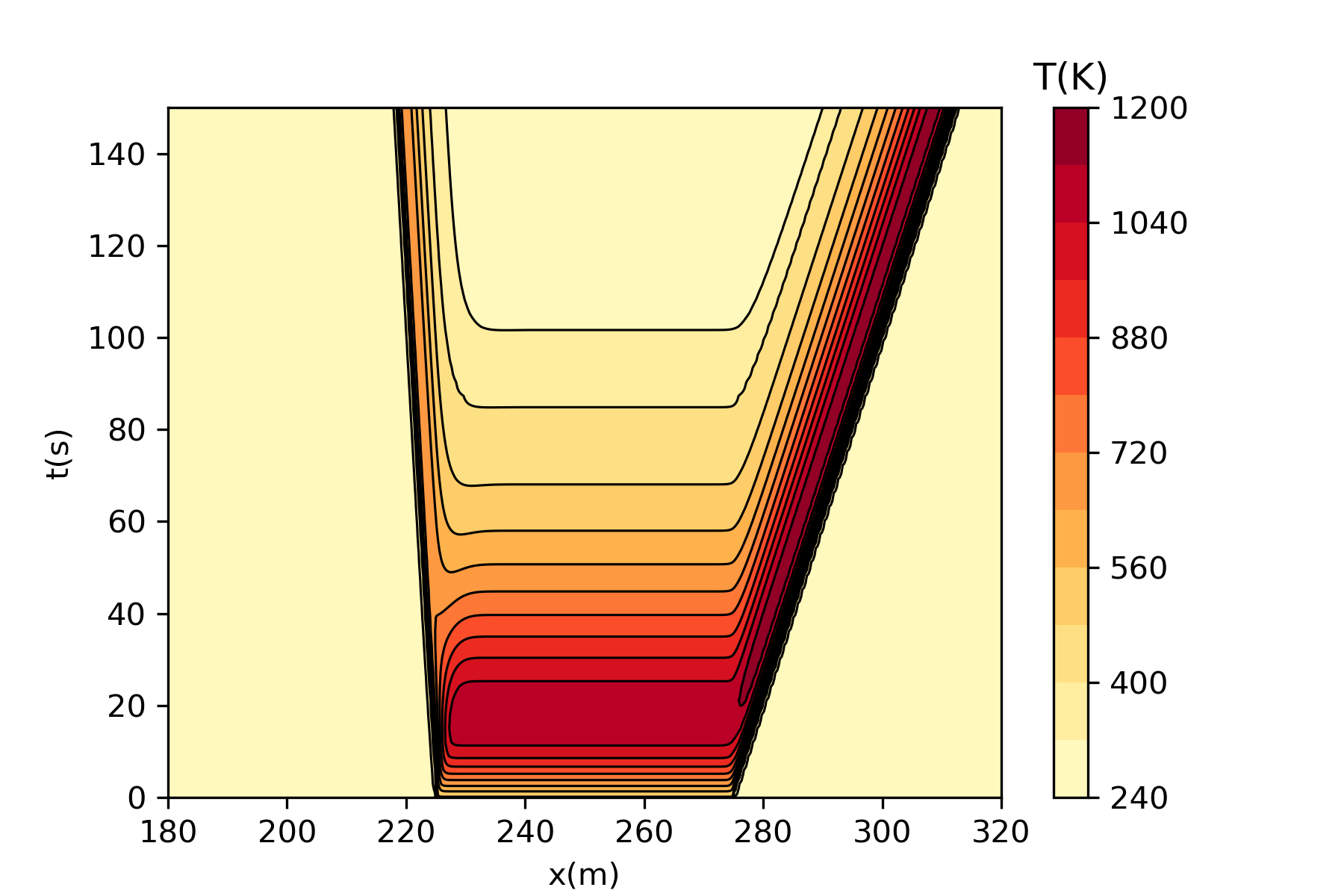}
   \end{subfigure}
   \begin{subfigure}[b]{0.45\textwidth}
		 \includegraphics[width=\textwidth]{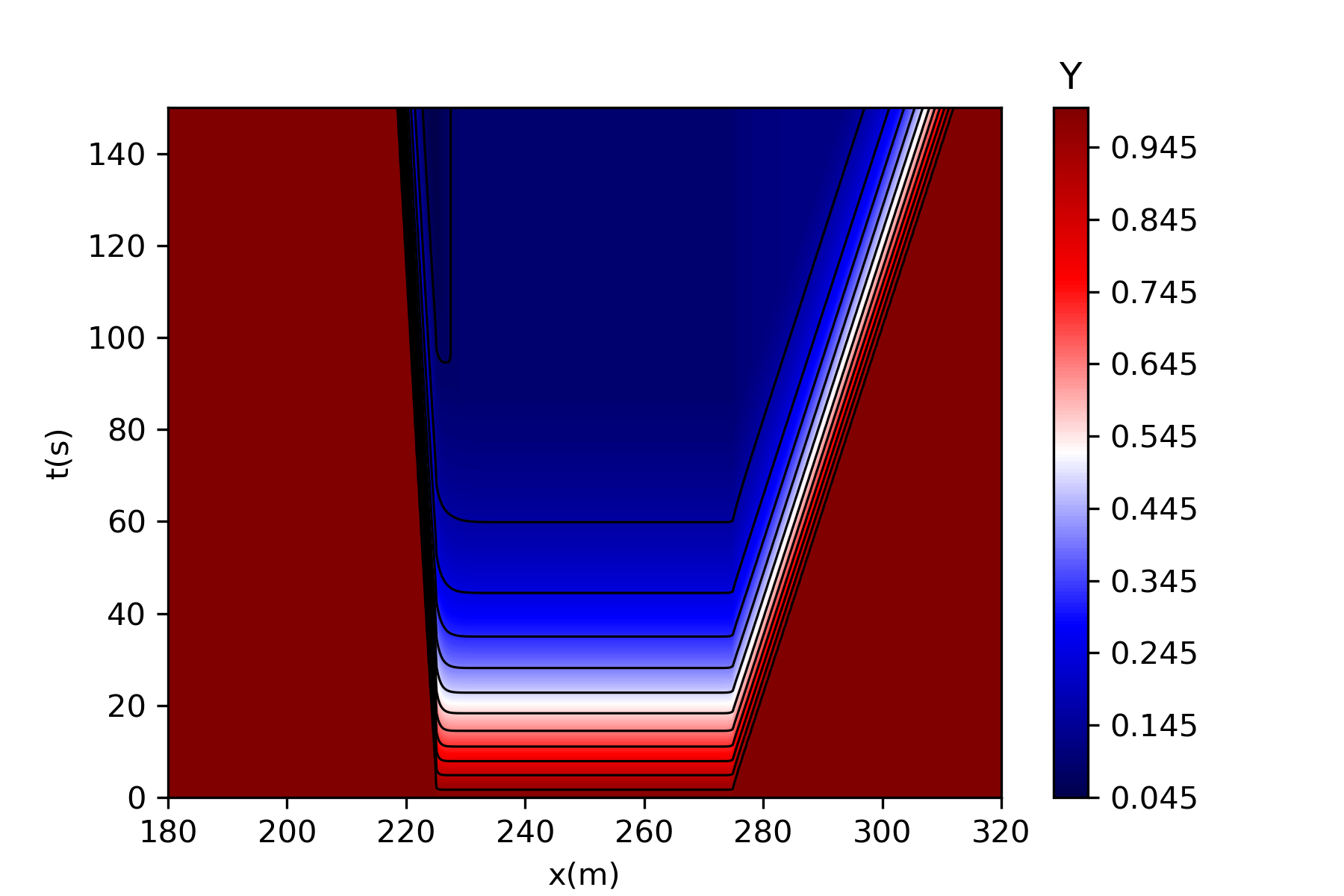}
   \end{subfigure}
   \begin{subfigure}[b]{0.45\textwidth}
		 \includegraphics[width=\textwidth]{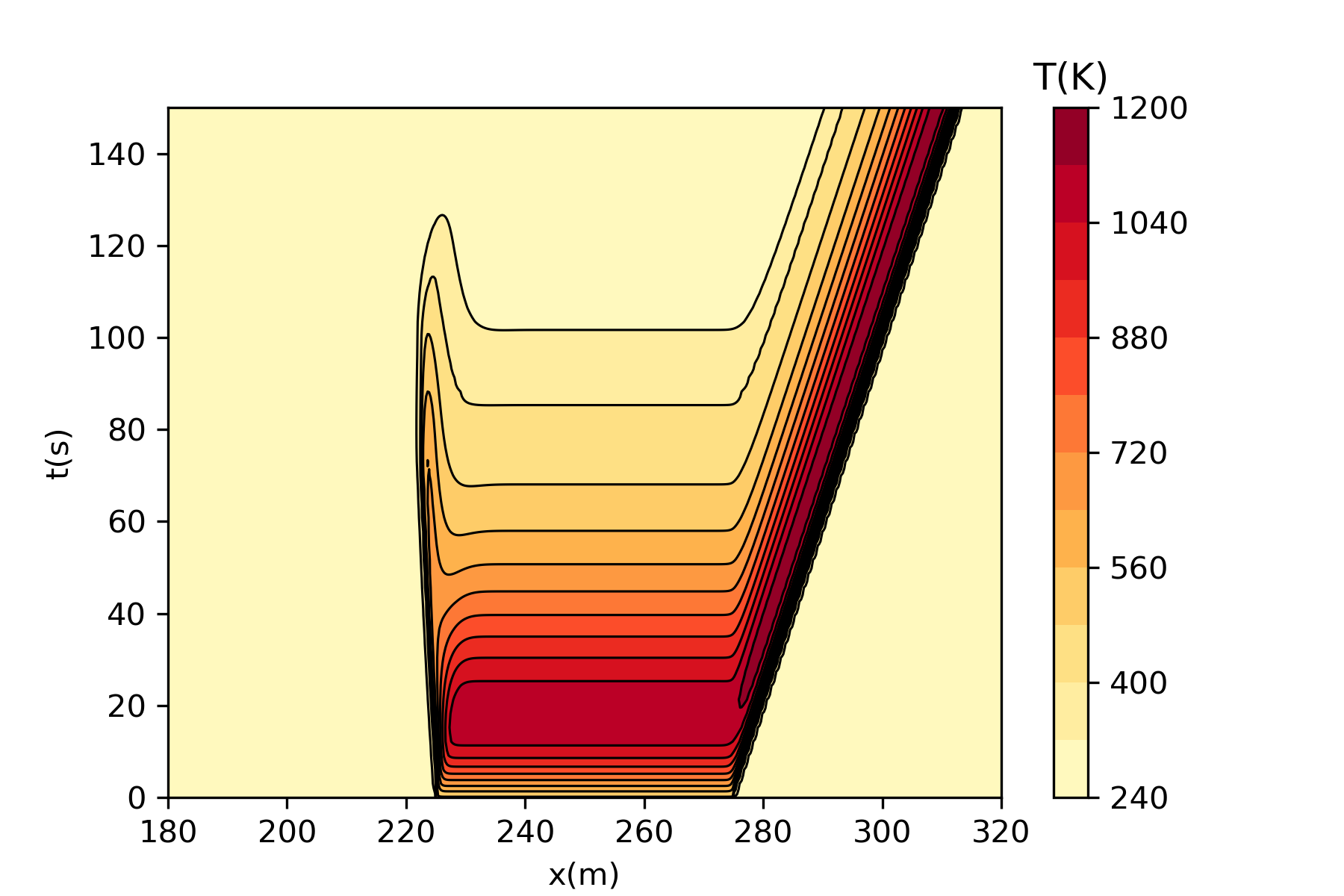}
   \end{subfigure}
   \begin{subfigure}[b]{0.45\textwidth}
		 \includegraphics[width=\textwidth]{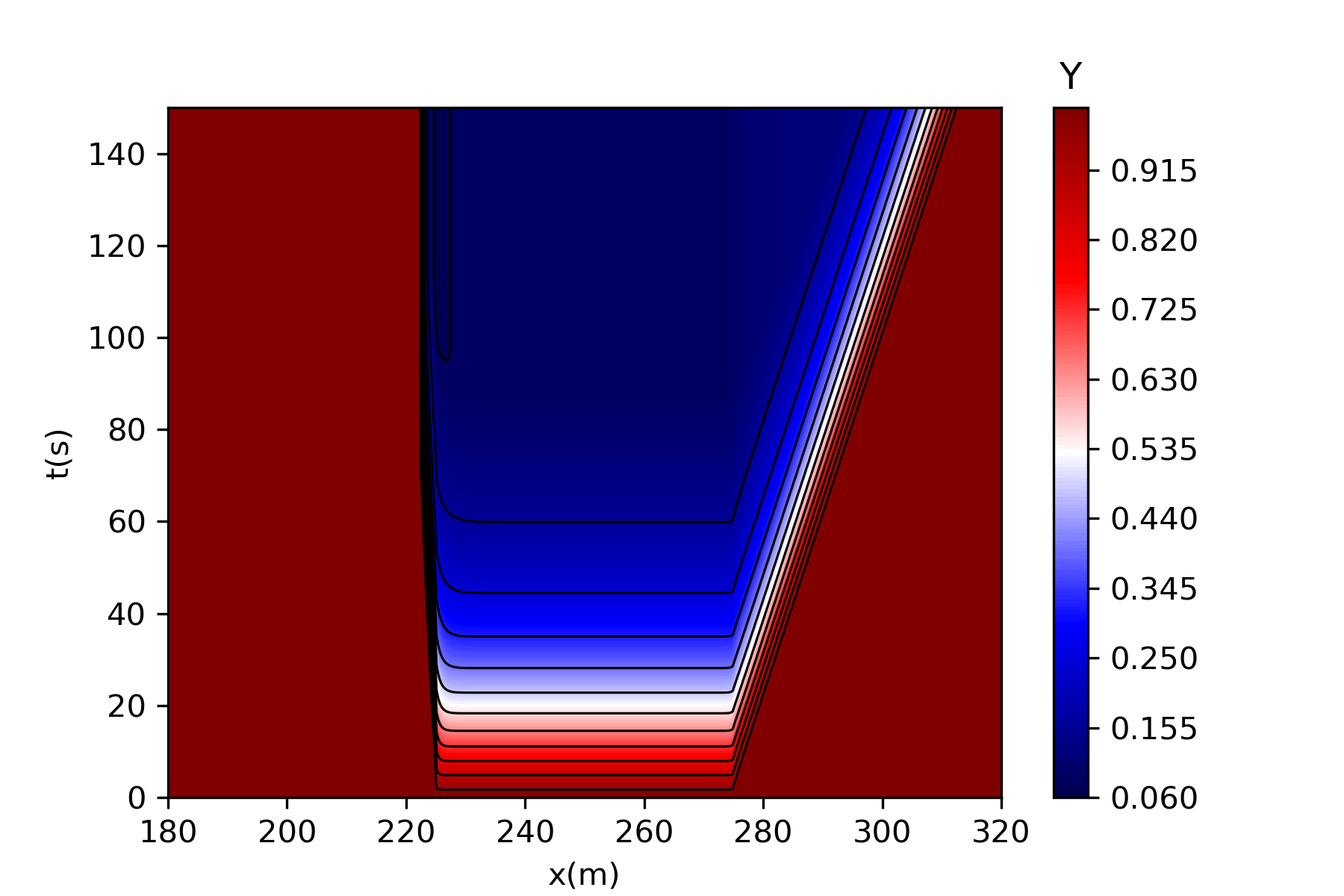}
   \end{subfigure}
	 \caption{ Temperature and biomass evolution for different wind speed $v=0.09$ m/s (top) and $v=0.093$ m/s (bottom) and the parameters of Tab.~\ref{table:general-parameters}.
}
	 \label{fig:evol_detail}
 \end{figure}

Fig.~\ref{fig:evol_detail} compares the evolution of temperature and biomass over time for a wind speed with two travelling waves and with only one.
For $v=0.09$ m/s, the travelling wave propagating opposite to the wind direction still exists, even if the propagation speed is slow.
In contrast, for $v=0.093$ m/s, the wave propagating against the wind stops spreading after some time.
The temperature profile does not form towards a travelling wave profile and there is only one wave in the direction of the wind, moving to the right-hand side.

\subsection{Two-dimensional wildfire simulation}

We simulate wildfire in a forest ecosystem with a domain of $[0,500]\times[0,500]$ m$^2$. The forest features spatially heterogeneous initial biomass organised in patches, defined as a multi-scale random distribution. Random maps of different scales (i.e. 2, 5, 10 and 100 m) are combined to produce the initial condition for the biomass, $Y(x,y,0)$. The initial biomass distribution resembles areas with reforestation, for example in areas with previously dominant monoculture forest. The initial condition for the temperature is given by

\begin{equation}\label{eq:icreal}
   T(x,y,0) = 
		\left\{
		\begin{array}{lll}
      	470 \mbox{ K}  &\mbox{if} &  r(x,y)< 15 \mbox{ m} \\
        300 \mbox{ K} & \multicolumn{2}{l}{\mbox{otherwise}}  \\  
		\end{array}
		\right. 
\end{equation}
with 
\begin{equation}\label{colliding3}
r(x,y)=\sqrt{(x-x_1)^2+(y-y_1)^2}
\end{equation}
with $(x_1,y_1)=(50,50)$ m.  The computational mesh has 500 cells in each Cartesian direction. The simulation runs for 500 seconds. The thermal diffusivity is set to $k=3$ kW·m$^{-1}$· K$^{-1}$ and the wind velocity is set to $\mathbf{v}=(0.5,0.5)$ m/s; default values from Tab. \ref{table:general-parameters} are used for the other parameters. 

Simulation results are plotted in Fig.~\ref{fig:realistic_solutions}, where at the top, temperature contours and at the bottom, the burned biomass contours at the end of the simulation are overlaid with the terminal and the initial biomass distribution, respectively. The burnt area near to the ignition point preserves its symmetry more or less, but once the fire reaches a patch with low biomass, the symmetry is broken. The wildfire propagation is slowed down in this region. As the wildfire propagates through the domain, encountering more heterogeneity, the symmetry breaks down even further. The expansion towards the north is severely limited by patches of low biomass that draw a ``border'' along the diagonal of the domain. Besides, when the fire front reaches the patches with lower initial biomass, the temperature of the fire front decreases and thus the severity of the fire, as described in Sec.~\ref{sec:terminalbiomass}. This is how firebreaks and prescribed burn work in wildfire management. This slower propagation speed in areas with lower biomass was discussed in Sec.~\ref{sec:travellingwaves} for the advection-free model. The linearised wave speed in Eq.~\eqref{eq:lin_wavespeed} depends on the biomass. Fig.~\ref{fig:ODEvsPDE_phaseY} also shows this dependency of the advection speed on the biomass. We further observe that the wind speed is sufficiently high to stop the wildfire propagation in the opposite direction. This is consistent with the one-dimensional wildfire model results in Fig.~\ref{fig:velocities_w}.

 \begin{figure}
	 \centering
		 \includegraphics[width=0.7\textwidth]{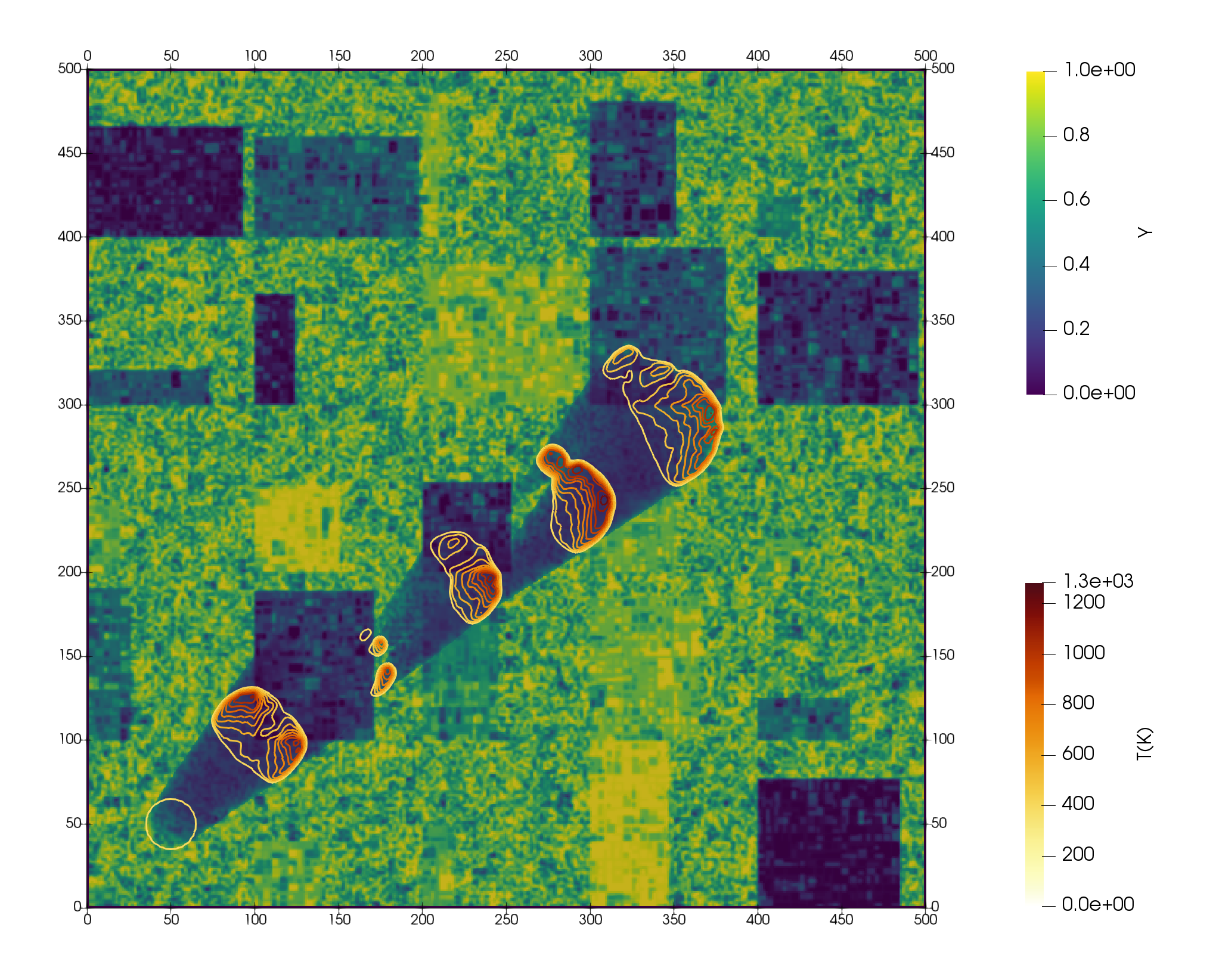}
   \includegraphics[width=0.7\textwidth]{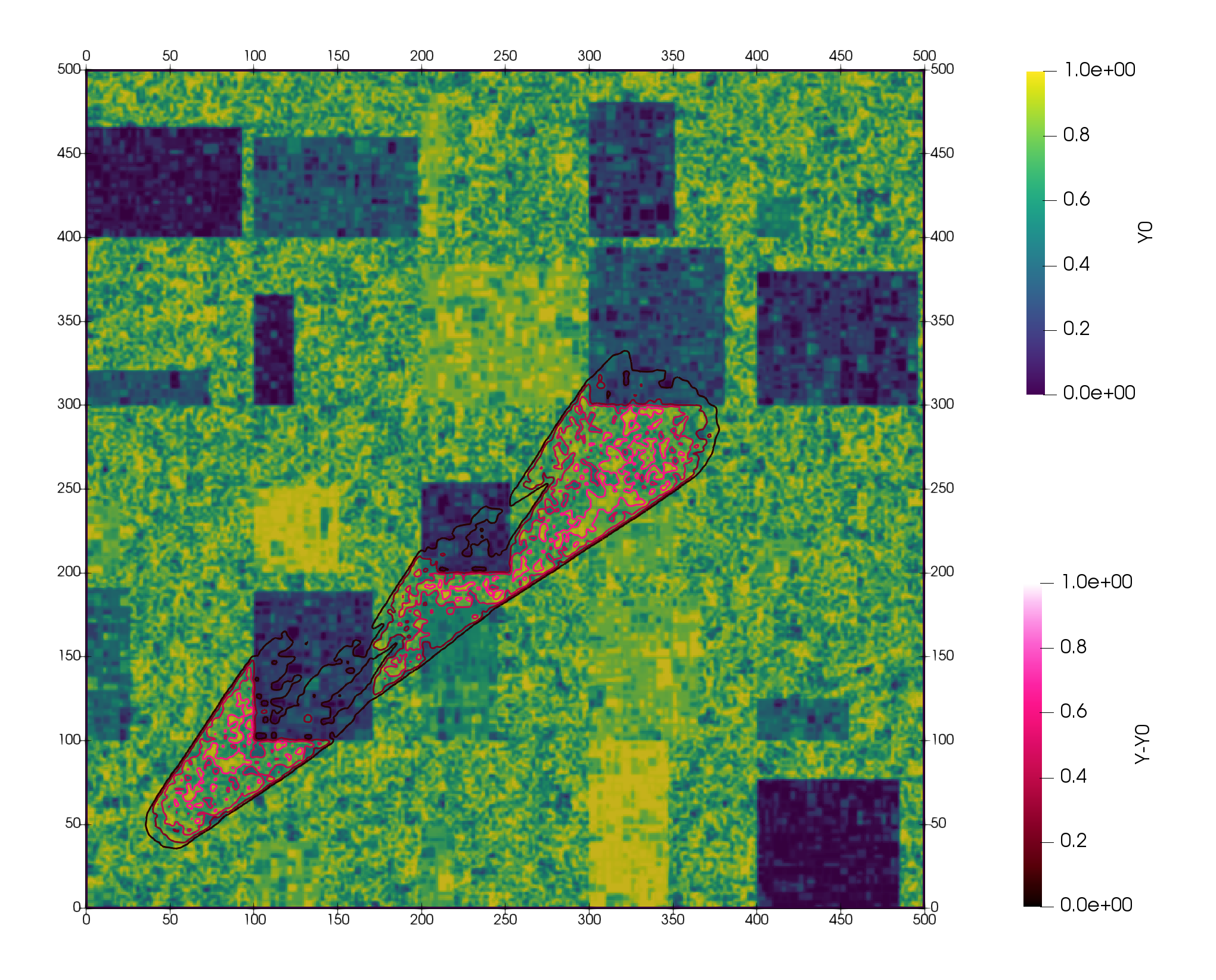}
	 \caption{ Two-dimensional numerical simulation of wildfire propagation through a heterogeneous domain. Top: Temperature contours (every 100 s) and terminal biomass. Bottom: Initial biomass and difference between terminal and initial biomass. Results are shown for 500 s.
  }
	 \label{fig:realistic_solutions}
 \end{figure}

\section{Conclusions} \label{sec:conclusions}

The advection--diffusion--reaction wildfire model is a non-linear system of several coupled mechanisms between the energy of a fire, represented by its temperature in the context of this model, and the biomass. Sub-models including selected mechanisms form a hierarchical model family and shed light on the complex mechanisms one by one. Numerical simulations of the sub-models support the analytical findings and extend the insights by providing non-formal limit case observations. 

In our discussion, we focus on three interlinked properties that we consider important for wildfire management applications: the terminal biomass, the maximum temperature and the propagation speed of the wildfire.  The terminal biomass indicates the fraction of the forest ecosystem that went unharmed by the wildfire and is relevant to assess the damage and the potential of recovery, see \cite{falk_mechanisms_2022, hood_fire_2018}.   The maximum temperature is an indicator of damage, where higher temperatures will destroy more biomass.  The propagation speed of the wildfire influences the amount of burnt area per time.

The spatially lumped ordinary differential equation model in Sec. \ref{sec:ODE} provides insight into the relevance of the reaction parameters on the terminal biomass and the maximum temperature.  In Sec. \ref{sec:combustionfree} and \ref{sec:advectionfree}, diffusion and advection mechanisms are added to obtain spatially explicit sub-models.

A sensitivity analysis of the lumped ordinary differential equation model shows a high sensitivity of the tipping line for the maximum temperature to the activation temperature, the bulk density and the heating value. This tipping line connects the two relevant aspects of terminal biomass and maximum temperature: shifting the tipping line further up means having a lower maximum temperature and a possibly higher remaining biomass after the fire---see Fig.~\ref{fig:sensitivity_tipping}. Meanwhile, the initial biomass has a rather low influence on the terminal biomass---see Fig.~\ref{fig:ODE_example11}, Fig.~\ref{fig:ODEvsPDE_phaseY} and Fig.~\ref{fig:realistic_solutions} for the reaction--diffusion model.  In other words, according to the model, reducing the biomass before the occurrence of a wildfire will not help the forest ecosystem to recover more quickly afterwards. 
However, the analysis of the reaction--diffusion model reveals one advantage of reducing the initial biomass: a lower initial biomass might lead to a reduced wildfire propagation speed---see Fig.~\ref{fig:speeds_bounds} and Fig.~\ref{fig:ODEvsPDE_phaseY}. We found a surprisingly good approximation of the propagation speed by linearisation techniques combined with fixing the biomass $Y$, compare Eq.~\eqref{eq:lin_wavespeed} and Fig.~\ref{fig:ODEvsPDE_phaseY}.  
Diffusion drives the temperature to propagate as a travelling wave with a fixed wave profile.  The propagation speed depends on the reaction parameters, the wind speed, and the diffusion parameter. 
A strong wind affects both the terminal biomass and the wildfire propagation speed. In the direction opposite to the wind, the terminal biomass is smaller and the propagation speed is slower, which might form a natural fire break.  In the  wind direction, we have a faster wildfire propagation but a higher terminal biomass. 
The influence of the wind is therefore ambiguous: on the one hand, the wind enforces the wildfire in the wind direction and leads to a much faster propagation and larger burnt area in a shorter time; however, it also has a potential benefit of reducing damage to vegetation, resulting in higher terminal biomass.  On the other hand, the area opposite to the wind direction remains unaffected, which serves as a small advantage compared to scenarios with no wind. We observed this model property as well in the two-dimensional simulation in Fig.~\ref{fig:realistic_solutions}.

The wind-driven advection mechanism in the present ADR wildfire model is a rather simplified representation of real-world processes.  In our opinion, it also presents the most interesting area of future research activities.  From a mathematical perspective, the shift of the system's behaviour depending on the wind speed is an emergent property that we observe numerically.  A full proof and analysis of travelling waves for the full system is still an open problem. Further investigations into the connection between the approximated models and the full model may concretise the quality of the linearised propagation speed.

With regard to model development, the influence of the wind on the fire dynamics could be simulated with higher physical fidelity by coupling the ADR wildfire model to an atmospheric model. This coupling is computationally challenging but promises a more precise prediction of the propagation direction and speed.  Such a coupled model in combination with real-world forest structures in a two-dimensional setting would provide the option to test wildfire management strategies like firebreaks or planting formations. It is also crucial for understanding plume-dominated fires, which are expected to become more frequent due to climate change.

\section*{Acknowledgments}

The work of C.R. and I.O. is funded as part of the inter-fire project by the Seed Funding Programme of Technische Universit\"at Braunschweig, {\it 2022 Interdisciplinary Collaboration Strengthening Interdisciplinarity --- Expanding Research Collaboration}. The work of A.N. is funded by the Ministerio de Ciencia e Innovación (Agencia Estatal de Investigación) under project-nr. PID2022-141051NA-I00 and by the Fundación Universitaria Antonio Gargallo under project-nr. 2021/B010. The work of A.N. has also been partially funded by Gobierno de Aragón, Spain through Fondo Social Europeo (T32-23R, Feder 2023–2025). The research stay by A.N. at Technische Universit\"at Braunschweig that allowed for this collaboration was funded by the Programa de Estancias Fundación CAI-Ibercaja/Universidad de Zaragoza.


\section*{Author contributions}

C.R.: conceptualisation, investigation, simulation, analysis, writing. A.N.: conceptualisation, investigation, simulation, analysis, writing. I.O.: conceptualisation, investigation, writing.  All authors reviewed the manuscript.

\appendix

\section{Verification of the numerical solver}

The numerical solvers are explained in Sec.~\ref{section:solver}.
Here, we present verification results for sub-models.

\subsection{Verification case 1: spatially lumped model}\label{sec:validation1}

The system in Eq. (\ref{eq:modelODE}) is solved for the initial conditions $T_0=470$ K and $Y_0=1$ using the numerical model in Section \ref{section:solver} with time steps $\Delta t=\left\{ 1,0.1,0.01\right\}$ s (i.e. a implicit RK2 time integrator), as well as with a simple implicit Euler time integrator. The simulation time is $T=150$ s. The computed evolution in time of the temperature and biomass is depicted in Figure  \ref{fig:ODE_validation1}. The trajectory of the solutions is also depicted in the phase space portrait in Fig.~\ref{fig:ODE_map}. From the figures, we observe that the numerical errors in temperature and biomass are negligible for the implicit RK2 method. The solver accurately reproduces the solution in the phase space diagram.

The numerical errors are computed using the $L_{\infty}$ error norm inside the full domain and subdomain $t=[0,75]$ s, and are presented in Tab. \ref{table:ODE_validation1} and \ref{table:ODE_validation2}, respectively. Note that the  subdomain $t=[0,75]$ s does not include the deactivation of the combustion term, thus the solution is regular inside it. The implicit RK2 approach provides remarkably smaller values of errors than the implicit Euler integrator. For the implicit Euler method, the largest error is obtained at the point of  maximum temperature. On the other hand, the implicit RK2 integrator offers a high accuracy in smooth regions (e.g. near the maximum temperature), with the largest error at the point when the combustion term switches off (i.e. when the solution losses its regularity). Thus, we observe large differences between Tab. \ref{table:ODE_validation1} and \ref{table:ODE_validation2}, showing better convergence rates in the latter where the solution is smooth. We evidence that the proposed implicit RK2 integrator in Section \ref{section:solver} possesses good convergence properties and is sufficiently accurate for the applications herein considered.

  \begin{figure}
	 \centering
		 \includegraphics[width=0.99\textwidth]{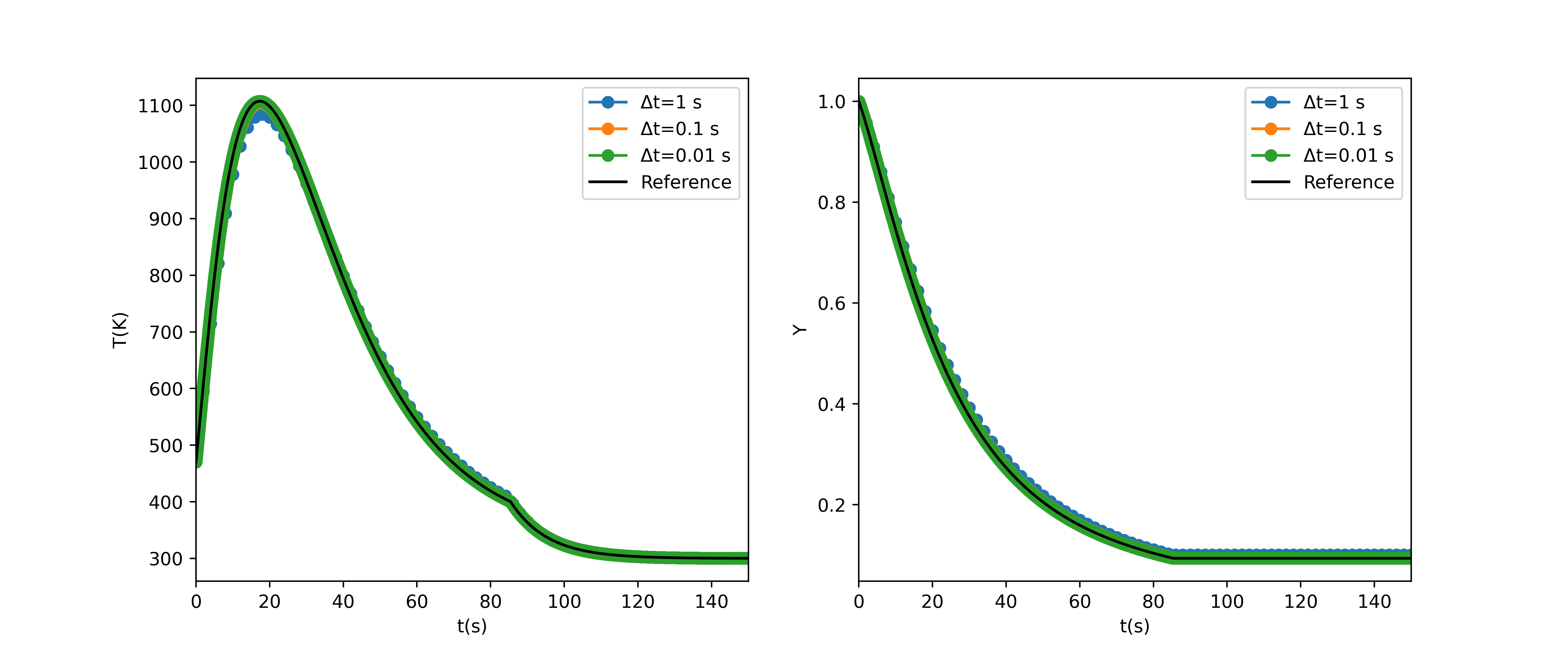}
    \includegraphics[width=0.99\textwidth]{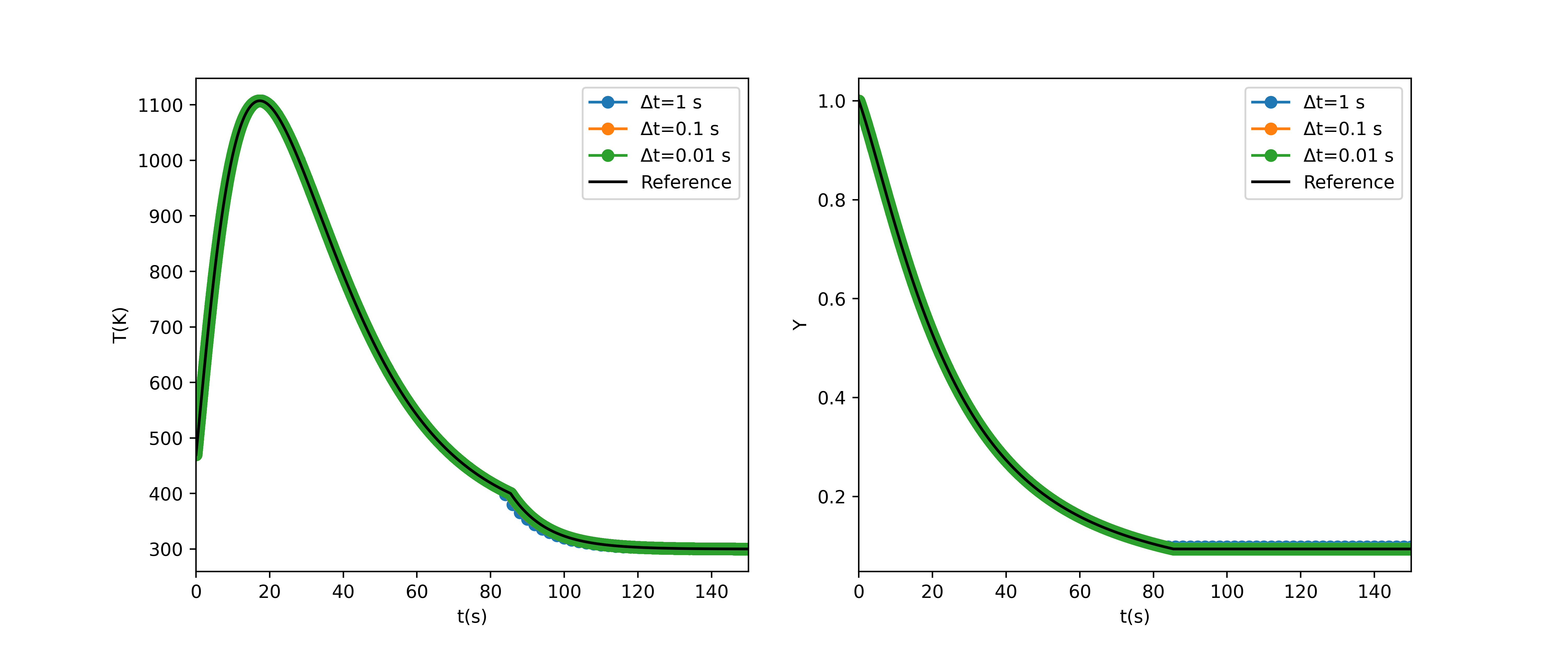}
	 \caption{Numerical solution for $T$ (left) and $Y$ (right) computed by the implicit Euler solver (top) and the implicit RK2 solver (bottom). The reference solution is depicted with black solid line.}
	 \label{fig:ODE_validation1}
 \end{figure}

   \begin{figure}
	 \centering
		 \includegraphics[width=0.44\textwidth]{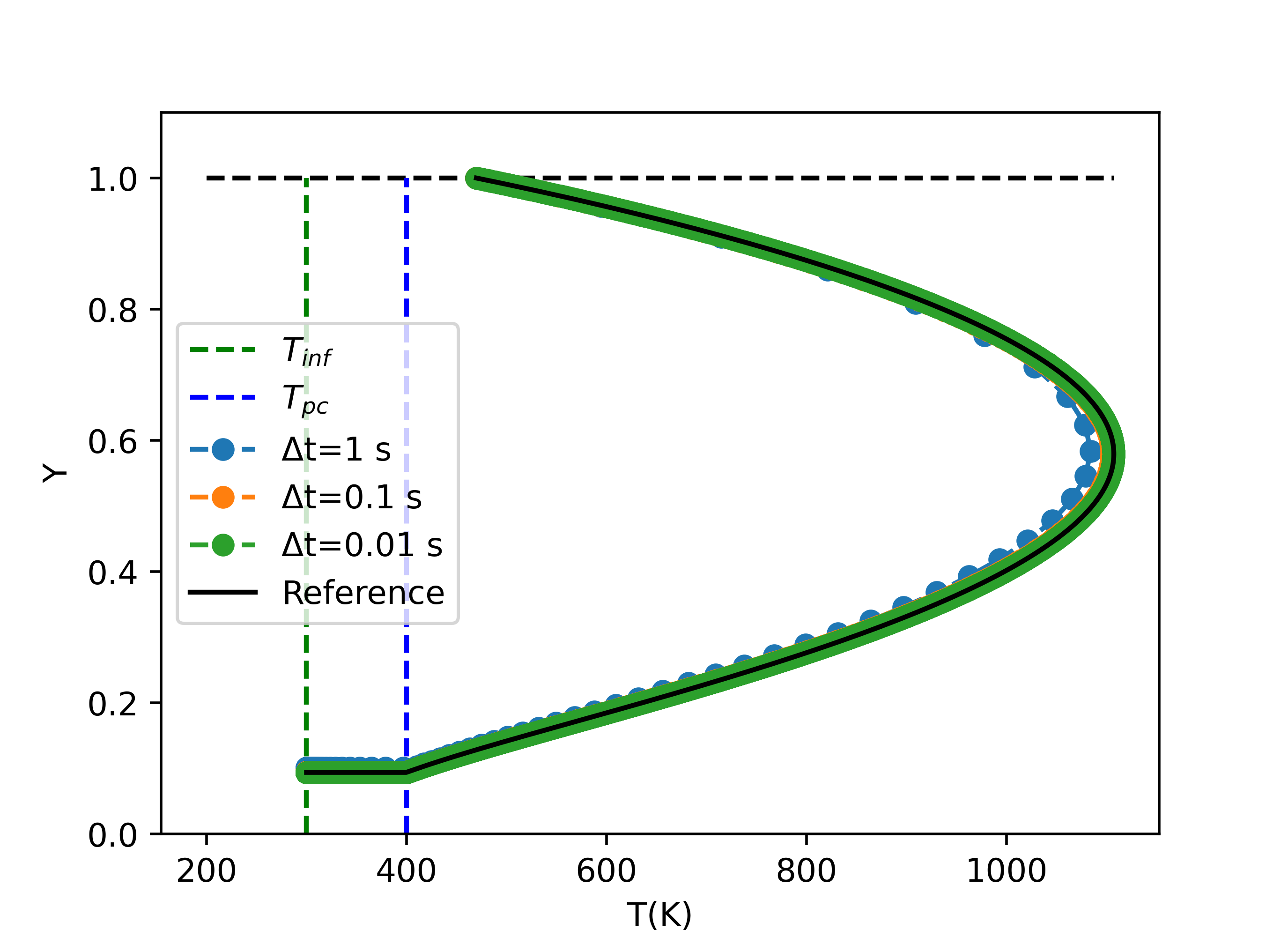}
   \includegraphics[width=0.44\textwidth]{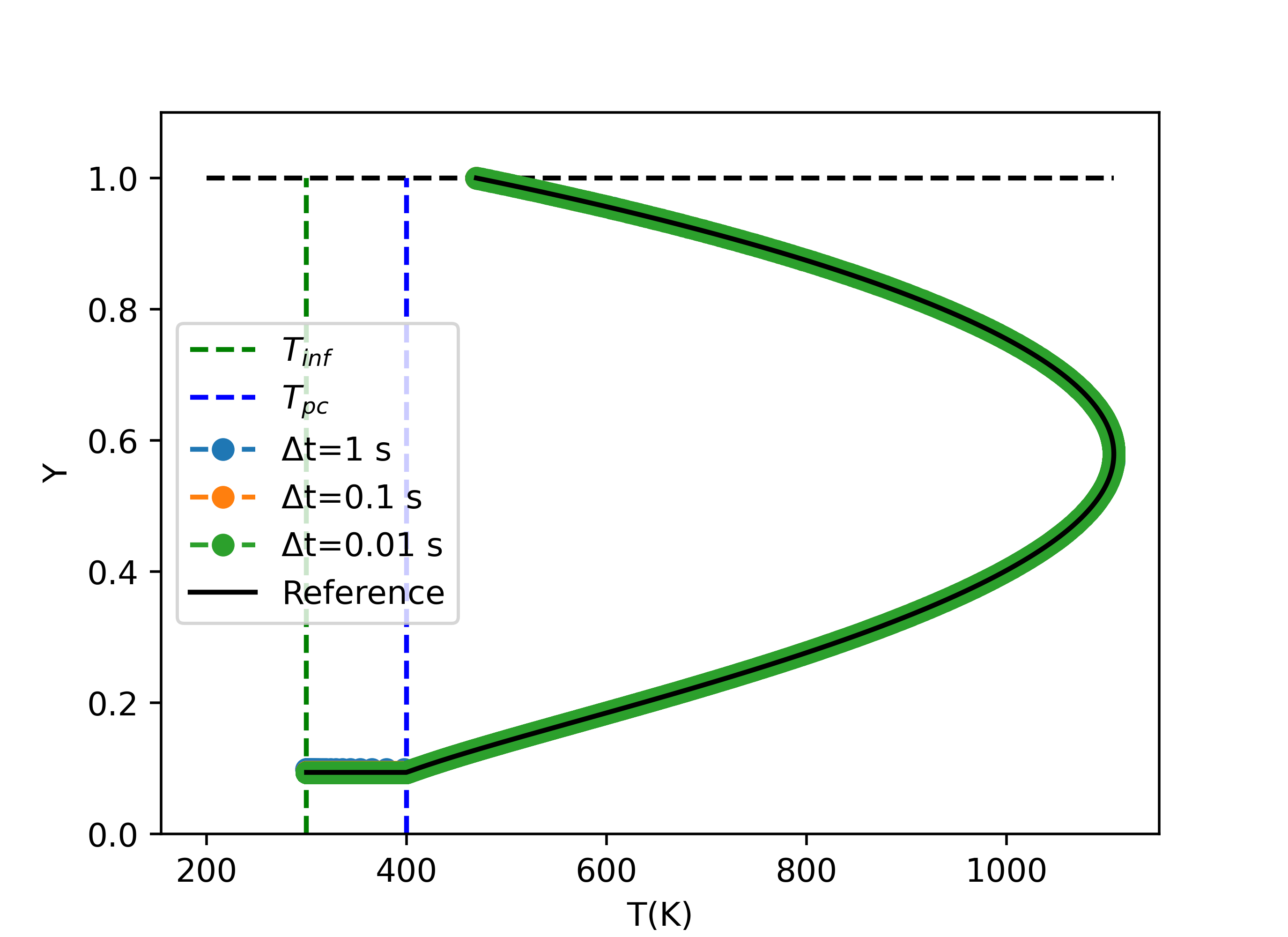}
	 \caption{Computed phase space portrait by the the implicit Euler solver (left) and the implicit RK2 solver (right). The reference solution is depicted with black solid line.}
	 \label{fig:ODE_validation2}
 \end{figure}

 \begin{table}
\centering
\scalebox{0.99}{
\begin{tabular}{ccccc}
& \multicolumn{2}{c}{Implicit Euler} & \multicolumn{2}{c}{Implicit RK2} \\
\hline
$\Delta t$   & $L_{\infty}(T)$  & $L_{\infty}(Y)$ & $L_{\infty}(T)$  & $L_{\infty}(Y)$ \\
\hline
1.0 & 32.71 & 1.79e-02& 14.61  & 4.40e-03  \\
0.1 & 3.38 & 1.79e-03 & 9.55e-01  & 2.44e-04  \\
0.01 & 0.33 & 1.78e-04 & 8.48e-02  & 2.14e-05\\
\end{tabular}}
\caption{Numerical errors for $T$ and $Y$ measured using the $L_{\infty}$ norm in the full domain. }
\label{table:ODE_validation1}
\end{table}

 \begin{table}
\centering
\scalebox{0.99}{
\begin{tabular}{ccccc}
& \multicolumn{2}{c}{Implicit Euler} & \multicolumn{2}{c}{Implicit RK2} \\
\hline
$\Delta t$   & $L_{\infty}(T)$  & $L_{\infty}(Y)$ & $L_{\infty}(T)$  & $L_{\infty}(Y)$ \\
\hline
1.0 & 32.71 & 1.79e-02& 9.74e-01  & 7.81e-04  \\
0.1 & 3.38 & 1.79e-03 & 1.29e-02  & 8.29e-06  \\
0.01 & 0.33 & 1.78e-04 & 2.35e-04  & 9.03e-08 \\
\end{tabular}}
\caption{Numerical errors for $T$ and $Y$ measured using the $L_{\infty}$ norm in the subdomain $t=[0,75]$ s. }
\label{table:ODE_validation2}
\end{table}

\subsection{Verification case 2: advection-diffusion-reaction model}\label{sec:validation2}

In this case, we do not consider combustion ($A=0$) and we will assume $k_t(T)=k$ constant. We solve

  \begin{equation} \label{eq:case1}
     \left\{ \begin{array}{l}
       \frac{\partial T}{\partial t}  + v\frac{\partial T}{\partial x}   =  \alpha \frac{\partial^2 T}{\partial x^2} -  \beta(T-T_{\infty})  \\
 \frac{\partial Y}{\partial  t}=0
     \end{array}\right.
 \end{equation}
 with $\alpha=k/(\rho_0 {C})$ the thermal diffusivity and $ \beta=h/(\rho_0 {C})$, using the numerical methods explained in Section~\ref{section:solver}. We set the following initial condition
   \begin{equation} \label{eq:case_validation2ic}
       T_0(x)=300+100\exp{(-0.001(x-x_0)^2)}
 \end{equation}
 which allows to derive an exact solution for the temperature
  \begin{equation} \label{eq:case1sol}
       T(x,t)= 300+\frac{1}{\sqrt{1+0.004\alpha t}}\exp{\left(- 0.001\frac{(x-x_0-vt)^2}{1+0.004\alpha t}\right)}\exp{(- \beta t)}
 \end{equation}

For the simulation, we choose the following parameters: {$v=5$ m/s, $\rho_0=1$ kg/m$^3$, $C=1$ \hbox{kJ·kg$^{-1}$·K$^{-1}$ },  $k=10$ kW m$^{-1}$ K$^{-1}$, $ h=0.01$ kW·m$^{-3}$·K$^{-1}$, 
and $x_0=250$ m}. The spatial domain is $\Omega=[0,1000]$ m and the final time is set to $t=100$. The solution is computed using $N=100$ computational cells and a Courant--Friedrichs--Lewy number of $\mathsf{CFL}=0.1$. Periodic boundary conditions are set. The numerical solution is computed using a 1-st, 3-rd, 5-th and 7-th order advection scheme.

 Fig.~\ref{fig:validation_case2} shows the computed temperature profile at $t=100$ s. The lower order schemes (1-st and 3-rd) introduce a high numerical diffusion, which would overestimate the conduction and radiation heat transfer processes. On the other hand, the 5-th and 7-th order schemes provide accurate results for this grid. The results evidence that the advective terms must be computed with high order of accuracy to avoid unphysical diffusion.

\begin{figure}
	 \centering
		 \includegraphics[width=0.99\textwidth]{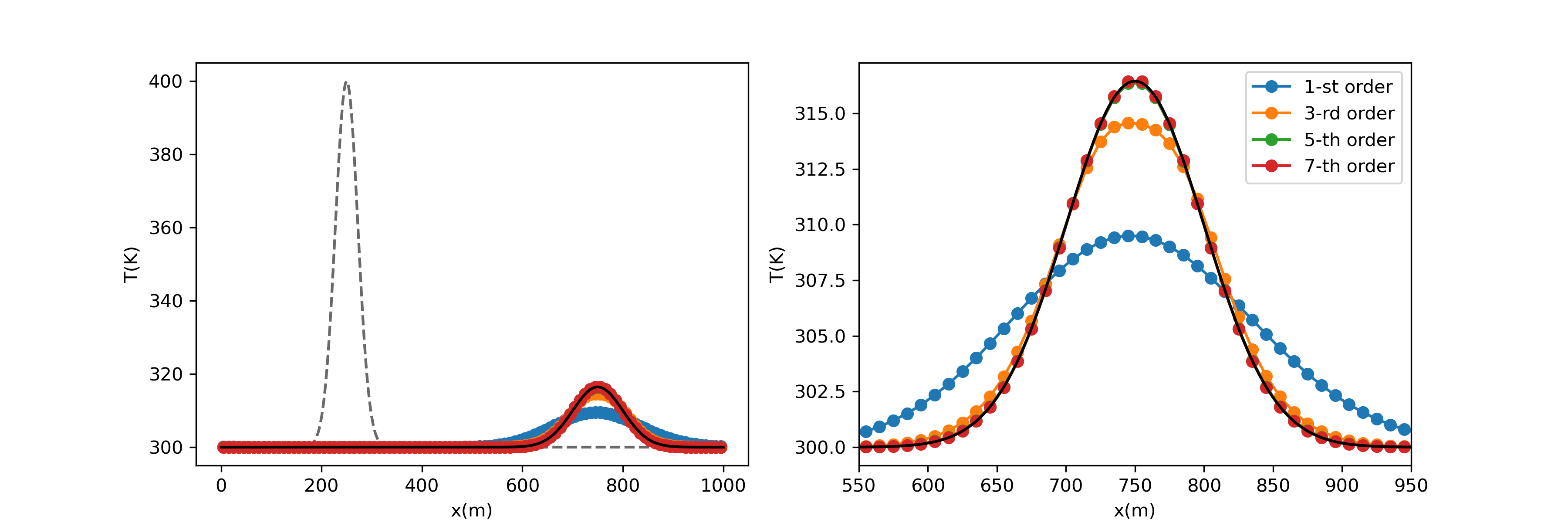}
	 \caption{Numerical solution computed by a 1-st, 3-rd, 5-th and 7-th order, compared with the exact solution (black solid line). Initial condition depicted with dashed line.}
	 \label{fig:validation_case2}
 \end{figure}

\subsection{Dependency of the travelling wave speed on the discretisation}

The simulation of the reaction--diffusion equation in Eq.~\eqref{eq:model_pure_rd} shows travelling wave solutions.
The wave speed there depends on the number of discretisation points used in the spatial discretisation, see Fig.~\ref{fig:speed_dx}.

 \begin{figure}
	 \centering
 	 \includegraphics[width=0.6\textwidth]{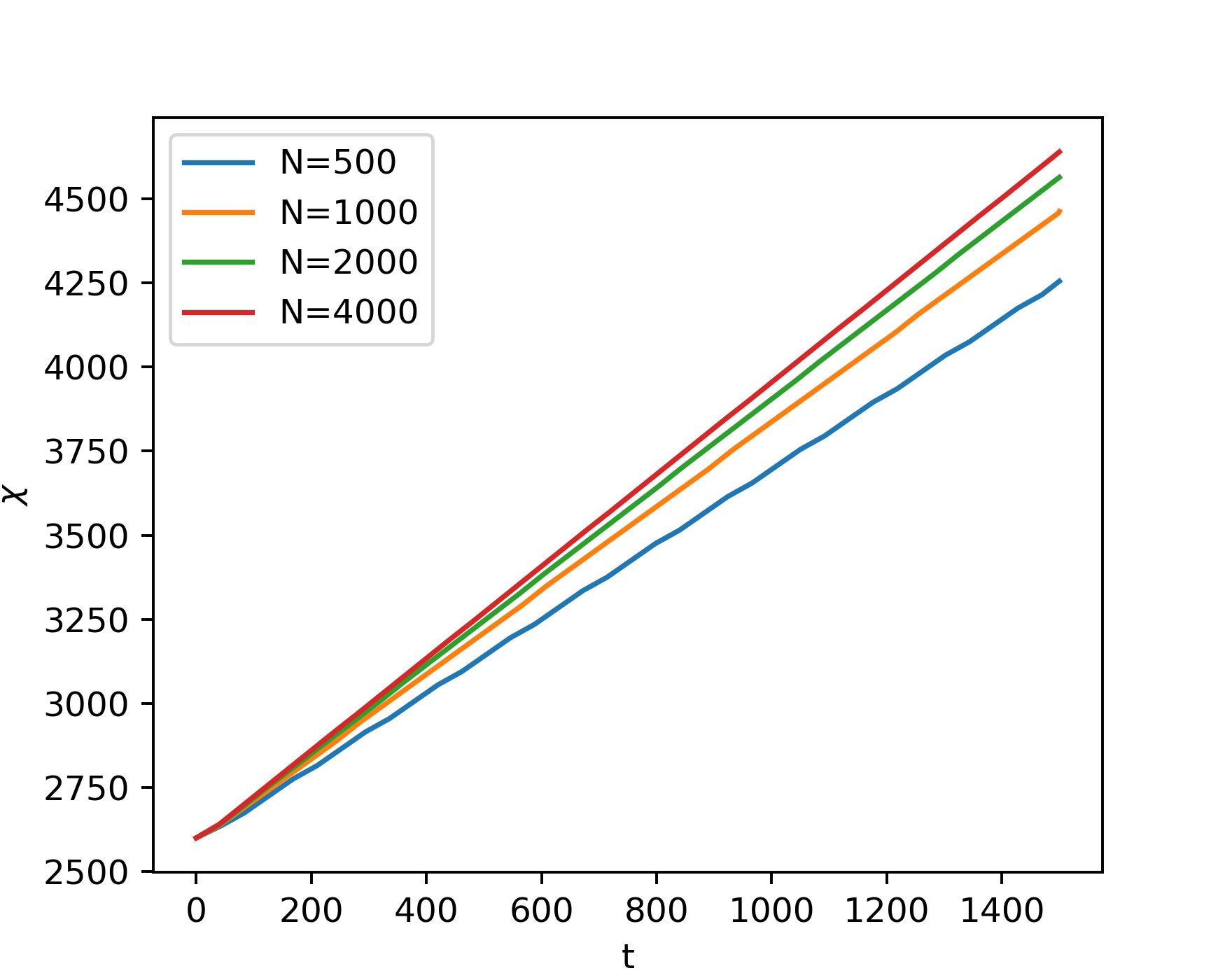}
   	 \caption{Position of the wave front depending on the number of space discretisation points. }
	 \label{fig:speed_dx}
 \end{figure}

 In the simulations the parameter $h$ is given by $h=1$ {kW·m$^{-3}$·K$^{-1}$}.
 The wave speed increases with a larger number of spatial discretisation points with a tendency towards a threshold wave speed.
 This phenomenon is known for the discretisation of wave-like solutions, compare \cite[p. 169f.]{vreugdenhil_numerical_1994}.

\subsection{Configuration of the simulations}

Details for the simulations are given in Table~\ref{table:summary-simulations}, including the domain size ($\Omega$), the final time ($T$), the initial (IC) and boundary conditions (BC), the number of cells ($N$), the $\mathsf{CFL}$ number to compute the time step and the order of accuracy for advection. The section where each of the simulations is located is also included in the table, as well as the case identifier, if any.

Regarding the initial condition for $Y$ in the case in Section 4.1 (*), it is defined by overlapping random distributions of $Y$ with different length scales. We consider four  layers composed of  quadrilateral patches of maximum size of 2, 5, 10 and 100 m. The size of the quadrilateral patches is randomly selected, as well as the magnitude of $Y$ inside each patch. The final value for $Y$ taken as IC is given by the sum of the values computed in each map $j$, denoted by $Y_j$, i.e. $Y(x,y)=Y_1(x,y)+Y_2(x,y)+Y_3(x,y)+Y_4(x,y)$. Note that the final value for $Y$ must be within $[0,1]$ and thus adequate bounds when generating the random numbers must be considered.

 \begin{table}
\centering
\scalebox{0.99}{
\begin{tabular}{lcccccccc}
\hline
Section   & Case & $\Omega$ & $T$ & IC &BC & $N$ & $\mathsf{CFL}$ & Order\\
\hline
3.2.1 &  &  $[0,1000]\times[0,1000]$& 1000 & Eq. \eqref{eq:ic1} &Periodic & $100^2$ & 0.1 & 7 \\
3.2.2 &  &  $[0,1000]\times[0,1000]$& 1000 & Eq. \eqref{eq:ic1} & Periodic & $100^2$ & 0.1 & 7 \\
3.3.1 & A &  $[0,500]$& 800 &  Eq. \eqref{eq:ic}& Transmissive & 2000 & 0.1 & 7 \\
3.3.1 & B &  $[0,500]$& 800 &  Eq. \eqref{eq:ic}& Transmissive & 2000 & 0.1 & 7 \\
4.1 & C &  $[0,500]$& 800 &  Eq. \eqref{eq:ic}& Transmissive & 2000 & 0.1 & 7 \\
4.1 &  &  $[0,500]\times[0,500]$& 500 &  Eq. \eqref{eq:icreal}(*) & Transmissive & $500^2$ & 0.1 & 7 
\end{tabular}}
\caption{Summary and configuration details of the test cases presented. }
\label{table:summary-simulations}
\end{table}

\printbibliography

\end{document}